\documentclass[a4paper,11pt]{amsart}
\usepackage[utf8x]{inputenc}
\usepackage[margin=1.25in]{geometry}
\usepackage[T1]{fontenc}
\usepackage{ae,aecompl}
\usepackage{tikz-cd}
\usepackage{mathrsfs} 
\usepackage{verbatim,amsmath,amsthm,amsfonts,amssymb,latexsym,graphicx,mathtools,extpfeil,color,mathabx}
\usepackage{tikz}
\usetikzlibrary{cd}
\usepackage[all]{xy}
\usepackage{graphicx}
\usepackage{booktabs}
\usepackage{caption}
\usepackage{tabularx}
\usepackage{subcaption}
\usepackage{float}
\usepackage{pdflscape}
\usepackage{cancel}
\usepackage{slashed}
\DeclareMathAlphabet{\mathpzc}{OT1}{pzc}{m}{it}

\usepackage[colorlinks,pagebackref,hypertexnames=false]{hyperref} \usepackage[alphabetic,backrefs,msc-links]{amsrefs}

\usepackage{aliascnt}
\numberwithin{equation}{section}

\renewcommand{\epsilon}{\varepsilon}

\def\({\mathopen{}\left(}
\def\){\right)\mathclose{}}
\def\<{\mathopen{}\left<}
\def\>{\right>\mathclose{}}

\usepackage{multicol, color}

\definecolor{gold}{rgb}{0.85,.66,0}
\definecolor{cherry}{rgb}{0.9,.1,.2}
\definecolor{burgundy}{rgb}{0.8,.2,.2}
\definecolor{orangered}{rgb}{0.85,.3,0}
\definecolor{orange}{rgb}{0.85,.4,0}
\definecolor{olive}{rgb}{.45,.4,0}
\definecolor{lime}{rgb}{.6,.9,0}
\definecolor{green}{rgb}{.2,.7,0}
\definecolor{grey}{rgb}{.4,.4,.2}
\definecolor{brown}{rgb}{.4,.3,.1}

\def\makeautorefname#1#2{\AtBeginDocument{\expandafter\def\csname#1autorefname\endcsname{#2}}}

\newcommand{\mynewtheorem}[2]{
  \newaliascnt{#1}{equation}          
  \newtheorem{#1}[#1]{#2}
  \aliascntresetthe{#1}
  \makeautorefname{#1}{#2}
}
\mynewtheorem{theorem}{Theorem}
\mynewtheorem{prop}{Proposition}
\mynewtheorem{cor}{Corollary}
\mynewtheorem{construction}{Construction}
\mynewtheorem{lemma}{Lemma}
\mynewtheorem{conjecture}{Conjecture}
\mynewtheorem{shitu}{Situation}
\mynewtheorem{conds}{Condition}

\numberwithin{substep}{step}
\makeautorefname{step}{Step}
\makeautorefname{substep}{Step}

\numberwithin{subcase}{case}
\makeautorefname{case}{Case}
\makeautorefname{subcase}{case}

\theoremstyle{remark}
\mynewtheorem{remark}{Remark}

\theoremstyle{definition}
\mynewtheorem{definition}{Definition}
\mynewtheorem{example}{Example}
\mynewtheorem{exercise}{Exercise}
\mynewtheorem{convention}{Convention}
\newtheorem*{convention*}{Convention}
\newtheorem*{conventions*}{Conventions}
\mynewtheorem{question}{Question}
        
\makeautorefname{chapter}{Chapter}
\makeautorefname{section}{Section}
\makeautorefname{subsection}{Section}
\makeautorefname{subsubsection}{Section}

\theoremstyle{introthm}
\newtheorem{introthm}{Theorem}

\newcommand\bbZ{{\mathbb Z}}
\newcommand\bbR{{\mathbb R}}
\newcommand\bbC{{\mathbb C}}

\title[Lagrangian Floer theory in divisor complement]{Monotone Lagrangian Floer theory in smooth divisor complements: I}
\author{Aliakbar Daemi, Kenji Fukaya}

\date{}

\begin{document}
\address{Department of Mathematics and Statistics, Washington University in St. Louis, MO 63130}\email{adaemi@wustl.edu}
\address{Simons Center for Geometry and Physics, State University of New York, Stony Brook, NY 11794-3636, USA} \email{kfukaya@scgp.stonytbrook.edu}

\maketitle

\begin{abstract}
In this paper, Floer homology for Lagrangian submanifolds in an open symplectic manifold given as the complement of a smooth divisor is discussed. The main new feature of this construction is that we do not make any assumption on positivity or negativity of the divisor. To achieve this goal, we use a compactification of the moduli space of pseudo-holomorphic discs into the divisor complement satisfying Lagrangian boundary condition that is stronger than the stable map compactification and is inspired by the compactifications that are used in relative Gromov--Witten theory. This is the first of a series of three papers, this compactification is introduced and some of its fundamental properties as a topological space, essential for the definition of Lagrangian Floer homology, are established. 
\end{abstract}

\vspace{0.25cm}
{
  \hypersetup{linkcolor=black}
  \tableofcontents
}

\section{Introduction}

Lagrangian Floer theory plays a central role in recent developments in symplectic topology; it has been used to study symplectic rigidity of Lagrangain submanifolds \cite{ALP:lag} and it lies at the heart of the homological mirror symmetry program \cite{Ko:HMS}. This theory associates a homology group to a pair of Lagrangians in a symplectic manifold. In order to define this invariant, one needs to make some restrictive assumptions on the Lagrangians and the underlying symplectic manifold. Hence, there are various flavors of Lagrangian Floer homology in the literature \cite{Flo88,oh,fooobook,fooobook2,AJ:immerse}. 

In this paper, we study Floer homology for Lagrangians in open symplectic manifolds obtained by removing a divisor from a closed symplectic manifold. The novelty of our construction is that we do not make any convexity assumption about the ends of such open manifolds. 
The main application that we have in mind is in gauge theory. Motivated by \cite{mw}, we are planning to use the construction of the present article and its subsequents to define symplectic instanton Floer homology for arbitrary ${\rm U}(N)$-bundles over 3-manifolds. As it is discussed in the last section of \cite{part3:FH}, it is plausible that the construction of this paper can be useful in various other contexts.

\subsection{Statement of Results}

Let $(X,\omega)$ be a compact symplectic  manifold and $\mathcal D$ be a codimension two symplectic submanifold of $X$. We call any such submanifold $\mathcal D$ of $X$ a {\it smooth divisor} in $(X,\omega)$.

\begin{definition}\label{defn1111}
	Let $L$ be a compact Lagrangian submanifold of $X\backslash \mathcal D$.
	We say $L$ is {\it monotone in $X\setminus \mathcal D$}, if there exists $c>0$ such that the following holds 
	for any $\beta \in {\rm Im}(\pi_2(X\setminus \mathcal D,L) \to \pi_2(X,L))$:
	\[
	   \omega(\beta)=c \mu(\beta).
	\]
	Here $\mu : H_2(X,L;\bbZ) \to \bbZ$ is the Maslov index associated to the 
	Lagrangian submanifold $L$. (See, for example, \cite[Subsection 2.1.1]{fooo:tech2-1}.) 
	The minimal {\it Maslov number} of a Lagrangian $L$ in $X \setminus \mathcal D$ is defined to be:
	\[
	  \inf \{\mu(\beta) \mid \beta \in {\rm Im}(\pi_2(X\setminus \mathcal D,L) \to \pi_2(X,L)),\omega(\beta) >0 \}.
	\]
\end{definition}
\begin{definition}
	Let $L_0, L_1$ be compact subspaces of $X\setminus \mathcal D$.
	We say $L_0$ is {\it Hamiltonian isotopic to $L_1$ in $X\setminus \mathcal D$} if there exists a 
	compactly supported time dependent Hamiltonian 
	$H : (X \setminus \mathcal D) \times [0,1] \to \bbR$ such that the Hamiltonian 
	diffeomorphism $\psi_H : X \setminus \mathcal D \to X \setminus \mathcal D$
	maps $L_0$ to $L_1$.
	Here $\psi_H$ is defined as follows.
	Let $H_t(x) = H(x,t)$ and $X_{H_t}$ be the Hamiltonian vector field associated to $H_t$. 
	We define $\psi^H_t$ by 
	\[
	  \psi^H_0(x) = x, \qquad \frac{d}{dt}\psi^H_t = X_{H_t} \circ \psi^H_t.
	\]
	Then $\psi_H := \psi^H_1$. We say that $\psi_H$ is the {\it Hamiltonian diffeomorphism associated to the 
	(non-autonomous) Hamiltonian $H$}.
\end{definition}
The main result of this series of papers is the following.

\begin{introthm}\label{mainthm}
	Let $L_0,L_1\subset X\setminus \mathcal D$ be compact, oriented and spin Lagrangian submanifolds 
	such that they are monotone in $X\setminus \mathcal D$.
	Suppose one of the following conditions holds:
	\begin{enumerate}
		\item[(a)] The minimal Maslov numbers of $L_0$ and of $L_1$ are both strictly greater than 2.
		\item[(b)] $L_1$  is Hamiltonian isotopic to $L_0$. 
	\end{enumerate}
	Then there is a Floer homology group $HF(L_1,L_0;X\setminus \mathcal D)$,
	which is a vector space over a Novikov ring depending only on the Hamiltonian isotopy classes of $L_0$ and $L_1$.
	If $L_0$ is transversal to $L_1$ then we have:
	\[
	  \dim(HF(L_1,L_0;X\setminus D)) \le \# (L_0 \cap L_1).
	\]
	and if $L_0 = L_1 =L$, then there exists a spectral sequence 
	whose $E^2$ page is the singular homology group of $L$ and which converges to 
	$HF(L,L;X\setminus \mathcal D)$.
\end{introthm}
See \cite{part3:FH} for a more detailed and slightly stronger version of this theorem.

In \cite{Flo88}, Floer proved the analogue of Theorem \ref{mainthm} in the case that $\pi_2(X,L_i)=0$ and the coefficient ring is $\bbZ/2\bbZ$. In this case, the analogue of the spectral sequence for the Floer homology of the pair $(L,L)$ collapses in the second page and the Floer homology is isomorphic to singular homology of $L$ \cite{Flo88, Flo89}. Oh generalizes Floer's construction to the case that $L_0$ and $L_1$ are monotone in $X$ \cite{oh}. He also constructed a spectral sequence from the homology of $L$ to the Lagrangian Floer homology of the pair $(L,L)$ in \cite{oh-spectral}. (See also \cite[Chapter 2]{fooobook} and \cite{Cornea-Biran}.). 

The main new feature of Theorem \ref{mainthm} is that we assume monotonicity of $L_0$ and $L_1$ only in $X\setminus \mathcal D$. Roughly speaking, the Floer homology $HF(L_1,L_0;X\setminus \mathcal D)$ is defined using only holomorphic disks which `do not intersect' $\mathcal D$. Therefore, Theorem \ref{mainthm} can be regarded as an extension of Oh's monotone Floer homology \cite{oh} to open manifolds where the geometry at infinity is controlled by a smooth divisor. As we mentioned earlier, this extension of Floer homology is partly motivated by monotone Lagrangian submanifolds in divisor complements which are constructed by gauge theory \cite{mw,DF}. 

There are various other special cases of Theorem \ref{mainthm} which already appear in the literature. In the case that $X\setminus \mathcal D$ is convex at infinity,  the methods of \cite{fooobook,fooobook2}  can be used to define a Floer homology group $HF(L_1,L_0;X\setminus \mathcal D)$ satisfying the properties mentioned in Theorem \ref{mainthm}. In particular, this setup can be applied to the case that each component of $\mathcal D$ is a positive multiple of the Poincar\'e dual of $[\omega]$, the cohomology class of the symplectic form. Starting with the remarkable work of Seidel \cite{seidel:ICM,seidel:quartic}, such Floer homology groups have been used to study Fukaya category of $X$ and to verify homological mirror symmetry for some special examples.

As in any other versions of Lagrangian Floer homology, the main geometrical input in the definition of the Floer homology of Theorem \ref{mainthm} is the moduli space of holomorphic maps from the standard disc to $X$ (equipped with appropriate almost complex structures), which satisfy Lagrangian boundary conditions. There is a standard compactification of this moduli space called the {\it stable map compactifiaction}, which plays an  essential role in the definition of previous versions of Lagrangian Floer homology. This compactification, however, is not suitable for our purposes. 

A key observation for this series of papers is that this issue can be resolved by a different and stronger compatification of the above moduli space of holomorphic discs, which we we call {\it RGW compactification}. The definition of this compatctification is inspired by the theory of {\it Relative Gromov--Witten} invariants, where one uses the moduli spaces of holomorphic maps from a closed surface to $X$, which intersect a divisor $\mathcal D$ in a prescribed way, to construct numerical invariants of $(X,\mathcal D)$. There are also formal similarities between the RGW compactification and the compactification of the moduli spaces used in symplectic field theory.

In Section \ref{main-idea} of the present paper, we explain in more detail why stable map compactifiaction comes short for our purposes, and how our work is related to several approaches to Relative Gromov--Witten theory. The definition of RGW compactification as a set is given in Section \ref{Sec:RGW-Compactification}, and in Section \ref{sec:topology} after defining the topology of this set we prove the compactness of this topological space. In our second paper of this series \cite{part2:kura}, we study the analytical features of the RGW compactification. In particular, we show that this space admits a Kuranishi structure with boundary and corners. We believe that the analytical methods of \cite{part2:kura} provide an approach to address some of the foundational questions for relative Gromov-Witten invariants for a pair of a symplectic manifold and a smooth divisor. In the third paper \cite{part3:FH}, we use the results of the present paper and \cite{part2:kura} to prove Theorem \ref{mainthm}.

{\it Acknowledgements.}
We thank Paul Seidel, Mark Gross, Mohammad Tehrani and Aleksey Zinger for helpful conversations. We also thank the anonymous referee for giving several helpful comments on an earlier version of this series of papers. We are grateful to the Simons Center for Geometry and Physics for providing a stimulating environment for our collaboration on this project.


\section{Main idea of the construction}\label{main-idea}
Suppose $(X,\omega)$ is a closed symplectic manifold. Any flavor of Lagrangian Floer homology of two Lagrangians $L_0$ and $L_1$ in a symplectic manifold $(X,\omega)$ is defined using holomorphic strips into $X$ satisfying Lagrangian boundary conditions where the holomorphic curve equation is defined with respect to a tame almost complex structure $J$ on $X$. To be more specific, let $p,q \in L_0 \cap L_1$, and consider maps 
\[
  u : \bbR\times [0,1] \to X
\]
that are $J$-holomorphic and 
\begin{equation}\label{asymptotic}
	u(\tau,0) \in L_0,\qquad u(\tau,1) \in L_1,\qquad \lim_{\tau\to -\infty} u(\tau,t) = p,
	\qquad \lim_{\tau\to +\infty} u(\tau,t) = q.
\end{equation}
Any such map represents a homology class in the following sense.
\begin{definition}\label{defn2121}
        We say (not necessarily holomorphic) maps $u,\,u':\bbR\times [0,1] \to X$ satisfying \eqref{asymptotic} are 
        {\it homologous} to each other if there exists $v : \Sigma \to X$ with the following properties.
        \begin{enumerate}
	        \item $\Sigma$ is an oriented 3 dimensional manifold with corners. $\partial \Sigma$ is 
	        identified with $(\bbR \times [0,1] \times \{0,1\})\cup S_0 \cup S_1$,
        		where $\partial S_0 \cong \bbR \times \{0\}\times \{0,1\}$ and $\partial S_1 \cong \bbR \times \{1\}\times \{0,1\}$.
	        \item $v : \Sigma \to X$ is a continuous map.
        		\item $v(\tau,t,0) = u(\tau,t)$ and $v(\tau,t,1) = u'(\tau,t)$.
        \item $v(S_0) \subset L_0$, $v(S_1) \subset L_1$.
        \item Complement of a compact subspace of $\Sigma$ is identified with 
        $((-\infty,-C] \times [0,1]^2) \cup ([C,\infty) \times [0,1]^2)$
        and 
        $$
        \lim_{\tau \to -\infty} v(\tau,x) = p,
        \qquad
        \lim_{\tau \to +\infty} v(\tau,x) = q.
        $$
        \end{enumerate}
	The set of such homology classes is denoted  by $\Pi_{2}(X;L_1,L_0;p,q)$.
\end{definition}

For $\beta \in \Pi_{2}(X;L_1,L_0;p,q)$, we define $\mathcal M^{\rm reg}(L_1,L_0;p,q;\beta)$ to be the set of all equivalence classes of $J$-holomorphic maps $u :\bbR\times [0,1] \to X$ satisfying \eqref{asymptotic} and representing $\beta$ with respect to the equivalence relation given by translation along the $\bbR$ factor of $\bbR\times [0,1]$. Namely, $u \sim u'$, if there exists $\tau_0$ such that $u'(\tau,t) = u(\tau+\tau_0,t)$.

The Lagrangian Floer homology group $HF(L_1,L_0;X)$ is the homology of a chain complex $(CF(L_1,L_0),\partial)$ where $CF(L_1,L_0)$ is the vector space generated by the elements of $L_0 \cap L_1$ and $\partial:CF(L_1,L_0)\to CF(L_1,L_0)$ is defined as
\begin{equation} \label{differential}
  \partial ([p]) = \sum_{q,\beta} \#\mathcal M^{\rm reg}(L_1,L_0;p,q;\beta) [q].
\end{equation}
Here the sum on the right hand side is taken over all $(q,\beta)$ such that the {\it virtual dimension} of $\mathcal M^{\rm reg}(L_1,L_0;p,q;\beta)$ is $0$. 

As the next step, moduli spaces $\mathcal M^{\rm reg}(L_1,L_0;p,q;\beta)$ of virtual dimension $1$ are used to show that $\partial$ is a differential. In fact, one first compactifies these moduli spaces and then characterizes the coefficient of $[q]$ in  $\partial\circ \partial([p])$ in terms of the boundary points of the compactified moduli spaces. The foundational issue is that one should also expect other contributions to the boundary of the moduli spaces in correspondence to the {\it disc bubbles}. 

To spell this out in more detail, we need to consider other types of moduli spaces of holomorphic curves. First let $\mathcal M^{\rm reg}_{0,1}(L_1,L_0;p,q;\beta)$ and $\mathcal M^{\rm reg}_{1,0}(L_1,L_0;p,q;\beta)$ be the $J$-holomorphic maps $u :\bbR\times [0,1] \to X$ satisfying \eqref{asymptotic} and representing $\beta$. Thus these spaces agree with each other and their quotient with respect to the translation action is $\mathcal M^{\rm reg}(L_1,L_0;p,q;\beta)$. We define the evaluation maps ${\rm ev}_{0,1} : \mathcal M^{\rm reg}_{0,1}(L_1,L_0;p,q;\beta) \to L_0$ and ${\rm ev}_{1,0} : \mathcal M^{\rm reg}_{1,0}(L_1,L_0;p,q;\beta)\to L_1$ by
\[
  {\rm ev}_{0,1}(u) = u(0,0),\qquad{\rm ev}_{1,0}(u) = u(0,1).
\]
The set $\mathcal M^{\rm reg}_{0,1}(L_1,L_0;p,q;\beta)$ can be regarded as the moduli space of marked holomorphic strips $(u,z_0)$ modulo the translation action where $z_0$ is a marked point on  $\bbR \times \{0\}$. Any such marked strip has a unique representative where $z_0 = (0,0)$. Similarly, $\mathcal M^{\rm reg}_{1,0}(L_1,L_0;p,q;\beta)$ can be regarded as a moduli space of marked strips $(u,z_1)$ modulo the translation action where $z_1$ is a marked point on $\bbR \times \{1\}$. See Definition \ref{defn33strip} for the generalization where we allow more marked points on $\bbR \times \{0\}$ and $\bbR \times \{1\}$.

For any $\alpha$ in the image $\Pi_2(X,L;\bbZ)$ of the Hurewicz homomorphism $\pi_2(X,L) \to H_2(X,L;\bbZ)$, we define $\mathcal M^{\rm reg}_1(L;\alpha)$ to be the set of all equivalence classes of $J$-holomorphic maps maps $u : (D^2,\partial D^2) \to (X,L)$ in the homotopy class $\alpha$, where $u \sim u'$ if there exists a bi-holomorphic map $v : D^2 \to D^2$ such that $u \circ v = u'$ and $v(1) = 1$. We also define the evaluation map ${\rm ev} : \mathcal M^{\rm reg}_1(L;\alpha)\to L$ by ${\rm ev}(u) = u(1)$. When the choice of $L$ is clear from the context, we write $\mathcal M^{\rm reg}_1(\alpha)$ for $\mathcal M^{\rm reg}_1(L;\alpha)$.

There is a topology, called the stable map topology, on the spaces
\begin{equation}\label{4-moduli-pre-cpct}
  \mathcal M^{\rm reg}(L_1,L_0;p,q;\beta),\hspace{.3cm}
  \mathcal M^{\rm reg}_{0,1}(L_1,L_0;p,q;\beta),\hspace{.3cm}
  \mathcal M^{\rm reg}_{1,0}(L_1,L_0;p,q;\beta),\hspace{.3cm}
  \mathcal M^{\rm reg}_1(L;\alpha).
\end{equation}
See Subsection \ref{stablemaptopology} for a review of the stable map topology. There are also compactifications of these topological spaces denoted by
\begin{equation}\label{4-moduli}
  \mathcal M(L_1,L_0;p,q;\beta),\hspace{.5cm}
  \mathcal M_{0,1}(L_1,L_0;p,q;\beta),\hspace{.5cm}
  \mathcal M_{1,0}(L_1,L_0;p,q;\beta),\hspace{.5cm}
 \mathcal M_1(L;\alpha).
\end{equation}
These compactifications are metrizable. The evaluation maps ${\rm ev}$ naturally extends to maps to the compactified spaces. We use the same notation to denote these extensions.

In the compactification $\mathcal M(L_1,L_0;p,q;\beta)$ of $ \mathcal M^{\rm reg}(L_1,L_0;p,q;\beta)$ we add $J$-holomorphic maps where disc and sphere bubbles and broken strips are allowed. In particular, included in this compactification we can identify three types of {\it boundary} points. The first type corresponds to the products
\begin{equation}\label{type-I}
  \mathcal M(L_1,L_0;p,r;\beta_1)\times \mathcal M(L_1,L_0;r,q;\beta_2),
\end{equation}
where $r \in L_0\cap L_1$, $\beta_1 \in \Pi_2(L_1,L_0;p,r)$ and $\beta_2 \in \Pi_2(L_1,L_0;r,q)$ such that $\beta_1 \# \beta_2 = \beta$. Here $\#$ denotes the concatenation maps
\[
  \Pi_{2}(X;L_1,L_0;p,r)\times \Pi_{2}(X;L_1,L_0;r,q)\to \Pi_{2}(X;L_1,L_0;p,q).
\]  
The second and the third types correspond to 
\begin{equation}\label{type-II-III}
  \mathcal M_{0,1}(L_1,L_0;p,q;\beta_0)\times_{L_0} \mathcal M_1(L_0;\alpha_0),\hspace{1cm}
  \mathcal M_{1,0}(L_1,L_0;p,q;\beta_1)\times_{L_1} \mathcal M_1(L_1;\alpha_1),
\end{equation}
where $\beta_0, \beta_1 \in \Pi_2(L_1,L_0;p,q)$, $\alpha_0 \in \Pi_2(X,L_0;\bbZ)$ and $\alpha_1 \in \Pi_2(X,L_0;\bbZ)$ satisfy $\beta_0 \# \alpha_0 = \beta$ and $\beta_1 \# \alpha_1 = \beta$. Here we use the concatenation maps
\begin{equation}
	\aligned
	& \# : \Pi_{2}(X;L_1,L_0;p,q)\times \Pi_{2}(X;L_0,\bbZ)\to \Pi_{2}(X;L_1,L_0;p,q), \\
	& \# : \Pi_{2}(X;L_1,L_0;p,q)\times \Pi_{2}(X;L_1,\bbZ)\to \Pi_{2}(X;L_1,L_0;p,q).
	\endaligned
\end{equation}
See Figures \ref{Figure1}, \ref{Figure2} and \ref{Figure3} for schematic pictures of the three types of boundary points.

\begin{figure}[h]
\centering
\includegraphics[scale=0.55]{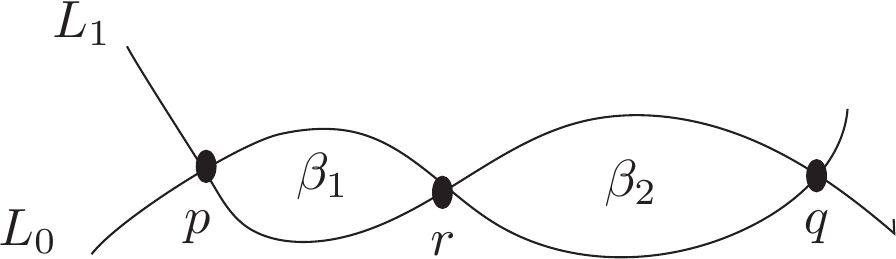}
\caption{A boundary element of type (1)}
\label{Figure1}
\end{figure}
\begin{figure}[h]
\centering
\includegraphics[scale=0.45]{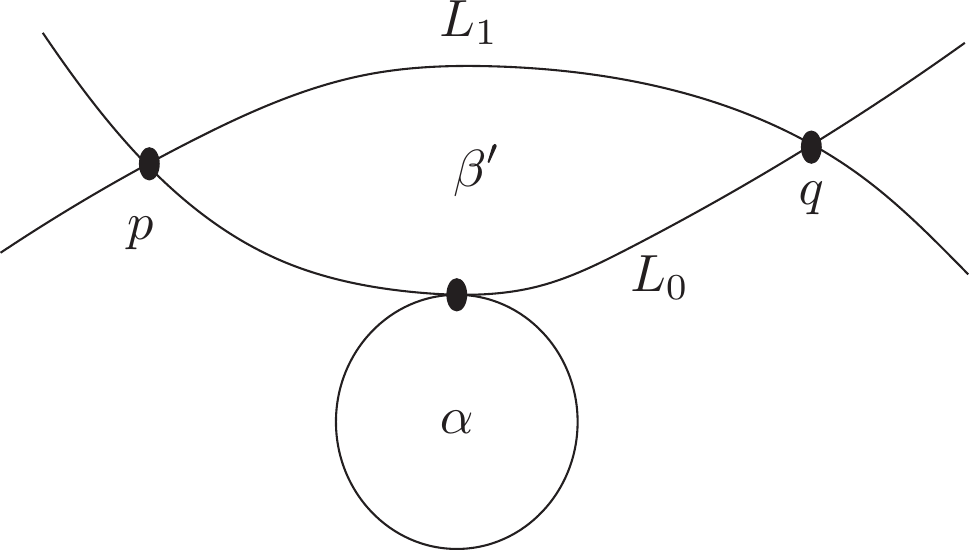}
\caption{A boundary element of type (2)}
\label{Figure2}
\end{figure}
\begin{figure}[h]
\centering
\includegraphics[scale=0.45]{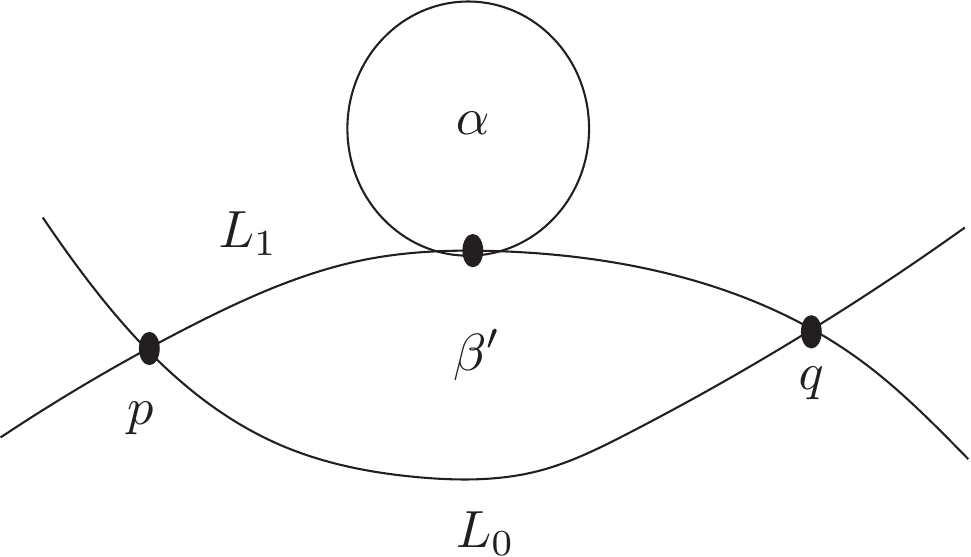}
\caption{A boundary element of type (3)}
\label{Figure3}
\end{figure}

In order to use the above moduli spaces in the construction of Lagrangian Floer homology, we need some sort of smooth structures on them. A general approach to achieve this is to show that there is a structure of a Kuranishi structure on any of these moduli spaces. Roughly speaking, a Kuranishi strcuture on a topological space $M$ implies that $M$ is locally homeomorphic to the zero set of a smooth map defined on a manifold (or more generally an orbifold), and the transition maps between such charts are well-behaved with respect to this structure. All of the moduli spaces in \eqref{4-moduli} admit Kuranishi structures with boundary and corners. In the case of $\mathcal M(L_1,L_0;p,q;\beta)$, the (normalized) boundary\footnote{See \cite[Definition 8.4]{fooo:tech2-1} for the definition of normalized boundary.} of this space as a Kuranishi space is the union of the three types of Kuranishi spaces given in \eqref{type-I} and \eqref{type-II-III}. Moreover, if the above Lagrangians are (relatively) spin, then all of the above Kuranishi structures are orientable. See \cite{fooobook2} for the proofs of the above claims.

Now, we turn back to showing that the operator $\partial$ in \eqref{differential} is a differential of a chain complex. The count of boundary elements of the first type in \eqref{type-I} gives the coefficient of $[q]$ in $\partial\circ \partial([p])$. Since the signed count of the boundary elements of a Kuranishi space of dimension $1$ is zero, this implies that $\partial\circ \partial=0$, assuming the boundary elements of the second the third types in \eqref{type-II-III} are empty. However, this does not happen in general and $\partial$ might not be a differential. In the special case that  $L_0,L_1$ are monotone in $X$ with minimal Maslov number greater than $2$, the contribution of the boundary points in \eqref{type-II-III} is trivial. This gives rise to the Oh's construction of Floer homology of monotone Lagrangians \cite{oh}.

In order to prove Theorem \ref{mainthm} where the Lagrangians $L_i$ are monotone only in $X \setminus \mathcal D$, we 
may try to restrict $\beta$ to the classes which satisfy the following additional condition.
\begin{conds}\label{cond1}
	We say $\alpha \in \Pi_2(X,L)$ has vanishing algebraic intersection 
	with $\mathcal D$, if
	\begin{equation}\label{form2626}
		[\alpha] \cdot \mathcal D = 0.
	\end{equation}
	Similarly,  $\beta \in \Pi_2(X;L_1,L_0;p,q)$ has vanishing algebraic intersection with $\mathcal D$, if
	\begin{equation}
		[\beta] \cdot \mathcal D = 0.
	\end{equation}
\end{conds}
In the definition of \eqref{differential}, suppose we only use homology classes $\beta \in \Pi_2(X;L_1,L_0;p,q)$ satisfying Condition \ref{cond1}. Then we might hope that the monotonicity of $L_0$, $L_1$ in $X\setminus \mathcal D$ allows us to repeat Oh's argument and avoid boundary elements of the second and the third types in \eqref{type-II-III}. This idea, however, does not work in general. Suppose $\beta$ has vanishing algebraic intersection with $\mathcal D$, and $\mathcal M(L_1,L_0;p,q;\beta)$ has virtual dimension $1$. Then this space could have boundary elements of the first type as in \eqref{type-I} associated to homology classes $\beta_1$, $\beta_2$ such that $\beta_1 \# \beta_2 = \beta$, $[\beta_1] \cdot \mathcal D<0$ and $[\beta_2] \cdot \mathcal D >0$. (See Figure \ref{Figure4} below.) Therefore, we would face again with a similar issue to show that $\partial\circ \partial =0$. 
(See Figure \ref{Figure4} below.)
\begin{figure}[h]
\centering
\includegraphics[scale=0.4]{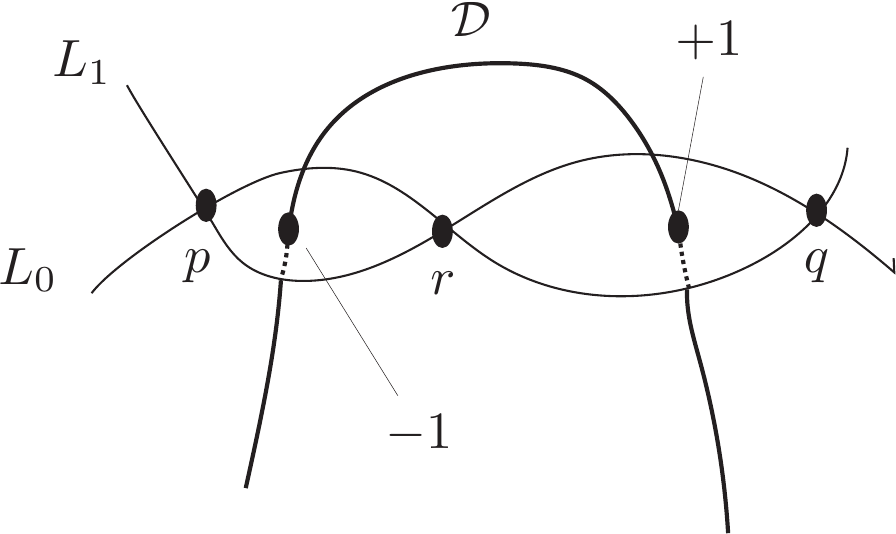}
\caption{Monotonicity is broken}
\label{Figure4}
\end{figure}

The arrangement in Figure \ref{Figure4} can be avoided by picking an appropriate almost complex structures on $X$. For instance, if $J$ is integrable in a neighboohod of $\mathcal D$ 
and $\mathcal D$ is a complex submanifold, then any $J$-holomorphic curve has a positive intersection with $\mathcal D$, which implies the claim in the following lemma. (See Subsection \ref{subsec:sympproj} for a more flexible family of almost complex structures on $X$ that satisfy a similar property.)
\begin{lemma}\label{lem27}
	Suppose the almost complex structure $J$ on $X$ is integrable in a neighboohod of $\mathcal D$.
	If $\mathcal M^{\rm reg}(L_1,L_0;p,q;\beta)$ (resp. $\mathcal M^{\rm reg}_1(L;\alpha)$) is nonempty, 
	then $\beta\cdot \mathcal D \ge 0$ (resp. $\alpha\cdot \mathcal D \ge 0$).
\end{lemma}

\begin{figure}[h]
\centering
\includegraphics[scale=0.4]{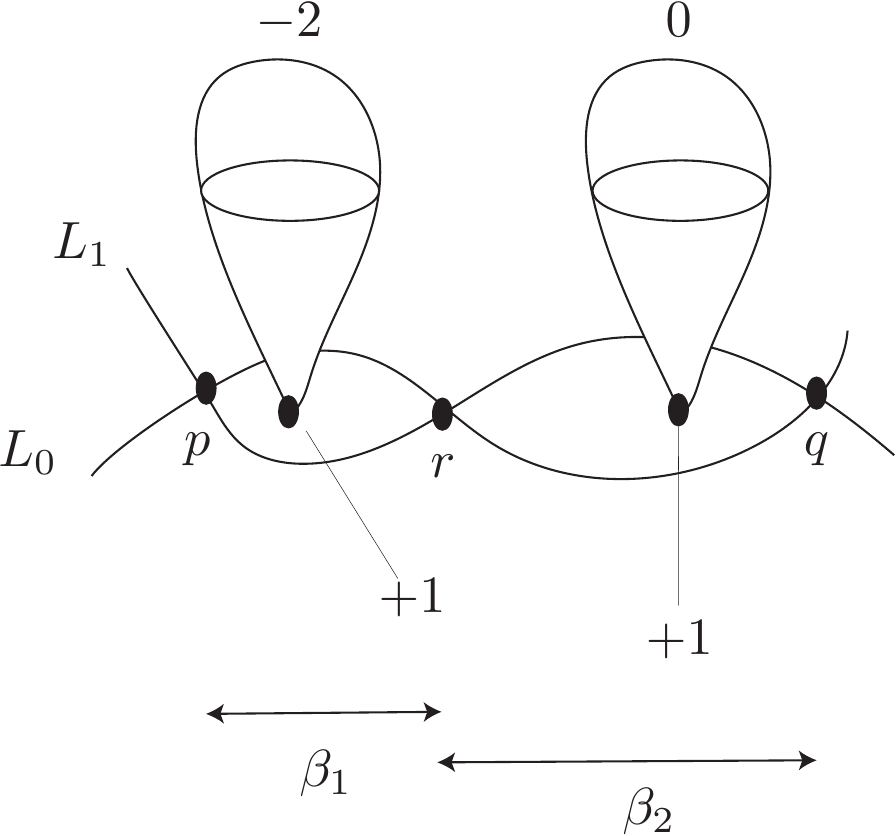}
\caption{Monotonicity is broken: 2}
\label{Figure5}
\end{figure}

Even if $J$ satisfies the assumption in Lemma \ref{lem27}, we may still have an arrangement as in Figure \ref{Figure5} that causes the issue we raised above. In Figure \ref{Figure5} two sphere bubbles are completely contained in the divisor $\mathcal D$. The numbers $0$ and $-2$ written at the top of the sphere components are the intersection numbers of those sphere bubbles with $\mathcal D$. (Note that under the assumption in Lemma \ref{lem27}, the intersection numbers can be nonpositive only when the spheres are contained in the divisor.) The two strips (joining $p$ to $r$ and $r$ to $q$) intersect with $\mathcal D$ at the {\it roots} of the sphere bubbles. The intersection number of the strips with $\mathcal D$ are both $+1$ as drawn in the figure. (This number is necessarily positive because of Lemma \ref{lem27}.) In this case, $\beta_1$ and $\beta_2$ are homology classes of the strips together with sphere bubbles on them. Therefore, $\beta_1 \cdot [\mathcal D] = -1$ and $\beta_2 \cdot [\mathcal D] = +1$.

The key idea to resolve this issue is to replace stable map topology with a stronger topology, called RGW topology. The main issue with the stable map compactification (and its Kuranishi structure) is its insensitivity with respect to the assumption that $\mathcal D$ is a divisor and the almost complex structure $J$ has a special behavior in a neighborhood of this submanifold of $X$. The compactified moduli spaces with respect to RGW topology admit Kuranishi structures that take into account the special geometry of $X$ around the divisor $\mathcal D$. In particular, an analogue of Lemma \ref{lem27} is built into the definition of the RGW topology. In the rest of this paper, we introduce the RGW topology and discuss the compactification of the moduli spaces with respect to this topology. Kuranishi structures on these moduli spaces are constructed in \cite{part2:kura,part3:FH} and then we explain there how we can use the monotonicity of $L_i$ in $X \setminus \mathcal D$ to adapt Oh's argument to prove Theorem \ref{mainthm}.

Before closing this section, we explain how the method of this paper compares with the existing works on relative Gromov--Witten invariants. Relative Gromov-Witten theory provides invariants of a pair $(X,\mathcal D)$ of a symplectic manifold together with a divisor using the moduli spaces of pseudo-holomorphic curves $u$ from a Riemann surface of surface $\Sigma_g$ with $k$ marked points $\{w_1,\dots,w_k\}\subset \Sigma_g$ into $X$ such that $u^{-1}(\mathcal D)$ is equal to $\{w_1,\dots,w_k\}$ and the multiplicity of the intersection of $u$ at the point $w_i$ with the divisor $\mathcal D$ is a fixed positive integer $m_i$. (More generally, one can consider the case that $\mathcal D$ is a normal crossing divisor.) One can approach this theory with the methods of algebraic geometry (assuming integrability of the pair $(X,\mathcal D)$) and symplectic geometry. On the algebro-geometric side, such a theory is developed and studied in J. Li \cite{JLi2,JLi}, Gross and Siebert \cite{groSie} and others. In the category of symplectic manifolds, relative Gromov-Witten invariants were defined in several works under various assumptions in the works of Li and Ruan \cite{LR}, Ionel and T. Parker \cite{IP,IP2}, B. Parker \cite{pa,pa2,pa3,pa5,pa4}, Tehrani and Zinger \cite{TZ} and Tehrani \cite{T}. (See also \cite{T2} on open Gromov-Witten invariants.) A review of some of these works can be found in \cite{TZ:exp}.

As it was mentioned in the introduction, relative Gromov-Witten theory is an important source of inspiration for parts of our construction. In particular, the basic idea of the notion of RGW topology and its compactification already appears in \cite[Proposition 7.3]{IP}, \cite[Theorem 6.1]{pa2} and \cite[Definition 3.7]{T}. However, we decided to give a self-contained review of the definition of the RGW compactification and the RGW topology on this space in our present setup. One reason is that the above works on relative Gromov-Witten invariants concern moduli of pseudo-holomorphic maps from source curves that have empty boundary. In the case of pseudo-holomorphic curves satisfying Lagrangian boundary condition on the boundary of the source curve, several new points about our moduli spaces need to be further studied. For instance, our construction of Floer theory requires an explicit understanding of codimension one boundary, where there is a new feature in our situation which does not appear in the previous works on Lagrangian Floer theory (see \cite[Subsection 2.2]{part3:FH} for more details). 

In order to construct the version of Floer homology promised in Theorem \ref{mainthm}, we use virtual fundamental chain techniques. In the case of relative Gromov-Witten theory, the relevant moduli spaces are expected to be manifolds without boundary (more precisely Kuranishi spaces without boundary), and one needs to associate a virtual fundamental {\it cycle} to any of them. On the other hand, the moduli spaces relevant for us are expected to have boundary and hence the construction of virtual fundamental chains in this setup would be a more delicate task. To achieve this goal, we need a detailed description of the strata of the RGW compactification. In Section \ref{Sec:RGW-Compactification}, we use some combinatorial data in the form of graphs with additional decorations to give such descriptions.

To the best of our understanding, a detailed construction of virtual fundamental cycles for relative Gromov-Witten invariants in the symplectic category (without any semi-positivity assumption on the divisor) appears only in the works of B. Parker on exploded manifolds. Parker constructs such virtual fundamental cycles using techniques from toric geometry and sheaf theory. It seems reasonable to expect that one can approach the construction of the virtual fundamental chains required for the proof of Theorem \ref{mainthm} from the same approach.

\section{RGW Compactification of the Moduli Space of Disks and Strips in $X \setminus \mathcal D$}
\label{Sec:RGW-Compactification}

As we explained in the last section, the stable map compactification is not suitable for the proof of Theorem \ref{mainthm}. In this section, we start the task of defining the alternative RGW compactification. In particular, we define the sets
\begin{equation}\label{4-moduli-RGW}
  \begin{array}{ccc}
	\mathcal M^{\rm RGW}(L_1,L_0;p,q;\beta),&&\mathcal M^{\rm RGW}_{0,1}(L_1,L_0;p,q;\beta),\\
	\mathcal M^{\rm RGW}_{1,0}(L_1,L_0;p,q;\beta),&&\mathcal M^{\rm RGW}_1(L;\alpha).
  \end{array}
\end{equation}
that contain the spaces in \eqref{4-moduli-pre-cpct}. In the next section, we define a topology on this sets that agree with the standard topology on the subspaces given in \eqref{4-moduli-pre-cpct}. Moreover, we shall show that the spaces in \eqref{4-moduli-RGW} are compact and metrizable. 

\begin{remark}
	In  \cite{part2:kura,part3:FH}, we show that the RGW compactifications admit Kuranishi structures. Moreover, the normalized boundary of $\mathcal M^{\rm RGW}(L_1,L_0;p,q;\beta)$
	can be split into three types similar to (but not exactly in the same way as in) the case of the stable compactification of moduli spaces of holomorphic strips into a compact symplectic manifold, 
	which was reviewed in the previous section.
\end{remark}

\begin{remark}
	We do not need the monotonicity assumption to prove the above claims (including the existence of Kuranishin structures) 
	about the spaces in \eqref{4-moduli-RGW}. 
	We use the monotonicity to derive Theorem \ref{mainthm} from the above claims on the structure of the spaces in \eqref{4-moduli-RGW}.
\end{remark}

\subsection{Partial {$\bbC_*$}-actions and divisors}
\label{subsec:Cstaraction}
Given a smooth divisor $\mathcal D$ in $X$, there is a partially defined $\bbC_*$ action in a regular neighborhood of $\mathcal D$. This action plays a central role in our construction of the RGW compactification. So we take a moment to formalize the notion of a partial $\bbC_*$ action on a manifold with an almost complex structure.
\begin{definition}\label{partial-C-action}
	Let $(Y,J)$ be an almost complex manifold and $D$  a codimension 2 submanifold of $Y$. 
	A partial $\bbC_*$ action on $(Y,D)$ is a pair $(\mathscr U,\frak m)$ where:
\begin{enumerate}
	\item $\mathscr U$ is an open neighborhood of $\bbC \times D$ in $\bbC\times  Y$.
	\item $\frak m : \mathscr U \to Y$ is a smooth map. For $(c,p) \in \mathscr U$ we write $c\cdot p$ for $\frak m(c,p)$.
		We say $c\cdot p$ is defined if  $(c,p) \in \mathscr U$.
	\item If $c_2 \cdot p$ and $c_1c_2\cdot p$ are defined, then $c_1\cdot (c_2 \cdot p)$ is also defined and is equal to
		$c_1c_2\cdot p$.
	\item If $c \cdot p$ is defined and $c \ne 0$, then $c^{-1}\cdot (c \cdot p)$ is defined and is $p$.
	\item If $p \in D$, then for any $c$, $c\cdot p$ is defined and is equal to $p$ .
	\item If $0\cdot p$ is defined, then $0 \cdot p \in D$.
\item
For $c \ne 0$, the map $p \mapsto c\cdot p$ preserves almost complex structure 
on the domain where it is defined.
\end{enumerate}
\end{definition}
\begin{definition}\label{nohrmalbunnC}
	Let $D$ be an almost complex manifold and $\pi: L\to D$ be a complex line bundle over $D$. Suppose also $\theta\in \Omega^1(L)$ is a connection $1$-form on $L$.
	Then for any $p\in L$, the {\it horizontal} subspace $H_p:=\ker(\theta_p)$ of $T_pL$ can be equipped with a complex structure by requiring that the derivative of $\pi$
	is complex linear as an isomorphism from $H_p$ to $T_{\pi(p)}D$. Thus we obtain a complex structure $J$ on $L$ by requiring that the tangent space to the fiber at $p$ with its standard complex structure and $H_p$ 
	are complex subspaces of $T_pL$. With respect to this complex structure, scaling in the fiber direction provides a partial $\bbC_*$ action on $(L,D)$ in the sense of Definition \ref{partial-C-action}.
\end{definition}
\begin{example}\label{nohrmalbunnC-gen}
	We can generalize Definition \ref{nohrmalbunnC} in the following way.
	Let $Y$ be an almost complex submanifold and $D$ be a codimension 2 submanifold of $Y$ such that for a neighborhood $U$ of $D$ there is a map $\Phi:U\to L$ that is a diffeomorphism into its image
	preserving the almost complex structures and the restriction of $\Phi$ to $D\subset U$ is the identity map into the zero section of $L$. Then we may pull back the partial $\bbC_*$ action on $(L,D)$ to obtain a partial 
	$\bbC_*$ action on $(Y,D)$.
\end{example}

Now suppose $(L,J)$ be as in Definition \ref{nohrmalbunnC} and $u:(\Sigma,j) \to (L,J)$ be a pseudo-holomorphic map from a Riemann surface to $L$. Thus $v:=\pi \circ u$ is a pseudo-holomorphic map into $D$, and $\theta_u:=v^*\theta$ defines a connection on the pulled back bundle $L_u:=v^*L$. Applying Definition \ref{nohrmalbunnC} to $L_u$ and $\theta_u$ determines a complex structure $J_u$ on the total space of $L_u$. The map $u$ induces a section $s_u:\Sigma\to L_u$ that is pseudo-holomorphic with respect to $j$ and $J_u$. We may also use the $(0,1)$ component $\overline \partial_\theta$ of $\theta_v$ to define a holomorphic structure on $L_u$: a local holomorphic section of $L_u$ is given by a solution of $\overline \partial_\theta v=0$. Since $\Sigma$ is a Riemann surface, there is no integrability obstruction to solve this equation to find local holomorphic charts for $L_u$. The connection $\theta_u$ with respect to a holomorphic chart for $L_u$ over some open subspace $V$ of $\Sigma$ has the form $d+\alpha$ where $\alpha\in \Omega^{(1,0)}(V)$. In particular, the underlying complex structure on the holomorphic bundle $L_u$ agrees with $J_u $, and $s_u$ is a holomorphic section of $L_u$. We summarize this discussion in the following lemma.

\begin{lemma}\label{pull-back-dble}
	Suppose $(L,J)$ is given as in Definition \ref{nohrmalbunnC}. For any pseudo-holomorphic map $u:(\Sigma,j) \to (L,J)$, the induced section $s_u$ of the line bundle $L_u$ over $\Sigma$ is holomorphic.
\end{lemma}


\subsection{Symplectic and Complex Structures on a Projective Space Bundle}
\label{subsec:sympproj}
We turn back to our setup where $\mathcal D$ is a divisor in $(X,\omega)$. In this subsection, we determine almost complex structures that are used in our main construction. Let $\omega_{\mathcal D}$ denote the induced symplectic structure on $\mathcal D$, and let $\mathcal N_{\mathcal D}(X) $ be the normal bundle of $\mathcal D$ in $X$. Fix an $\omega_{\mathcal D}$-compatible almost complex structure $J_{\mathcal D}$ on $\mathcal D$. Fix a compatible almost structure on the symplectic vector bundle $\mathcal N_{\mathcal D}(X) $ to turn it into a Hermitian line bundle. Let also $\theta$ be a ${\rm U}(1)$-connection on $\mathcal N_{\mathcal D}(X)$. Then Definition \ref{nohrmalbunnC} allows us to fix a complex structure on $\mathcal N_{\mathcal D}(X)$.

One of the ingredients of our RGW compactification are pseudo-holomorphic maps into the {\it projective bundle} ${\bf P}(\mathcal N_{\mathcal D}(X) \oplus \bbC)$. This space consists of equivalence classes of $(a,b) \in (\mathcal N_{\mathcal D}(X)  \times \bbC) \setminus \mathcal D \times \{0\}$ such that $(a,b) \sim (a',b')$ if there exists $\lambda \in \bbC_*$ with $(a,b) = (\lambda a',\lambda b')$. This space is a ${\bf P}^1$-bundle, which determines a compactification of the normal bundle $\mathcal N_{\mathcal D}(X) \to \mathcal D$. The complex structure on $\mathcal N_{\mathcal D}(X)$ determines an almost complex structure $J_{\bf P}$ on ${\bf P}(\mathcal N_{\mathcal D}(X) \oplus \bbC)$.  There is an action of $\mathbb C_*$ on ${\bf P}(\mathcal N_{\mathcal D}(X) \oplus \bbC)$ where $\lambda\in \bbC$ maps the equivalence class $[a,b]$ to $[\lambda a,b]$. The complex structure $J_{\bf P}$ is invariant with respect to the action of $\bbC_*$.

We also fix a symplectic structure on ${\bf P}(\mathcal N_{\mathcal D}(X) \oplus \bbC)$, invariant with respect to the action of $S^1 \subset \bbC_*$. Let $\varphi:[0,\infty)\to [0,2)$ be a smooth function satisfying the following properties:
\begin{enumerate}
	\item[(i)] $\varphi(r)=\frac{r^2}{2}$ for $r\in [0,1]$;
	\item[(ii)] $\varphi(r)=2-\frac{1}{r}$ for $r\in [2,\infty)$;
	\item[(iii)] $\varphi'(r)>0$.
\end{enumerate}
The Hermitian structure on $\mathcal N_{\mathcal D}(X)$ determines a length function $r:\mathcal N_{\mathcal D}(X) \to [0,\infty)$. The exact 2-form $d(\varphi(r)\theta)$ on $\mathcal N_{\mathcal D}(X)\setminus \mathcal D$ extends to a smooth closed 2-form on the projective bundle ${\bf P}(\mathcal N_{\mathcal D}(X) \oplus \bbC)$ which is $S^1$-invariant and whose restriction to each ${\bf P}^1$-fiber is a volume form. Therefore, if $K$ is a large enough constant, then the following form defines an $S^1$-invariant symplectic form on ${\bf P}(\mathcal N_{\mathcal D}(X) \oplus \bbC)$:
\[
  \omega_{\bf P}:=\pi^*\omega_{\mathcal D}+\frac{1}{K}d(\varphi(r)\theta).
\]
By choosing $K$ large enough, we may also assume that $J_{\bf P}$ is tame with respect to $\omega$.

Darboux's Theorem for symplectic submanifolds implies that there is a symplectomorphism $\Phi$ from a neighborhood of the zero section in ${\bf P}(\mathcal N_{\mathcal D}(X) \oplus \bbC)$ to a neighborhood of $\mathcal D$ in $X$ that restricts to the identity map on the zero section \cite[Theorem 3.30]{MS}. We use one such symplectomorphism to push forward $J_{\bf P}$ to a tame almost complex structure on a neighborhood of $\mathcal D$ in $X$. Then we extend this almost complex structure into a tame almost complex structure $J$ on $X$. With this choice of almost complex structure on $X$, there is a partial $\bbC_*$ action on $(X,\mathcal D)$. The moduli spaces of pseudo-holomorphic curves in $X$ and ${\bf P}(\mathcal N_{\mathcal D}(X) \oplus \bbC)$ in the rest of the paper are defined with respect to $J$ and $J_{\bf P}$. The following lemma about such pseudo-holomorphic curves is crucial for our construction.

\begin{lemma}\label{pos-mult}
	Suppose $u:(\Sigma,j)\to (X,J)$ is a pseudo-holomorphic curve. Then the multiplicity of any intersection point of $u$ 
	and $\mathcal D$ is a positive integer. A similar claim holds if $(X,J)$ is replaced with 
	$({\bf P}(\mathcal N_{\mathcal D}(X) \oplus \bbC,\mathcal D_0\cup \mathcal D_\infty)$ where $\mathcal D_0$ and 
	$\mathcal D_\infty$ are given by the sections at zero and infinity.
\end{lemma}

\begin{proof}
	It suffices to consider the pseudo-holomorphic maps $u$ that are mapped to a tubular neighborhood of $\mathcal D$ where 
	the almost complex structure has the form given in Definition \ref{nohrmalbunnC}. In this case, the claim follows from 
	Lemma \ref{pull-back-dble} and the corresponding positivity result in the holomorphic category.
\end{proof}

\begin{remark}\label{weak-contr}
	The almost complex structures $J$ and $J_{\bf P}$ depend on $J_{\mathcal D}$, the compatible complex structure on the symplectic vector 
	bundle $\mathcal N_{\mathcal D}(X) $, the connection $\theta$, the 
	symplectomorphism $\Phi$ and the extension of $\Phi_*(J_{\mathcal D})$ into a tame almost complex structure on $X$. The 
	space of all such choices has trivial homotopy groups. 
	This fact will be used in \cite{part3:FH} to show that our version of Lagrangian Floer homology does not depend on the choices of almost complex structures.
\end{remark}

\begin{remark}
	We use the partial $\bbC_{*}$ action on $(X,\mathcal D)$ in the definition of the RGW topology and the gluing analysis. To carry this out, it is important for us that our almost complex structure is invariant with respect to this 
	partial $\bbC_*$ action.
	We believe that a similar construction can be carried out for a slightly more general family of almost complex structures at the expense of more work on the analysis part. 
	In fact, it is reasonable to expect that similar assumptions on almost complex structures as in the literature on relative Gromov-Witten theory, where one does not require 
	invariance with respect to a partial $\bbC_*$ action, would be sufficient for our purposes. For instance, \cite{IP} works with almost complex structures that satisfy an infinitesimal integrability in the normal direction to
	 $\mathcal D$, and $\mathcal D$ is an almost complex submanifold \cite[Definition 3.2]{IP}. The almost complex structure $J$ constructed above satisfies these conditions. 
	 On the other hand, integrable almost complex structures are often not invariant with respect to any partial $\bbC_*$ action, but still fits into the framework of \cite[Definition 3.2]{IP}.	
	 Since we do not gain anything from working in such generality and the analytical aspects of partially $\bbC_*$-invariant almost complex structures are simpler, we content ourselves with this more restricted family of
	 almost complex structures.
\end{remark}

\subsection{RGW Compactification in a Neighborhood of the Divisor} \label{subsec:nbdofdivisor}
Fix non-zero integers $m_0,m_1,\dots,m_{\ell} \in \bbZ \setminus \{0\}$ and let ${\bf m} =(m_0,m_1,\dots,m_{\ell})$. Let also $\Pi_2(\mathcal D)  = {\rm Im}(\pi_2(\mathcal D) \to H_2(\mathcal D))$. For $\alpha \in \Pi_2(\mathcal D;\bbZ)$, the moduli space  $\mathcal M^{0}(\mathcal D\subset X;\alpha;{\bf m})$ is defined as follows:

\begin{definition}\label{defn3434}
	An element of $\mathcal M^{0}(\mathcal D\subset X;\alpha;{\bf m})$ is an isomorphism class of a triple 
	$((\Sigma,\vec w);u;s)$ with the following properties:
        \begin{enumerate}
        \item
      	  $(\Sigma,\vec w)$ is a nodal curve of genus zero 
	        with $\ell+1$ marked points $\vec w = (w_0,\dots,w_{\ell})$.
        		The marked points $w_i$ are not nodal points.
        \item
	        $u : \Sigma \to \mathcal D$ is a holomorphic map representing the homology class $\alpha$.
        \item
        		$s$ is a section of $u^*\mathcal N_{\mathcal D}(X)$ on $\Sigma \setminus \vec w$
	        and is extended to a meromorphic section on $\Sigma$.
	\item
		For $i=0,1,\dots,\ell$, the section $s$ has a zero of multiplicity $m_i$ at $w_i$  if $m_i > 0$, and
		it has a pole of multiplicity $-m_i$ at $w_i$ if $m_i < 0$.
		$s$ is nonzero on $\Sigma \setminus \{w_0,\dots,w_{\ell}\}$.
	\item
		The stability condition defined in Definition \ref{stable} holds.
	\end{enumerate}
	The definition of isomorphism between two such elements are given in Definition \ref{defn35}. 
\end{definition}

\begin{definition}\label{defn35}
	Let ${\bf x} = ((\Sigma,\vec w);u;s)$ and ${\bf x}' = ((\Sigma',\vec w^{\,\prime});u';s')$
	be as in Definition \ref{defn3434}.
	An {\it isomorphism} from ${\bf x}$ to ${\bf x}'$ is 
	a pair $(v,c)$ such that:
	\begin{enumerate}
	\item $v : \Sigma \to \Sigma'$ is a biholomorphic map 
		such that $u' \circ v = u$,
	\item $c$ is a nonzero complex number such that $s' \circ v = cs$.
	\end{enumerate}
	We say ${\bf x}$ is {\it isomorphic} to ${\bf x}'$ if there exists an 
	isomorphism between them. 
	We say ${\bf x}$ is {\it strongly isomorphic} to ${\bf x}'$ if we can additionally assume $c=1$.
	We denote by $\widetilde{\mathcal M}^{0}(\mathcal D\subset X;\alpha;{\bf m})$
	the set of all strong isomorphism classes.
\end{definition}
\begin{definition}\label{stable}
We say an element ${\bf x} = [(\Sigma,\vec w);u;s]$
as in Definition \ref{defn3434} is {\it stable}
if the set of isomorphisms from ${\bf x}$ to itself 
is a finite set.
\end{definition}

Note that there exists a $\bbC_{*}$ action on $\widetilde{\mathcal M}^{0}(\mathcal D\subset X;\alpha;{\bf m})$ such that
\begin{equation}\label{formula3636}
	\widetilde{\mathcal M}^{0}(\mathcal D\subset X;\alpha;{\bf m})/\bbC_{*} \\
	=
	{\mathcal M}^{0}(\mathcal D\subset X;\alpha;{\bf m}).
\end{equation}

\begin{remark}
	The holomorphic structure on $u^*\mathcal N_{\mathcal D}(X)$ used in Definition \ref{defn3434} is the one discussed in Subsection \ref{subsec:Cstaraction}.
\end{remark}

\begin{remark}
	We define:
	\[
	  d(\alpha):=[\mathcal D] \cdot \alpha.
	\]
	In the right hand side, we regard $[\mathcal D]$ and $\alpha$ as homology classes in $X$.
	We call $d$ the {\it degree} of $(\Sigma,u)$. If $\mathcal M^{0}(\mathcal D\subset X;\alpha;{\bf m})$ 
	is non-empty, then the definition implies that
	\begin{equation}\label{form3535}
		d(\alpha) = \sum_{i=0}^{\ell} m_i.
	\end{equation}
\end{remark}

\begin{definition}\label{evaluationmap1}
We define evaluation maps
$$
{\rm ev} = ({\rm ev}_0,\dots,{\rm ev}_{\ell}) 
: \mathcal M^{0}(\mathcal D\subset X;\alpha;{\bf m}) \to \mathcal D^{\ell+1}
$$
by
\begin{equation}\label{evdefin}
{\rm ev}_i((\Sigma,\vec w);u;s) = u(w_i).
\end{equation}
\end{definition}

Let $((\Sigma,\vec w);u;s)$ represent an element of $\mathcal M^{0}(\mathcal D\subset X;\alpha;{\bf m})$. Define the map $U$ from $\Sigma$ to the ${\bf P}^1$-bundle ${\bf P}(\mathcal N_{\mathcal D}(X) \oplus \bbC)$ over $\mathcal D$ as follows (see Figure \ref{Figure9}):
\begin{equation}\label{oomojiU}
	U(z) = [s(z):1] \in {\bf P}(\mathcal N_{\mathcal D}(X) \oplus \bbC).
\end{equation}
The discussion of Subsection \ref{subsec:Cstaraction} shows that $U$ is $J_{\bf P}$-holomorphic. The homology class of this map in $H_2({\bf P}(\mathcal N_{\mathcal D}(X) \oplus \bbC);\bbZ)$, denoted by $\hat\alpha$,
is uniquely determined by the following two properties:
\begin{enumerate}
\item[(i)]
	Projection of $\hat\alpha$ to $H_2(\mathcal D)$ is $\alpha$. Here
	$\mathcal D \subset {\bf P}(\mathcal N_{\mathcal D}(X) \oplus \bbC)$ 
	is identified with the zero section.
\item[(ii)]
	The algebraic intersection of the infinity section $\mathcal D_\infty$ and $\hat{\alpha}$ is given by:
	\[
	  \hat{\alpha} \cap [\mathcal D_{\infty}] =  - \sum_{m_i<0} m_i.
	\]
\end{enumerate}
Therefore, $((\Sigma,\vec w);U)$ defines an element of $\mathcal M_{\ell+1}({\bf P}(\mathcal N_{\mathcal D}(X) \oplus \bbC);\hat\alpha)$, the moduli space of stable maps of genus zero in ${\bf P}(\mathcal N_{\mathcal D}(X) \oplus \bbC)$ of homology class $\hat\alpha$ and with $\ell+1$ marked points. In particular, the stability in Definition \ref{stable} implies the stability of $U$ (as a holomorphic map from a nodal Riemann surface with marked points). The $\bbC_{*}$ action $c [a:b] = [ca:b]$ on ${\bf P}(\mathcal N_{\mathcal D}(X) \oplus \bbC)$ induces a $\bbC_{*}$ action on $\mathcal M_{\ell+1}({\bf P}(\mathcal N_{\mathcal D}(X) \oplus \bbC);\hat\alpha)$. The element $[((\Sigma,\vec w);U)]$ in the quotient space $\mathcal M_{\ell+1}({\bf P}(\mathcal N_{\mathcal D}(X) \oplus \bbC);\hat\alpha)/\bbC_{*}$ is independent of the choices of the representative $((\Sigma,\vec w);u;s)$ and so we may define a map 
\begin{equation}\label{map35}
\mathcal M^{0}(\mathcal D\subset X;\alpha;{\bf m})
\to
\mathcal M_{\ell+1}({\bf P}(\mathcal N_{\mathcal D}(X) \oplus \bbC);\hat\alpha)/\bbC_{*},
\end{equation} 
which is injective.

\begin{figure}[h]
\centering
\includegraphics[scale=0.4]{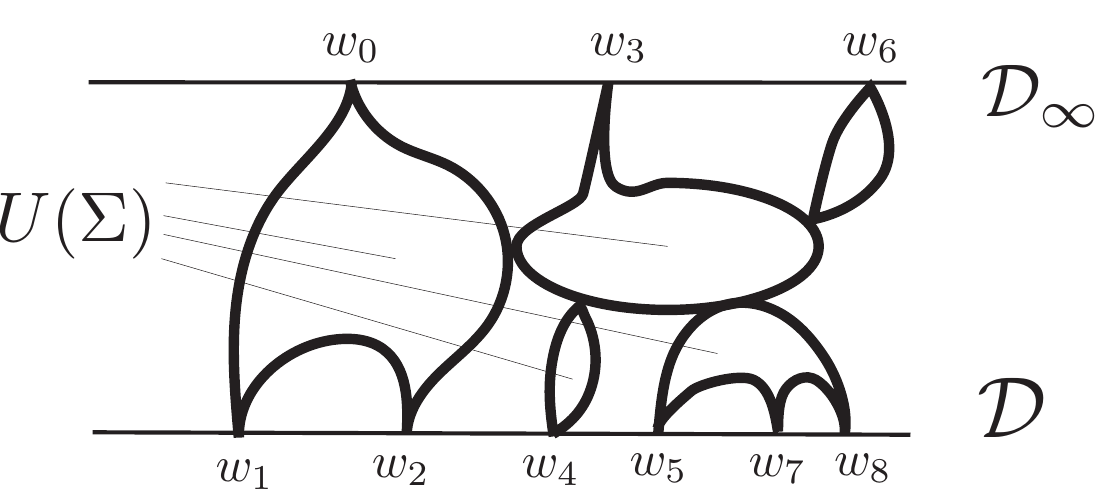}
\caption{An element of $\mathcal M^{0}(\mathcal D\subset X;\alpha;{\bf m})$}
\label{Figure9}
\end{figure}

Let $\mathcal M_{\ell+1}(\mathcal D;\alpha)$ be the moduli space of stable $J_{\mathcal D}$-holomorphic maps of genus $0$ in the manifold $\mathcal D$ with $\ell+1$ marked points and of homology class $\alpha$.
We define a map
\begin{equation}\label{form3737}
\mathcal M^{0}(\mathcal D\subset X;\alpha;{\bf m}) \to \mathcal M_{\ell+1}(\mathcal D;\alpha)
\end{equation}
by sending $[(\Sigma,\vec w);u;s]$ to $[(\Sigma,\vec w);u]$, namely, we  forget $s$ in 
$((\Sigma,\vec z);u;s)$.
(Note that stability is preserved by this process.)
\par
We denote by 
$\mathcal M^{00}(\mathcal D\subset X;\alpha;{\bf m})$
(resp. $\mathcal M^0_{\ell+1}(\mathcal D;\alpha)$) 
the subset of $\mathcal M^{0}(\mathcal D\subset X;\alpha;{\bf m})$
(resp. $\mathcal M_{\ell+1}(\mathcal D;\alpha)$)
consisting of elements such that $\Sigma$ is a sphere, namely,  
it consists of elements without nodal points.

\begin{lemma}\label{lem3.12}
	The map \eqref{form3737} is injective.
	Moreover, \eqref{form3737} induces a bijection
	\begin{equation}\label{312312}
		\mathcal M^{00}(\mathcal D\subset X;\alpha;{\bf m})\to \mathcal M^0_{\ell+1}(\mathcal D;\alpha)
	\end{equation}
	if \eqref{form3535} holds.
\end{lemma}
\begin{proof}
	Suppose  $[(\Sigma,\vec w);u;s]$, $[(\Sigma,\vec w);u;s']$ are 
	two elements of $\mathcal M^{0}(\mathcal D\subset X;\alpha;{\bf m})$ 
	mapped to the same element of $\mathcal M^0_{m+1}(\mathcal D;\alpha)$.
	Then the ratio $s'/s$ defines a holomorphic function on $\Sigma$ 
	which is nonzero everywhere. Therefore, $s'/s$ is constant.
	Thus $[(\Sigma,\vec w);u;s]=[(\Sigma,\vec w);u;s']$ in 
	$\mathcal M^{0}(\mathcal D\subset X;\alpha;{\bf m})$.
	It is also easy to see that \eqref{312312} is surjective when \eqref{form3535} holds.
\end{proof}
In the same way, we can prove:
\begin{lemma}\label{lem313}
	\eqref{form3737} is a bijection onto an open subset of $\mathcal M_{\ell+1}(\mathcal D;\alpha)$ (with respect to the stable map topology).
\end{lemma}
\begin{proof}
	Let $((\Sigma,\vec z),u)$ be an element of 
	$\mathcal M_{\ell+1}(\mathcal D;\alpha)$.
	We decompose $\Sigma$ into irreducible components as
	\[
	  \Sigma = \bigcup_a \Sigma_a.
	\]
	We can easily show that $[(\Sigma,\vec w),u]$ is in the image of \eqref{form3737}
	if and only if the next equalities hold for each $a$.
	\begin{equation}\label{condform311}
		\mathcal D \cdot u_*[\Sigma_a] = \sum_{z_i \in \Sigma_a} m_i.
	\end{equation}
	Condition \eqref{condform311}
	is an open condition with respect to the stable map topology.
	Therefore, the image of \eqref{form3737} is open.
\end{proof}

%

If we topologize $\mathcal M^{0}(\mathcal D\subset X;\alpha;{\bf m})$ using the bijection in Lemma \ref{lem313}, the resulting space is  not compact. This issue stems from non-compactness of $\bbC_{*}$.
Our next task is to compactify $\mathcal M^{0}(\mathcal D\subset X;\alpha;{\bf m})$.
The compactification has a stratification 
and each stratum is described by an appropriate fiber product of the 
spaces of the form
$\mathcal M^{0}(\mathcal D\subset X;\alpha';{\bf m}')$
for various choices of $\alpha'$, ${\bf m}'$.
The strata of our compactification are labeled with {\it decorated rooted trees}.
They also encode the data of how to take fiber products:
\begin{definition}\label{defn315}
A decorated rooted tree is a quadruple $\mathcal T =(T,\alpha,m,\lambda)$ 
with the following properties:
\begin{enumerate}
\item
$T$ is a tree with the set of vertices $C_0(T)$ and the set of edges $C_1(T)$. We are given a decomposition of $C_0(T)$ into the disjoint union of two subsets $C^{\rm out}_0(T)$ and $C^{\rm ins}_0(T)$.
We call an element of $C^{\rm out}_0(T)$ (resp. $C^{\rm ins}_0(T)$)
{\it an outside vertex} (resp. {\it  inside vertex}).
\item
All the outside vertices have valency one.
We call an edge incident to an outside vertex an {\it outside edge}. Any of the remaining edges is called an {\it inside edge}.
\item
There is a distinguished outside vertex $v^0$ of $T$.
Let $e^0$ be the unique edge which contains $v^0$.
We call $v^0$ and $e^0$ the {\it root vertex} and the {\it root edge}, respectively.
We call them root if it is clear from the context whether we mean the root vertex or the root edge.
We also fix a labeling $\{v^0,v^1,\dots,v^{\ell}\}$ of the outside vertices.
\item
$\alpha : C^{\rm ins}_0(T) \to \Pi_2(\mathcal D;\bbZ)$ is a map from the set of 
the inside vertices to $\Pi_2(\mathcal D;\bbZ)$.
We call $\alpha(v)$ the {\it homology class of $v$}.
\item
$m : C_1(T) \to \bbZ \setminus \{0\}$ is a $\bbZ \setminus \{0\}$-valued 
function 
which assigns a nonzero integer to each edge.
We call $m(e)$ the {\it multiplicity of $e$}.
\item
$\lambda : C^{\rm ins}_0(T) \to \bbZ_{+}$ is a $\bbZ_{+}$ valued function.
For a vertex $v$, we call $\lambda(v)$ the {\it level of $v$.} 
There exists $\vert \lambda\vert \in \bbZ_{+}$ such that 
the image of $\lambda$ is $\{1,2,\dots,\vert \lambda\vert\}$, namely, $\lambda$ is a surjective map to 
$\{1,2,\dots,\vert \lambda\vert\}$.
We call $\vert \lambda\vert$ the {\it number of levels}.
\item
	For each vertex $v \in  C^{\rm ins}_0(T)$, there exists a unique edge of $v$ which is contained in the 
	same connected component as the root in $T \setminus v$.
We call it the {\it first edge of $v$} and denote it by $e(v)$.
We then require the following {\it balancing condition}:
\begin{equation}\label{bananching}
m(e(v)) + \alpha(v) \cdot [\mathcal D] = \sum_{e \in  C_1(T): v \in e, e\ne e(v)} m(e).
\end{equation}
(This condition is the analogue of (\ref{form3535}).)
\item (Stability condition)
Each vertex $v \in  C^{\rm ins}_0(T)$ satisfies at least 
one of the following conditions:
\begin{enumerate}
	\item $v$ contains at least 3 edges.
	\item $\alpha(v) \cap [\omega_{\mathcal D}] > 0$. 
\end{enumerate}
\item
Let $e$ be an inside edge incident to the vertices $v$, $v'$.
We assume $e$ is $e(v')$, the first edge  of $v'$.
\begin{enumerate}
\item
	If $m(e) > 0$, then $\lambda(v) < \lambda(v')$.
\item
	If $m(e) < 0$, then $\lambda(v) > \lambda(v')$.
\end{enumerate}
In particular, $\lambda(v) \ne \lambda(v')$.
\end{enumerate}
For a decorated rooted tree ${\mathcal T}$, we define the {\it homology class of ${\mathcal T}$}
by
\begin{equation}
\alpha(\mathcal T) = \sum_{v \in  C^{\rm ins}_0(T)} \alpha(v)
\end{equation}
We say $m(e^0)$ is the {\it input multiplicity of $\mathcal T$} and the set $\{m(e^i) \mid e^i \in C^{\rm out}_1(T)\}$ gives the {\it output multiplicities of $\mathcal T$}. 
\par
Let $e$ belong to the set of inside edges $C^{\rm ins}_1(T)$.
The two vertices incident to $e$ are the {\it target vertex} $t(e)$ and {\it source vertex} $s(e)$ of $e$,
if $e$ is the first edge of $t(e)$.
\end{definition}
\begin{remark}\label{rem321322}
        We orient the edges so that it starts from the vertex $s(e)$ and ends 
        at the vertex $t(e)$. 
        (See Figure \ref{Figure10}.)
        Then for given vertex $v$ all the  
        edges other than the first edge $e(v)$ goes from $v$ to other vertices.
        This is consistent with the convention in \eqref{bananching} only $e(v)$ is on the left hand side.
\end{remark}
\begin{example}
An example of a decorated rooted tree $\mathcal T$ is given in Figure \ref{Figure10}.
In the figure, the outside vertices are drawn by black circles and 
inside vertices are drawn by white circles.
The input multiplicity of $\mathcal T$  is $3$ 
and its output multiplicity is $-1$. The number of levels is 4.
We also have 
$$
\aligned
&e(v_1) = e_0, \quad e(v_2) = e_1, \quad 
e(v_3) = e_2, \quad e(v_4) = e_3, 
\quad  e(v_5) = e_4, \\
& \alpha(v_2)\cdot\mathcal D = -3,
\quad \alpha(v_3)\cdot\mathcal D = -2, \quad \alpha(v_4)\cdot\mathcal D = -1, \quad \alpha(v_5)\cdot\mathcal D = 2,
\endaligned
$$
and
\[
  \alpha(v_1) =0.
\]
\end{example}

\begin{figure}[h]
\centering
\includegraphics[scale=0.4]{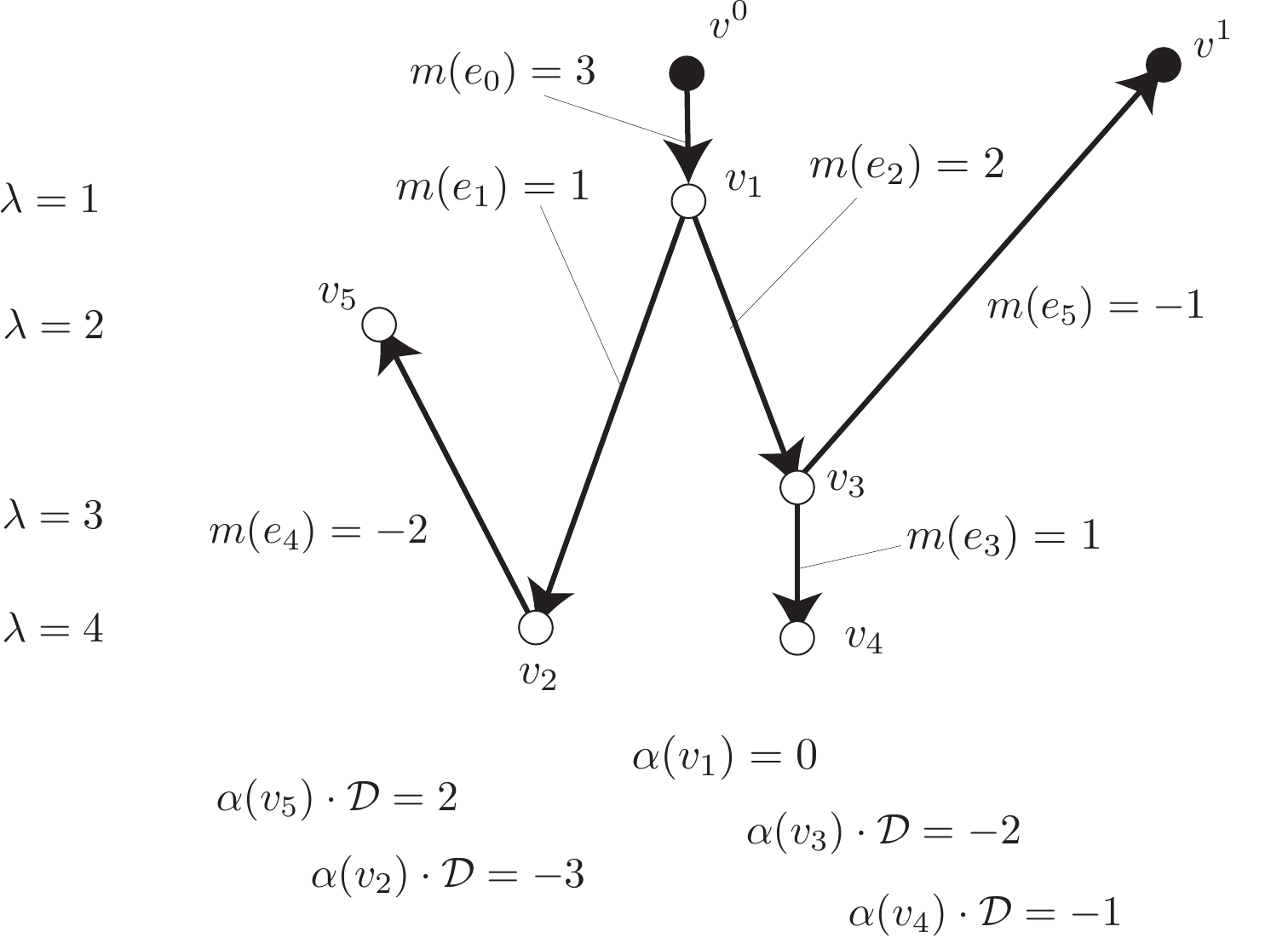}
\caption{Decorated rooted tree}
\label{Figure10}
\end{figure}
\begin{remark}
The notion of level here is similar to the one 
appearing in the compactification of the moduli space 
of pseudo-holomorphic curves in symplectic field theory
\cite{BEHWZ} and its generalization to stable Hamiltonian structures \cite{CV}.
\end{remark}
\begin{definition}\label{defn313}
To each decorated rooted tree $\mathcal T=(T,\alpha,m,\lambda)$, 
we associate a moduli space $\widetilde{\mathcal M}^0(\mathcal D;\mathcal T)$
as follows.
Fix an inside vertex $v \in  C^{\rm ins}_0(T)$.
Define $m^v_0$ to be $-m(e(v))$ where $e(v)$ is  the first edge of $v$ defined 
in Definition \ref{defn315} (7). Let $e^v_1,\dots,e^v_{\ell(v)}$ be the 
remaining edges of 
$v$, $m^v_i = m(e^v_i)$ and 
${\bf m}^v = (m^v_0,m^v_1,\dots,m^v_{\ell(v)})$.
Define:
\begin{equation}\label{fpr315315}
\widetilde{\mathcal M}^{0}(\mathcal D;\mathcal T;v)
= \widetilde{\mathcal M}^{0}(\mathcal D\subset X;\alpha(v);{\bf m}^v).
\end{equation}
Note that the right hand side is independent of the 
order of the edges $e^v_1,\dots,e^v_{\ell(v)}$.
\par 
Consider the evaluation map
\begin{equation}\label{formnew3188}
{\rm EV} :
\prod_{v \in  C^{\rm ins}_0(T)}\widetilde{\mathcal M}^0(\mathcal D;\mathcal T;v) 
\to \prod_{e \in C^{\rm ins}_1(T)} (\mathcal D \times \mathcal D)
\end{equation}
defined as follows.
Let $\vec{\bf x}= ({\bf x}_v;v \in  C^{\rm ins}_0(T))$
be an element of the domain of ${\rm EV}$. 
The $e$-th component ${\rm EV}(\vec{\bf x})_e$ of  ${\rm EV}(\vec{\bf x})$
is by definition:
\begin{equation}\label{Ev317def}
{\rm EV}(\vec{\bf x})_e
=
({\rm ev}_0({\bf x}_{t(e)}),{\rm ev}_i({\bf x}_{s(e)})) \in \mathcal D \times \mathcal D.
\end{equation}
Here ${\rm ev}_0$ and ${\rm ev}_i$ are as in 
Definition \ref{evaluationmap1} and $i$ is taken so that $e$ is the $i$-th edge of $s(e)$.
($s(e)$ and $t(e)$ are defined at the end of Definition \ref{defn315}.)
Let $\Delta  \subset \mathcal D \times \mathcal D$ be the diagonal.
We now define
\begin{equation}\label{form31666}
\aligned
\widetilde{\mathcal M}^0(\mathcal D;\mathcal T) 
= 
\prod_{v \in  C^{\rm ins}_0(T)}\widetilde{\mathcal M}^{0}(\mathcal D;\mathcal T;v)\,\,
{}_{\rm EV}\times_{\star} 
\prod_{e \in C^{\rm ins}_1(T)} \Delta.
\endaligned
\end{equation}
Here we take the fiber product over the space
$
\prod_{e \in C^{\rm ins}_1(T)} (\mathcal D \times \mathcal D)
$
and $\star$ denotes the inclusion of $\prod \Delta$ into $\prod(\mathcal D \times \mathcal D)$.
The other map in the definition of the fiber product is given in \eqref{formnew3188}. We topologize $\widetilde{\mathcal M}^0(\mathcal D;\mathcal T;v)$ with the fiber product topology.
\end{definition}
\begin{example}
Figure \ref{Figure11} sketches an element of $\widetilde{\mathcal M}^{0}(D;\mathcal T;v)$ 
for the decorated ribbon tree $\mathcal T$ given in Figure \ref{Figure11}.
Each of the vertical lines (4 of them) in the figure 
corresponds to a ``trivial cylinder'' that is 
a map to a single fiber of ${\bf P}(\mathcal N_{\mathcal D}(X) \oplus \bbC)$.
(More precisely, it is an $\vert m(e)\vert$ fold covering to a fiber.)
The number assigned to a vertical edge or a double point is the 
multiplicity of the corresponding edge of our tree.
The map $u : \Sigma \to \mathcal D$ 
in Figure \ref{Figure11} corresponding to ${\bf x}_{v_1}$
has homology class $0$ and is a constant map to a point $p$ of $\mathcal D$.
So the image of the map $U$ corresponding to ${\bf x}_{v_1}$ is contained 
in a fiber of the normal bundle $\mathcal N_{\mathcal D}(X)$ at $p$.
\begin{figure}[h]
\centering
\includegraphics[scale=0.4]{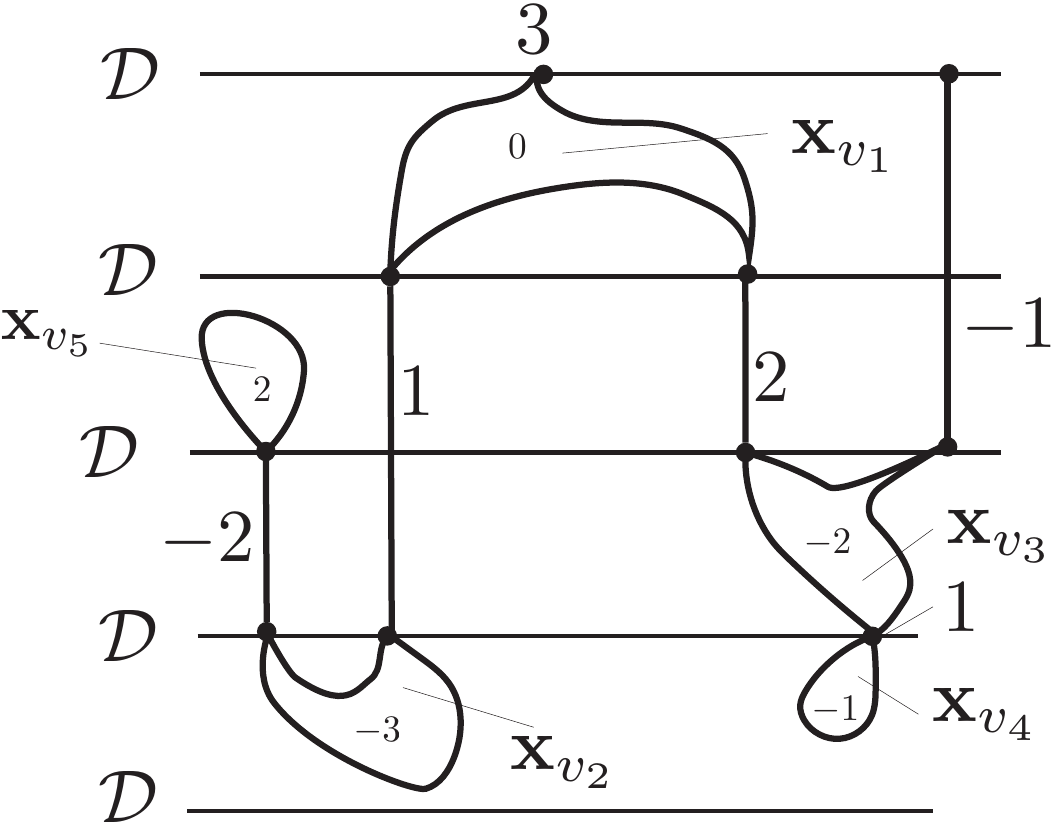}
\caption{Configuration associated to the decorated rooted tree
in Figure \ref{Figure10}}
\label{Figure11}
\end{figure}
\end{example}
\begin{definition}\label{Cstaractiondefn313}
	Let $\vert \lambda\vert$ be the number of the levels of $\mathcal T$. We define a $\bbC_{*}^{\vert \lambda\vert}$ action on 
	$\widetilde{\mathcal M}^0(\mathcal D;\mathcal T)$ as follows.
	For $\vec{\rho} = (\rho_1,\dots,\rho_{\vert \lambda\vert}) \in \bbC_{*}^{\vert \lambda\vert}$ and $\vec{{\bf x}} = ({\bf x}_v) 
	\in \prod_{v \in  C^{\rm ins}_0(T)}\widetilde{\mathcal M}(\mathcal D;\mathcal T;v)$.
        We have: 
        \begin{equation}\label{quotientbyci}
        \vec{\rho} \cdot \vec{{\bf x}} = (\rho_{\lambda(v)}{\bf x}_v).
        \end{equation}
        Note that $\lambda(v)$ is the level of the vertex $v$ and $\rho_{\lambda(v)} \in \bbC_{*}$.
        The $\bbC_{*}$ action on 
        $\widetilde{\mathcal M}(\mathcal D;\mathcal T;v)$ is defined 
        as in (\ref{formula3636}).
        The quotient space of this action 
        is denoted by
        $$
        \widehat{\mathcal M}^0(\mathcal D;\mathcal T)
        = \widetilde{\mathcal M}^0(\mathcal D;\mathcal T)/\bbC_{*}^{\vert \lambda\vert}.
        $$
        We use the quotient topology to topologize this space.
\end{definition}

Next, we take the quotient of $\widehat{\mathcal M}^0(\mathcal D;\mathcal T)$ with respect to the action of the 
group of automorphisms of the decorated rooted tree $\mathcal T$.
\begin{definition}
An automorphism of the decorated rooted tree $\mathcal T$ is an 
automorphism of the tree $T$ which fixes all the outside 
vertices and commutes with $\alpha$, $m$ and $\lambda$.  The group of automorphisms of $\mathcal T$, which is a finite group, 
is
denoted by ${\rm Aut}(\mathcal T)$.
\par
An element of ${\rm Aut}(\mathcal T)$ exchanges the vertices of 
$T$. Thus it  induces an automorphism of 
$\widetilde{\mathcal M}^0(\mathcal D;\mathcal T)$.
This action is compatible with the $\bbC_{*}^{\vert \lambda\vert}$ action.
Therefore, we obtain an action of ${\rm Aut}(\mathcal T)$ 
on $\widehat{\mathcal M}^0(\mathcal D;\mathcal T)$.
We denote the quotient space by
$$
{\mathcal M}^0(\mathcal D;\mathcal T)
= 
\widehat{\mathcal M}^0(\mathcal D;\mathcal T)/{\rm Aut}(\mathcal T).
$$
\end{definition}
\par
We say decorated rooted trees $\mathcal T = (\mathcal T,\alpha,m,\lambda)$ and $\mathcal T'
 = (\mathcal T',\alpha',m',\lambda')$ are isomorphic, if there exists an isomorphism of the underlying trees $T$, $T'$
which sends root to root, outside vertices to outside vertices, 
$\alpha$ to $\alpha'$, $m$ to $m'$, $\lambda$ to $\lambda'$, and preserves the ordering 
of the outside vertices.
From now on, we do not distinguish between  a decorated rooted tree and its 
isomorphism class.
\par
We now define a compactification of 
$\mathcal M^{0}(\mathcal D\subset X;\alpha;{\bf m})$ as a set.
\begin{definition}
	Given ${\bf m} = (m_0,m_1,\dots,m_{\ell})$, we say that $\mathcal T$ is of type $(\alpha;{\bf m})$
	if the  homology class of $\mathcal T$ is $\alpha$, its input multiplicity is $m_0$ 
	and output multiplicities are $\{m_1,\dots,m_{\ell}\}$.
	Define
	$
	\mathcal M(\mathcal D\subset X;\alpha;{\bf m})
	$
	to be the disjoint union  
	\[
	  \coprod_{\mathcal T}{\mathcal M}^0(\mathcal D;\mathcal T)
	\] 
	where  $\mathcal T$ runs over all the decorated rooted trees of type 
	$(\alpha;{\bf m})$.
\end{definition}
\begin{prop}\label{prop322}
There exists a topology on 
$\mathcal M(\mathcal D\subset X;\alpha;{\bf m})$ 
which is compact and metrizable.
The induced topology on each subspace 
${\mathcal M}^0(\mathcal D;\mathcal T)$ coincides  
with the one defined above.
\end{prop}
The topology in Proposition \ref{prop322}
is called the {\it RGW topology}.
This proposition  is proved in Section \ref{sec:topology}.


\begin{remark}
        We can also take the following fiber product.
        \begin{equation}\label{formula317}
        \prod_{v \in  C^{\rm ins}_0(T)}{\mathcal M}^{0}(\mathcal D;\mathcal T;v)\,\,
        {}_{{\rm EV}}\times_{\star} 
        \prod_{e \in C^{\rm ins}_1(T)} \Delta
        \end{equation}
        instead of (\ref{form31666}), where $\widetilde{\mathcal M}^{0}(D;\mathcal T;v)$
        is replaced by ${\mathcal M}^{0}(D;\mathcal T;v)$, 
        which is by definition $\widetilde{\mathcal M}^{0}(D;\mathcal T;v)/\bbC_{*}$.
        Let $h = \#C_0^{\rm ins}(T) - \vert \lambda\vert$.
        Then there exists an action of the group 
        $$
        \prod_{i=1}^{\vert \ell\vert}\frac{\bbC_{*}^{h_i}}{\bbC_{*}} \cong \bbC_{*}^h
        $$
        on $\widehat{\mathcal M}^0(\mathcal D;\mathcal T)$ with $h_i = \#\{v \mid \lambda(v) = i\}$ such that
        $(\ref{formula317})$ is the quotient space.
        This $ \bbC_{*}^h$ action and the space \eqref{formula317} are not used in this paper. 
        In fact, the disjoint union of the spaces \eqref{formula317}
        for various $\mathcal T$ with the 
        natural quotient topology does not 
        carry a Kuranishi structure, because this disjoint union is not Hausdorff. 
        We can, however, shrink it by a finite-to-one map 
        and obtain a Hausdorff space. 
        The {\it log compactification} by \cite{T} seems to be related to this space. For the 
        proof of some of the conjectures in \cite[Section 6]{part3:FH} %
         (but not for the proof of Theorem \ref{mainthm}), 
        it seems necessary to construct a Kuranishi structure on 
        $\mathcal M(\mathcal D\subset X;\alpha;{\bf m})$ and perturbations 
        in a way that are invariant under these strata-wise $ \bbC_{*}^h$ actions.
        
\end{remark}

\begin{example}
	Consider a decorated rooted tree $\mathcal T$ of type $(\alpha;{\bf m})$ such that $\mathcal T$ 
	has exactly one inside vertex $v$, $\alpha(v) = \alpha$, and $\lambda(v) = 1$. There exists a unique such decorated 	rooted tree.
	(The multiplicities of outside edges are determined by 
	${\bf m}$ and there is no inside edge. The balancing 
	condition (7) in Definition \ref{defn315} is a consequence of 
	(\ref{form3535}).)
	We call this decorated rooted tree the {\it minimal tree} of type $(\alpha;{\bf m})$
	and denote it by $\mathcal T^0_{\alpha;{\bf m}}$.
	It is easy to see from the definition that
	\begin{equation}\label{identify30}
		{\mathcal M}^0(\mathcal D, \mathcal T^0_{\alpha;{\bf m}})= \mathcal M^0(\mathcal D\subset X;\alpha;{\bf m}).
	\end{equation}
\end{example}

Next, we define the notion of level shrinking of decorated rooted trees. This notion will be useful to describe the `closure' of a stratum ${\mathcal M}^0(\mathcal D;\mathcal T)$ 
in $\mathcal M(\mathcal D\subset X;\alpha;{\bf m})$.
\begin{definition}\label{defn326levelsh}
	Let $\mathcal T = (T,\alpha,m,\lambda)$ be a decorated rooted tree as in  
	Definition \ref{defn315} and $\vert \lambda\vert$ be the number of levels of $\mathcal T$. 
	Let $1 \le i < i+1 \le \vert \lambda\vert$.
	We define the {\it  decorated rooted tree obtained by $(i,i+1)$ level 
	shrinking from $\mathcal T$} as follows.
	\par
        We first define a tree $T'$.
        We shrink each of the edges $e$ in $T$ 
        such that $(\lambda(s(e)),\lambda(t(e))) = (i,i+1)$ or  $(\lambda(s(e)),\lambda(t(e))) = (i+1,i)$,
        to a point. We thus obtain a tree $T'$
        together with a map $\pi : T \to T'$.
        This map $\pi$ is an isomorphism on the 
        outside edges. The outside edges of $T'$ are by definition 
        the images of the outside edges of $T$.
        \par
        We define $\alpha' : C^{\rm ins}_0(T') \to \Pi_2(\mathcal D;\bbZ)$ as follows:
        \begin{equation}
        \alpha'(v') = 
        \sum_{v \in C^{\rm ins}_0(T) \cap \pi^{-1}(v')} \alpha(v).
        \end{equation}
        We observe that for any edge $e'$ of  $T'$,
        the inverse image $\pi^{-1}(e'\setminus \partial e')$ is $e \setminus 
        \partial e$ for some edge $e$ of $T$.
        We define $m' : C^{\rm ins}_1(T') \to \bbZ \setminus \{0\}$ by:
        \begin{equation}
        m'(e') = 
        m(e).
        \end{equation}
        We finally define a level function $\lambda' : C^{\rm ins}_1(T') \to \{1,\dots,\vert \lambda\vert -1\}$ 
        as follows.
        Let $v' = \pi(v)$ with $v \in  C^{\rm ins}_0(T)$, $v' \in  C^{\rm ins}_0(T')$.
        \begin{equation}
        \lambda'(v') 
        = 
        \begin{cases}
        \lambda(v) & \text{if $\lambda(v) < i$}, \\
        i & \text{if $\lambda(v) = \, i\,\,\, \text{or}\,\,\, i+1$}, \\
        \lambda(v)-1 & \text{if $\lambda(v) > i+1$}.
        \end{cases}
        \end{equation}
        It is easy to see that, in the second case, the right hand side 
        is independent of the choice of $v$ and it only depends on $v'$.
        It is also easy to show that $(T';\alpha',m',\lambda')$
        has the properties of Definition \ref{defn315}.
        We say $\mathcal T'$ is obtained from $\mathcal T$ by 
        {\it level shrinking} and write $\mathcal T < \mathcal T'$ 
        if $\mathcal T'$ is obtained from $\mathcal T$ by 
        a finite number of $(i,i+1)$ level shrinkings, possibly for different choices of $i$.
        We write $\mathcal T \le \mathcal T'$ if $\mathcal T < \mathcal T'$
        or $\mathcal T = \mathcal T'$.
\end{definition}
Figure \ref{Figure12} below is obtained from Figure \ref{Figure10}   by 
$(3,4)$ level 
shrinking.
\par
\begin{figure}[h]
\centering
\includegraphics[scale=0.4]{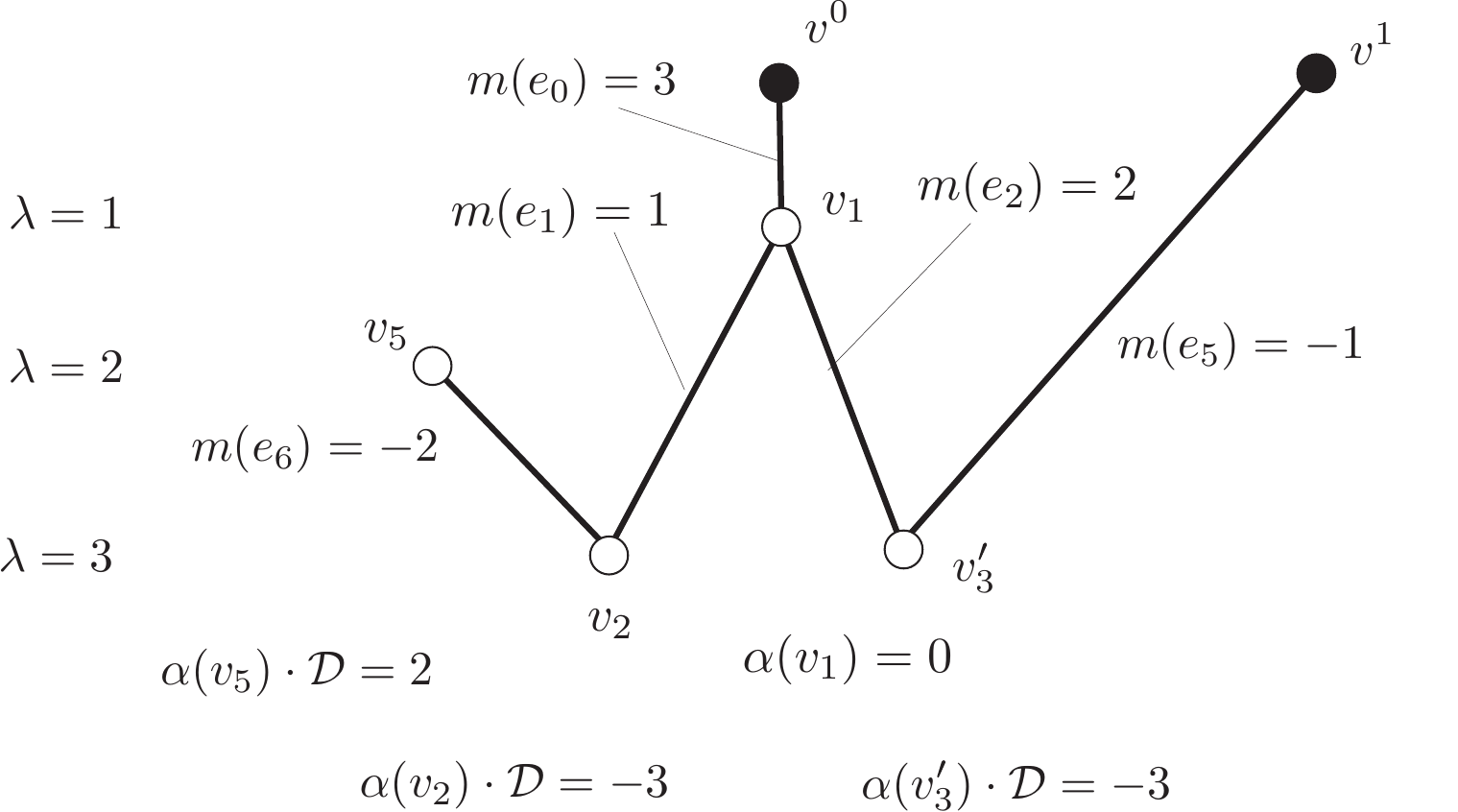}
\caption{level shrinking}
\label{Figure12}
\end{figure}

The following definition gives a notion of isotropy group for the elements of ${\mathcal M}^0(\mathcal D;\mathcal T)$. 
\begin{definition}
	Let $\frak x \in {\mathcal M}^0(\mathcal D;\mathcal T)$.
	By definition, $\frak x$ is the equivalence class of 
	an element $\widehat{\frak x} = ([{\bf x}_v]:v \in C_0^{\rm ins}(T))$ of 
	$\widehat{\mathcal M}^0(\mathcal D;\mathcal T)$
	by the action of ${\rm Aut}(\mathcal T)$
	where ${\bf x}_v \in \widetilde{\mathcal M}^0(\mathcal D;\mathcal T;v)$.
	We also fix a representative $((\Sigma_v,\vec w_v);u_v;s_v)$ for ${\bf x}_v$.
	The isotropy group $\Gamma_{\frak x}$ of $\frak x$ consists 
	of the elements $(g,\{I_v\}_{v \in C_0^{\rm ins}(T)})$ where $g\in {\rm Aut}(\mathcal T) $
	and $I_v:(\Sigma_v,\vec w_v)\to (\Sigma_{g(v)},\vec w_{g(v)})$ such that $u_{g(v)}\circ I_{v}=u_v$. 
	Moreover, if $1\leq i \leq |\lambda|$ with $|\lambda|$ being the number of the levels of $\mathcal T$,
	then there is a constant number $a_i\in \bbC_{*}$ such that for any $v$ with $\lambda(v)=i$ we have:
	\[
	  s_{g(v)}\circ I_{v}=a_i\cdot s_v.
	\]
	Projection of $(g,\{I_v\}_{v \in C_0^{\rm ins}(T)})$ to $g$ induces a map to ${\rm Aut}(\mathcal T)$ 
	and the kernel of this map is equal to $ \prod_{v \in C_0^{\rm ins}(T)} {\rm Aut}({\bf x}_v)$.
	In particular, if we define:
	\[
	  {\rm Aut}(\mathcal T;\frak x) =\{ g \in {\rm Aut}(\mathcal T) \mid g \widehat{\frak x} = \widehat{\frak x} \}.
	\]
	then we have the short exact sequence
	\begin{equation}
		1 \to  {\rm Aut}(\widehat{\frak x})\to \Gamma_{\frak x}\to {\rm Aut}(\mathcal T;\frak x) \to 1.
	\end{equation}
\end{definition}

\begin{remark}
	In \cite{part2:kura}, we construct a Kuranishi structures on $\mathcal M(\mathcal D\subset X;\alpha;{\bf m})$.
	Each element of a Kuranishi structure has an isotropy group by definition and 
	we may assume that the Kuranishi structure on $\mathcal M(\mathcal D\subset X;\alpha;{\bf m})$ 
	is chosen such that the isotropy group of $\frak x \in {\mathcal M}^0(\mathcal D;\mathcal T)
	\subset \mathcal M(\mathcal D\subset X;\alpha;{\bf m})$ is 
	$ \Gamma_{\frak x}$.
\end{remark}


We finally explain an alternative way to specify the level function.
This alternative approach shall be useful in \cite{part2:kura,part3:FH}. See \cite[Lemma 4.3]{T} for a related notion.
\begin{definition}\label{quasi-ord-def}
A binary relation $\le$ on a set $A$ defines a {\it quasi partial order}
if the following properties hold:
\begin{enumerate}
\item
$a \le b$ and $b \le c$ imply $a \le c$.
\item
$a \le a$.
\end{enumerate}
A {\it quasi order} on $A$ is a quasi partial order such that for any two elements $a,b\in A$, at least one of the relations $a \le b$ or $b \le a$ holds. For a quasi partial order, we write $a < b$ if $a \le b$ but not $b \le a$.
\par
Let $\le$, $\le'$ be two quasi partial orders on $A$. 
We say that $\le'$ is {\it finer} than $\le$ if the following holds.
\begin{enumerate}
\item[(*)]
If $a < b$ then $a <' b$.
\end{enumerate}
The similar notion is defined for quasi orders in an obvious way.
\end{definition}

Let $\mathcal T_0 = (T,\alpha,m)$ be an object, which has the properties of Definition \ref{defn315} except (6) and (9).
In particular, $\mathcal T_0$ is equipped with the decomposition of the 
set of vertices as in part (1) of Definition \ref{defn315},
the choice of the root and the enumeration of outside vertices as in part
(3) of Definition \ref{defn315}.
\par
We then obtain a  
quasi partial order $\le_0$ on $C^{\rm ins}_0(T)$
as follows:
\begin{enumerate}
\item
Suppose $v_1 \ne v_2$. We write $v_1 \le_{00} v_2$ if and only if there exists an 
edge of $e$ such that $\partial e =\{v_1, v_2\}$
and one of the following holds.
\begin{enumerate}
\item $s(e) = v_1$ and $m(e) > 0$.
\item $s(e) = v_2$ and $m(e) < 0$.
\end{enumerate}
\item
	Suppose $v \ne v'$.We write $v \le_{0} v'$ if and only if there exists $v_1,\dots,v_n$ such that 
	$v_1 = v$, $v_n = v'$ and $v_i  \le_{00} v_{i+1}$.
\item
	We also require $v \le_0 v$.
\end{enumerate}
\par
It is easy to see that $\le_0$ defines a 
quasi partial order on $C_0(T)$.
\begin{lemma}\label{lemma333333}
	Let  $\mathcal T_0$ be as above.
	There is a one to one correspondence between the following two objects:
	\begin{enumerate}
		\item The map $\lambda : C^{\rm ins}_0(T) \to \bbZ_{+}$ satisfying parts {\rm (6)} and {\rm (9)} of Definition \ref{defn315}.
		\item A quasi order on $C^{\rm ins}_0(T)$ which is finer than $\le_0$.
	\end{enumerate}
	In particular, $\mathcal T_0$ together with a quasi-order $\le$ finer than $\le_0$ 
	determines a decorated rooted tree.
\end{lemma}
\begin{proof}
Suppose $\lambda$ is given.
We define $\le$ by $v \le v'$ if and only if $\lambda(v) \le \lambda(v')$.
Definition \ref{defn315} (9) implies that $\le$ is finer than $\le_0$.
\par
Suppose we are given a quasi order 
$\le$ on $C^{\rm ins}_0(T)$ which is finer than $\le_0$.
We define a relation $\sim$ on $C^{\rm ins}_0(T)$
by $v \sim v'$ if and only if $v \le v'$, $v' \le v$.
It is easy to see that $\sim$ is an equivalence relation.
$\le$ induces an order on 
$C^{\rm ins}_0(T)/\sim$. Then $C^{\rm ins}_0(T)/\sim$, as an 
ordered set,
is isomorphic to $(\{1,\dots,\vert \lambda\vert\},\le)$.
We thus obtain $\lambda : C^{\rm ins}_0(T) \to \{1,\dots,\vert \lambda\vert\}$.
Definition \ref{defn315} (9) follows from the assumption that 
$\le$ is finer than $\le_0$.
\end{proof}

\subsection{RGW Compactification: a Single Disk or Strip Component}
\label{subsec:singledisk}

Next, we describe a part of the construction of the RGW compactification 
of the moduli space of pseudo-holomorphic disks and strips 
in $X \setminus \mathcal D$ where there is only one disk (or strip) 
component involved.
The story of strips is very similar to the case of disks.
So we discuss the case of moduli space of pseudo-holomorphic disks in detail 
and then we explain how the case of strips should be modified.
\par
Let $L \subset X \setminus \mathcal D$ be a compact Lagrangian submanifold
and $\beta \in \Pi_2(X,L;\bbZ)$, $\alpha \in \Pi_2(X;\bbZ)$.
Let ${\bf m} = (m_1,\cdots,m_{\ell})$ be an $\ell$-tuple of positive integers and $k\ge 0$.

\begin{definition}\label{defn334444}
We denote by $\mathcal M_{k+1}^{\rm reg, d}(\beta;{\bf m})$
the set of all the isomorphism classes of $((\Sigma,\vec z,\vec w),u)$ with the following 
properties.
\begin{enumerate}
\item $\Sigma$ is the union of a disk $D^2$ and trees of spheres 
rooted on ${\rm Int}(D^2)$.
(We require that any singularity of $\Sigma$ is a nodal singularity.)
\item $\vec z = (z_0,\dots,z_k)$ and $z_i \in \partial\Sigma$.
The points $z_0,\dots,z_k$ are distinct and respect the 
counter clockwise cyclic order on 
$S^1 = \partial\Sigma$.
\item $\vec w = (w_1,\dots,w_{\ell})$ and $w_i \in {\rm Int}(\Sigma)$.
The points $w_1,\dots,w_{\ell}$ are distinct and away from the nodes of $\Sigma$.
\item
$u : \Sigma \to X$ is a pseudo-holomorphic map, $u(\partial\Sigma) \subset L$,
and the homology class of $u$ is $\beta$.
\item $u(w_i) \in \mathcal D$. Moreover,
$u^{-1}(\mathcal D) = \{w_1,\dots,w_{\ell}\}$.
\item The order of tangency of $u$ to $\mathcal D$ at $w_i$ is $m_i$.
\item
$((\Sigma,\vec z,\vec w),u)$ is stable in the sense of stable maps.
(See, for example, \cite[Subsection 2.1]{fooobook2}.)
\end{enumerate}
We say $((\Sigma,\vec z,\vec w),u)$ is isomorphic to 
$((\Sigma',\vec z^{\,\prime},\vec w^{\,\prime}),u')$ if there exists a biholomorphic map 
$v : \Sigma \to \Sigma'$ such that $v(z_i) = z'_i$, $v(w_i) = w'_i$
and $u'\circ v = u$. We define evaluation maps
\begin{equation}\label{eval330d}
	\qquad\qquad{\rm ev}_i : 
	\mathcal M_{k+1}^{\rm reg, d}(\beta;{\bf m})
	\to \mathcal D,
	\qquad i = 1,\dots,\ell
\end{equation}
at $w$'s as follows:
\begin{equation}\label{eval331s}
{\rm ev}_i((\Sigma,\vec z,\vec w),u) = u(w_i).
\end{equation}
\end{definition}
Note that we include the case that ${\ell} =0$. In this case, we require that $u(\Sigma) \cap \mathcal D = \emptyset$ and write $\mathcal M_{k+1}^{\rm reg, d}(\beta;\emptyset)$ for the corresponding moduli space.

\begin{figure}[h]
\centering
\includegraphics[scale=0.4]{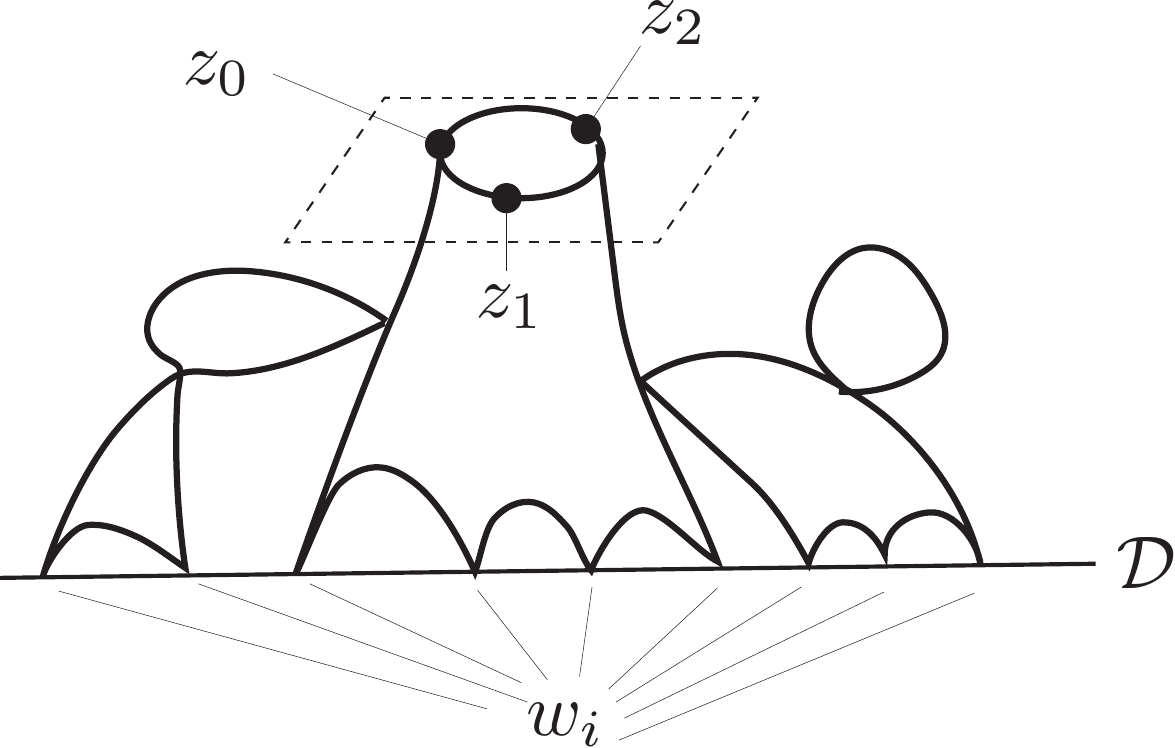}
\caption{An element of $\mathcal M_{k+1}^{\rm reg, d}(\beta;{\bf m})$.}
\label{Figurestrip1}
\end{figure}
\begin{definition}\label{defn333555}
        We denote by $\mathcal M^{\rm reg, s}(\alpha;{\bf m})$ 
        the set of the isomorphism classes of objects $((\Sigma,\vec w),u)$ with the following 
        properties.
        \begin{enumerate}
                \item $\Sigma$ is a connected nodal curve of genus $0$.
                \item $\vec w = (w_1,\dots,w_{\ell})$, $w_i \in \Sigma$. The points $w_1,\dots,w_{\ell}$ are  distinct and away from nodes.
                \item $u : \Sigma \to X$ is a holomorphic map. The homology class of $u$ is $\alpha$.
                \item $u(w_i) \in \mathcal D$. Moreover, $u^{-1}(\mathcal D) = \{w_1,\dots,w_{\ell}\}$.
                \item The order of tangency of $u$ to $\mathcal D$ at $w_i$ is $m_i$.
                \item $((\Sigma,\vec w),u)$ is stable in the sense of stable maps.
        \end{enumerate}
        The definition of isomorphisms of such objects is similar to Definition \ref{defn334444}.
        We also define the following evaluation maps in the same way:
        \begin{equation}\label{form331}
        {\rm ev}_i : 
        \mathcal M^{\rm reg, s}(\beta;{\bf m})
        \to \mathcal D
        \end{equation}
\end{definition}
To obtain our RGW compactification, we consider fiber products
of the above defined moduli spaces $\mathcal M_{k+1}^{\rm reg, d}(\beta;{\bf m})$, 
$\mathcal M^{\rm reg, s}(\alpha;{\bf m})$, and the moduli spaces of the form ${\mathcal M}^0(\mathcal D;\mathcal T)$, which we defined in the previous subsection.
The combinatorial objects to describe these fiber products are given in the following definition:
\begin{definition}\label{defb333366}
Suppose $L$ is a compact Lagrangian submanifold in the complement of the 
smooth divisor $\mathcal D$ in $X$.
A {\it disk-divisor describing tree}, or a {\it DD-tree} for short, is an 
object $\mathcal S = (S,c,m,\alpha,\mathcal T,\le)$
with the following properties:
\begin{enumerate}
\item $S$ is a tree.
The set of the vertices and the edges of $S$ are respectively denoted by $C_0(S)$ and $C_1(S)$.
\item
$c : C_0(S) \to \{{\rm d},{\rm s},{\rm D}\}$ assigns one of the symbols d,s or D to 
each vertex $\rm v$. We call $c({\rm v})$ the {\it color} of $\rm v$.\footnote{We use roman letters ${\rm v}$ and ${\rm e}$ to denote the vertices and the edges of $S$. In the case of decorated rooted trees, we use italic letters $v$ and $e$.}
\item
There exists exactly one vertex whose color is d. 
Here d stands for disk.
We call this vertex the {\it root} of $S$.
There is also a number $k$ associated to the root\footnote{$k+1$ is the number 
of boundary marked points.}.  
(The root will correspond to a moduli space of the form $\mathcal M_{k+1}^{\rm reg, d}(\beta;{\bf m})$.)

\item
	The root or a vertex with color s is 
	joined\footnote{We say a vertex $\rm v$ is joined to a vertex $\rm v'$ if 
	and only if there exists an edge which contains both  $\rm v$ and $\rm v'$.} 
	only to vertices with color D.  
	(Here s and D stand for sphere and divisor, respectively. 
	A vertex with color s 
	will correspond to a moduli space $\mathcal M^{\rm reg, s}(\alpha;{\bf m})$.
	A vertex with color D will correspond to a moduli space 
	${\mathcal M}^0(\mathcal D;\mathcal T({\rm v}))$.)
\item 
There is no edge joining two vertices of color D.
\item
$m : C_1(S) \to \bbZ_{+}$ assigns a positive number to each edge of $S$.
We call it the {\it multiplicity function}: The number $m({\rm e})$ is called 
the {\it multiplicity of the edge} ${\rm e}$.
\item
If $c({\rm v})$ is equal to d,  s, or D, then
$\alpha({\rm v})$ is respectively an element of  $\Pi_2(X,L;\bbZ)$,
$\Pi_2(X;\bbZ)$ or
$\Pi_2(\mathcal D;\bbZ)$.
We call $\alpha({\rm v})$ the {\it homology class of $\rm v$.}
\item
For each vertex $\rm v$ with color D or s,  there exists a unique edge  ${\rm e}_0({\rm v})$
in $S \setminus \{\rm v\}$
which lies in the same connected component of $S \setminus \{{\rm v}\}$
as the root.
This edge is called the {\it first edge} of ${\rm v}$.
\item
Let $\rm v$ be a vertex with color $\rm D$ 
with the first edge ${\rm e}_0({\rm v})$. 
Let the other edges incident to $v$ be denoted by 
${\rm e}_1({\rm v}),\dots,{\rm e}_{\ell({\rm v})}({\rm v})$.
(Here $\ell({\rm v})+1$ is the valency of ${\rm v}$.)
Then $\mathcal T$ assigns to ${\rm v}$ a decorated
rooted tree $\mathcal T({\rm v})$ such that
\begin{enumerate}
\item The homology class of $\mathcal T({\rm v})$ is $\alpha({\rm v})$.
\item
Its input multiplicity is $m({\rm e}_0({\rm v}))$.
\item
Its output multiplicities are $m({\rm e}_i({\rm v}))$, 
$i = 1,\dots,\ell({\rm v})$. 
\end{enumerate}
\item
$\le$ is a quasi  order on
\begin{equation}
C_0^{\rm ins}(\mathcal S) = \bigcup_{{\rm v} \in C_0(S)} C_0^{\rm ins}(\mathcal T({\rm v})),
\end{equation}
We require that the restriction of $\le$ to 
$C_0^{\rm ins}(\mathcal T({\rm v}))$
coincides with $\le_{\rm v}$,
which is induced by the level function of $\mathcal T({\rm v})$ 
using Lemma \ref{lemma333333}. We call any element of $C_0^{\rm ins}(\mathcal S)$
an {\it inside edge of $\mathcal S$}.
\end{enumerate}
If we want to specify the Lagrangian submanifold $L$, then we say 
$\mathcal S$ is a DD-tree for $L$.
\end{definition}
\begin{definition}
The homology class of a DD-tree $\mathcal S = (S,c,m,\alpha,\mathcal T,\le)$
is defined as:
\begin{equation}
\beta(\mathcal S) 
=
\sum_{{\rm v} \in C_0(S)} \alpha({\rm v}) \in \Pi_2(X,L;\bbZ).
\end{equation}
In the above expression, 
if ${\rm v}$ has color s or D we use the homomorphisms
$\Pi_2(X;\bbZ) \to \Pi_2(X,L;\bbZ)$ and
$\Pi_2(\mathcal D;\bbZ) \to \Pi_2(X,L;\bbZ)$
to define the right hand side. We say $\mathcal S$ is a DD-tree of type $(\beta(\mathcal S),k)$.
\end{definition}
\begin{example}\label{exa28}
In Figure \ref{Figure14} a DD tree $S$ is sketched.
The symbols d, s, D denote the vertices of  colors d, s, D.
The number written near each edge is the multiplicity of that edge.

\begin{figure}[h]
\centering
\includegraphics[scale=0.4]{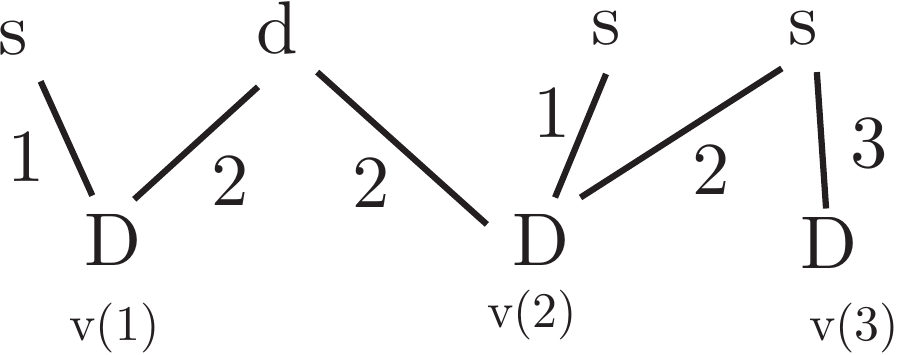}
\caption{An example of graph $S$}
\label{Figure14}
\end{figure}
\begin{figure}[h]
\centering
\includegraphics[scale=0.4]{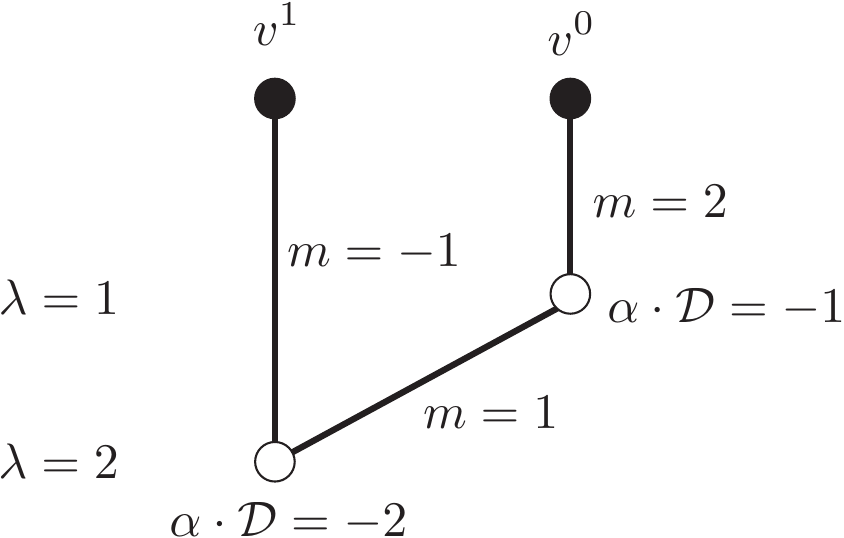}
\caption{$\mathcal T({\rm v}(1))$}
\label{Figure15}
\end{figure}
\begin{figure}[h]
\centering
\includegraphics[scale=0.4]{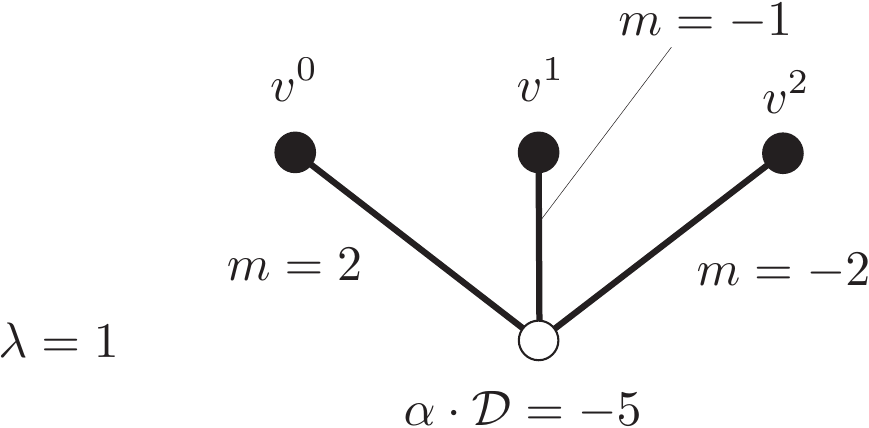}
\caption{$\mathcal T({\rm v}(2))$}
\label{Figure16}
\end{figure}
\begin{figure}[h]
\centering
\includegraphics[scale=0.4]{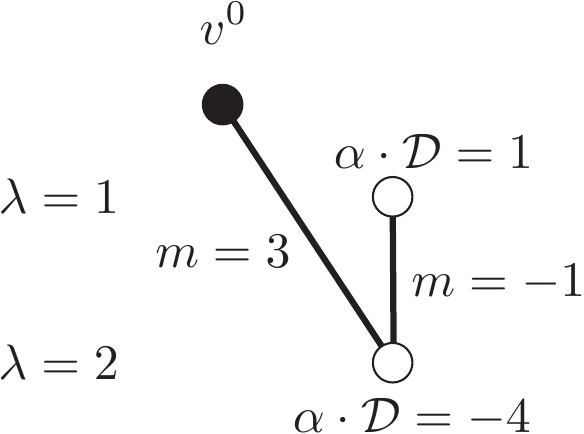}
\caption{$\mathcal T({\rm v}(3))$}
\label{Figure17}
\end{figure}
\end{example}
We associate the decorated rooted trees $\mathcal T({\rm v}(1))$, 
$\mathcal T({\rm v}(2))$, $\mathcal T({\rm v}(3))$ given in Figures \ref{Figure15}, \ref{Figure16}, \ref{Figure17}
to the vertices ${\rm v}(1)$, ${\rm v}(2)$, ${\rm v}(3)$, respectively.
We put these trees on the position of the corresponding 
vertices of $S$ and then identify the outside edges of $\mathcal T({\rm v}(i))$
with the corresponding edges of $S$. We thus obtain 
the tree $\hat S$ in Figure \ref{Figure18}.
\begin{figure}[h]
\centering
\includegraphics[scale=0.5]{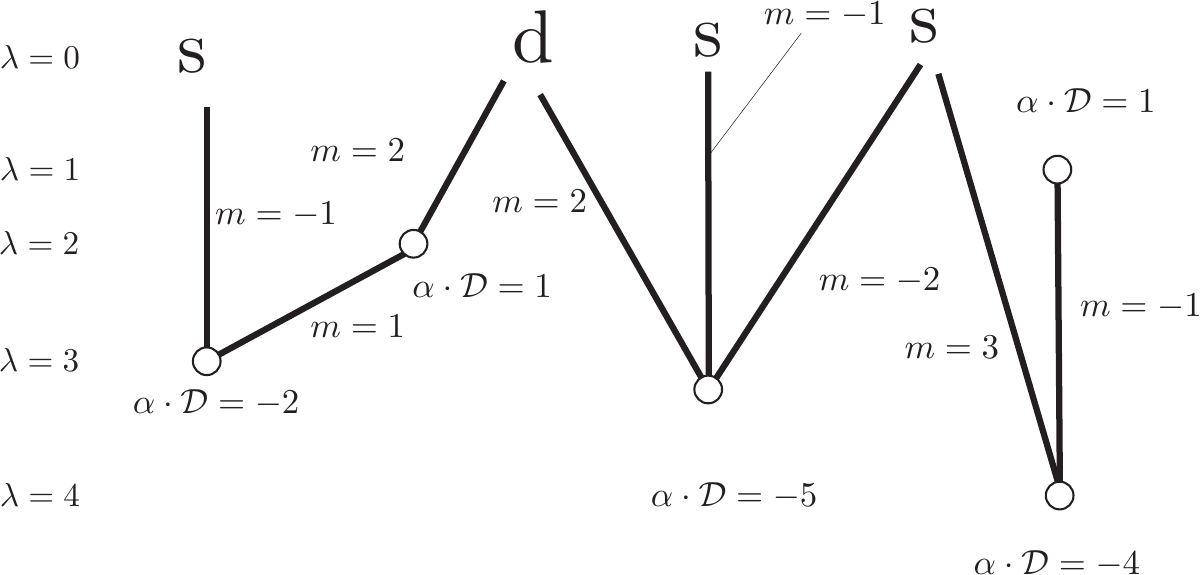}
\caption{Detailed tree $\widehat{S}$}
\label{Figure18}
\end{figure}
\begin{figure}[h]
\centering
\includegraphics[scale=0.3]{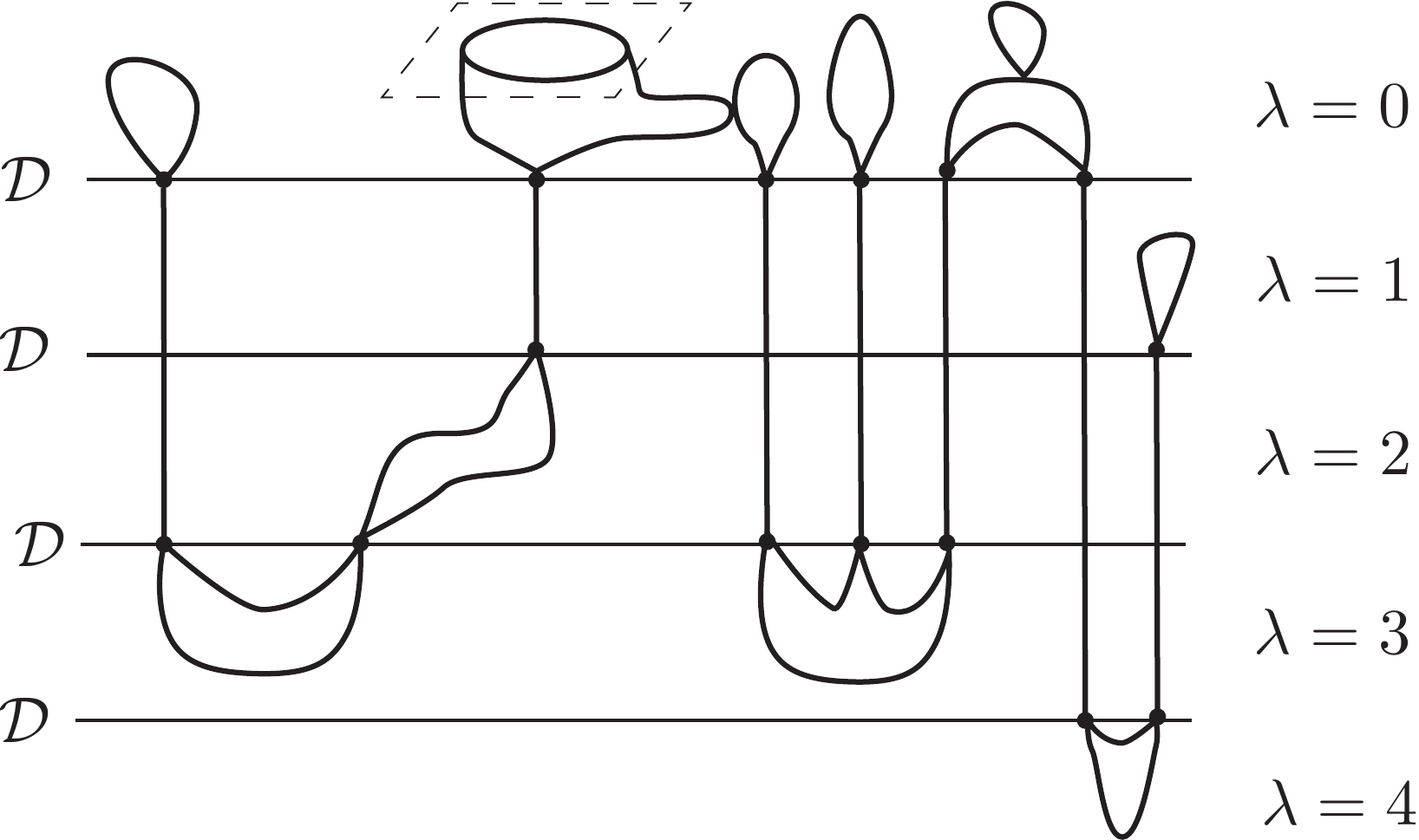}
\caption{object corresponding to $\widehat{S}$ in Figure \ref{Figure18}.}
\label{Figure19}
\end{figure}
We call $\widehat S$ the detailed tree associated to $\mathcal S$. The level function $\lambda$ in Figure \ref{Figure18} is defined using the quasi order on $C_0^{\rm ins}(\mathcal S)$.
Figure \ref{Figure18} describes a configuration of pseudo-holomorphic curves as in Figure \ref{Figure19}. (See (\ref{form339}).)
We summarize the properties of the detailed 
tree associated to $\mathcal S$ in the following lemma.

\begin{lemma} \label{detailed-tree-properties}
We can associate $(\hat S,c,\alpha,m,\lambda)$
to $\mathcal S$ which satisfies the following properties 
in addition to {\rm (1)-(4)} and {\rm (7), (8)} in Definition \ref{defb333366}.
\begin{enumerate}
\item[(i)]
$m : C_1(\hat S) \to \bbZ \setminus \{0\}$ assigns a nonzero integer to each edges of $S$.
We call it the {\it multiplicity function}. The number $m(e)$ is called 
the {\it multiplicity of the edge} $e$.
\item[(ii)]
The balancing condition \eqref{bananching} is satisfied at 
the vertices with color ${\rm D}$.
\item[(iii)]
The stability condition in part {\rm (8)} of Definition \ref{defn315}  is satisfied 
at vertices with color ${\rm D}$.
\item[(iv)]
We require 
$$
\alpha(v) \cdot \mathcal D = \sum_{v \in e} m(e)
$$
for vertices with color $\rm d$.
We also require 
$$
\alpha(v) \cdot \mathcal D = -m(e_0(v)) + \sum_{v \in e} m(e)
$$
for vertices with color $\rm s$. 
Here $e_0(v)$ is the edge containing $v$ such that $e_0(v)$ 
is contained in the same connected component of 
$S \setminus \{v\}$ as the root.
\item[(v)]
$\lambda : C_0(\hat S) \to \bbZ_{\ge 0}$ is a map such that:
\begin{enumerate}
\item
If $v$ has color ${\rm d}$ or ${\rm s}$ then $\lambda(v) = 0$.
\item
If $v$ has color ${\rm D}$ then $\lambda(v) > 0$.
\item
Part {\rm (9)} of Definition \ref{defn315} holds.
\item
The image of $\lambda$ is $\{0,\dots,\vert \lambda\vert\}$ for some $\vert \lambda\vert > 0$.
\end{enumerate}
\end{enumerate}
\end{lemma}
\begin{proof}
Suppose $\hat S$ is the detailed tree associated to the DD tree $\mathcal S$.
A level function 
can be defined 
on the set of the inside vertices of $\hat S$ 
by the quasi order $\le$ as in 
Lemma \ref{lemma333333}.
For vertices with color s or d, we define its level to be $0$.
The rest of the proof is straightforward.
\end{proof}
\begin{remark}
        Let $e$ be an edge which joins a level $0$ vertex $v$ and a vertex $v'$ of positive level.
        Part (v) of Lemma \ref{detailed-tree-properties} implies that $m(e) > 0$, if $c(v) = {\rm d}$. On the other hand, if $c(v) = {\rm s}$, then $m(e) > 0$ if and only if
        $e\neq e_0(v)$. 
\end{remark}
\begin{remark}\label{rem338}
The restriction of $\lambda$ to 
$C_0^{\rm ins}(\mathcal T({\rm v}))$ may not coincide
with the level function $\lambda_{\rm v}$ of $\mathcal T({\rm v})$ as it can be seen in Example \ref{exa28}.
\end{remark}
\begin{definition}\label{DDtreedetail}
	We call $(\hat S,c,\alpha,m,\lambda)$ the {\it detailed DD-tree} associated to the DD-tree $\mathcal S$. 
\end{definition}
Let $\mathcal S$ be a DD-tree and 
$(\hat S,c,\alpha,m,\lambda)$ be 
its associated detailed DD-tree.
For each vertex $v$ of $\hat S$, we associate a moduli space 
$\widetilde{\mathcal M}^0(\mathcal S;v)$ as follows.
Let $c(v) = {\rm d}$ and $e^v_1,\dots,e^v_{\ell(v)}$ be the 
edges incident to $v$. We then define:
\begin{equation}\label{form335}
\widetilde{\mathcal M}^0(\mathcal S;v)
=
\mathcal M_{k+1}^{\rm reg, d}(\alpha(v);{\bf m}^v),
\end{equation}
where ${\bf m}^v = (m(e^v_1),\cdots,m(e^v_{\ell(v)}))$ and 
$k$ is the non-negative integer associated to the root. 

Let $c(v) = {\rm s}$,  
$e_0^v$ be the first edge of  $v$, 
and $e^v_1,\dots,e^v_{\ell}$ be the other edges of $v$. We define:
\begin{equation}\label{form336}
\widetilde{\mathcal M}^0(\mathcal S;v)
=
\mathcal M^{\rm reg, s}(\alpha(v);{\bf m}^v),
\end{equation}
where ${\bf m}^v = (-m(e^v_0),m(e^v_1),\cdots,m(e^v_{\ell(v)}))$.
Finally, suppose $c(v) = {\rm D}$ with the first edge $e^v_0$, and the remaining edges $e^v_1,\dots,e^v_{\ell(v)}$. We define:
\begin{equation}\label{fpr315315rev}
	\widetilde{\mathcal M}^{0}(\mathcal S;v)= \widetilde{\mathcal M}^{0}(\mathcal D\subset X;\alpha(v);{\bf m}^v),
\end{equation}
where ${\bf m}^v = (-m(e^v_0),m(e^v_1),\cdots,m(e^v_{\ell(v)}))$.
This is similar to (\ref{fpr315315}).
\par
We can also define a map:
$$
{\rm EV} : \prod_{v \in C_0(\hat S)}\widetilde{\mathcal M}^{0}(\mathcal S;v)
\to \prod_{e \in C_1(\hat S)}(\mathcal D \times \mathcal D)
$$
similar to the map (\ref{formnew3188}) as follows:
Let $e \in  C_1(\hat S)$.
Let $\vec{\bf x}= ({\bf x}_v;v \in  C_0(\hat S))$
be an element of the domain of ${\rm EV}$. 
Then ${\rm EV}(\vec{\bf x})_e$, the component of  ${\rm EV}(\vec{\bf x})$
corresponding to the edge $e$, is defined as follows:
\begin{equation}\label{form338}
{\rm EV}(\vec{\bf x})_e
=
({\rm ev}_0({\bf x}_{t(e)}),{\rm ev}_i({\bf x}_{s(e)})) \in \mathcal D \times \mathcal D.
\end{equation}
Here ${\rm ev}_0$ and ${\rm ev}_i$ are given in 
Definition \ref{evaluationmap1}, (\ref{eval330d})  or (\ref{form331}), and $i$ is chosen such that $e$ is the $i$-th edge of $s(e)$.
Note that $t(e)$ has color either ${\rm s}$ or ${\rm D}$. Therefore,  the first edge of $t(e)$ is defined and is $e$.
\par
We define $\widetilde{\mathcal M}^{0}(\mathcal S)$ to be the following fiber product:
\begin{equation}\label{form339}
	\widetilde{\mathcal M}^{0}(\mathcal S)=\prod_{v \in C_0(\hat S)}\widetilde{\mathcal M}^{0}(\mathcal S;v)
	\,_{\rm EV}\times_{\star}\prod_{e \in C_1(\hat S)}\Delta
\end{equation}
where $
\star$ denotes the diagonal inclusion of 
$\prod_{e \in C_1(\hat S)}\Delta$ into $\prod_{e \in C_1(\hat S)}\mathcal D\times \mathcal D$ with $\Delta$ being the diagonal in $\mathcal D\times \mathcal D$.
We also define a $\bbC_{*}^{\vert \lambda\vert}$ action on $\widetilde{\mathcal M}^{0}(\mathcal S)$ 
by
\begin{equation}\label{form33940}
\vec{\rho} \cdot \vec{{\bf x}} = (\rho_{\lambda(v)}{\bf x}_v),
\end{equation}
where 
$\vec{\rho} = (\rho_1,\dots,\rho_{\vert \lambda\vert}) \in \bbC_{*}^{\vert \lambda\vert}$
and 
$\vec{{\bf x}} = ({\bf x}_v) 
\in \prod_{v \in  C^{\rm int}_0(\hat S)}\widetilde{\mathcal M}(\mathcal S;v)$.
In \eqref{form33940}, if $\lambda(v) = 0$, then $\rho_{\lambda(v)}{\bf x}_v$ is defined to be ${\bf x}_v$.
\par
Next, we define:
\begin{equation}\label{form3440}
\widehat{\mathcal M}^{0}(\mathcal S)
=
\widetilde{\mathcal M}^{0}(\mathcal S)/\bbC_{*}^{\vert \lambda\vert}.
\end{equation}
Let ${\rm Aut}(\mathcal S)$ be the group of automorphisms of 
the tree $\hat S$ which preserves $c$, $\alpha$, $m$, 
$\lambda$. The group ${\rm Aut}(\mathcal S)$ acts on 
$\widehat{\mathcal M}^{0}(\mathcal S)$. We finally define
\begin{equation}\label{form342new}
{\mathcal M}^{0}(\mathcal S) =
\widehat{\mathcal M}^{0}(\mathcal S)/{\rm Aut}(\mathcal S).
\end{equation}

\begin{definition}
For $k \in \bbZ_{\ge 0}$ and $\beta \in \Pi_2(X;L;\bbZ)$
define $\mathcal M^{0,{\rm RGW}}_{k+1}(L;\beta)$ 
to be the disjoint union of all the spaces ${\mathcal M}^{0}(\mathcal S)$, 
where $\mathcal S$ is a DD-tree with homology class $\beta$ such that 
$k$ is the nonnegative integer associated to the root.
\end{definition}

We next consider the case of holomorphic strips.
Let $L_0, L_1 \subset X \setminus \mathcal D$ be 
a pair of compact Lagrangian submanifolds.
We assume $L_0$ is transversal to $L_1$.
Let $p,q \in L_0 \cap L_1$ and $\beta \in \Pi_{2}(X;L_1,L_0;p,q)$ (See Definition \ref{defn2121}.)
and $k_0,k_1$ be nonnegative integers.
\begin{definition}\label{defn33strip}
We denote by $\mathcal M_{k_1,k_0}^{\rm reg}(L_1,L_0;p,q;\beta;{\bf m})$
the set of all  isomorphism classes of $((\Sigma,\vec z_0,\vec z_1,\vec w),u)$ with the following properties.
\begin{enumerate}
\item $\Sigma$ is a strip $\bbR \times [0,1]$ together with trees of spheres 
rooted in $\bbR \times (0,1)$.
(We require that the singularities of $\Sigma$ to be nodal singularities.)
\item $\vec z_i = (z_{i,1},\dots,z_{i,k_i})$ and $z_{i,j} \in \bbR \times \{i\}$, 
for $i=0,1$. The points
$z_{i,1},\dots,z_{i,k_i}$ are  distinct,
$z_{1,1} > \dots > z_{1,k_1}$ and $z_{0,1} < \dots < z_{0,k_0}$.

\item $\vec w = (w_1,\dots,w_{\ell})$, $w_i \in {\rm Int}(\Sigma)$.
$w_1,\dots,w_{\ell}$ are distinct points in the interior of $\Sigma$ and are away from the nodal 
singularities.
\item
$u : \Sigma \to X$ is a holomorphic map, $u(\bbR \times \{i\}) \subset L_i$,
\begin{equation}\label{form34342}
\lim_{\tau \to -\infty} u(\tau,t) = p,
\qquad
\lim_{\tau \to +\infty} u(\tau,t) = q,
\end{equation}
and
the homology class of $u$ is $\beta$.
\item 
$u^{-1}(\mathcal D) = \{w_1,\dots,w_{\ell}\}$.
\item The order of tangency of $u$ to $\mathcal D$ at $w_i$ is $m_i$.
\item
$((\Sigma,\vec z_0,\vec z_1,\vec w),u)$ is stable in the sense of stable maps.
\end{enumerate}
The definition of an isomorphism between two objects as above is similar to 
Definition \ref{defn334444}.
\end{definition}
\begin{figure}[h]
\centering
\includegraphics[scale=0.3]{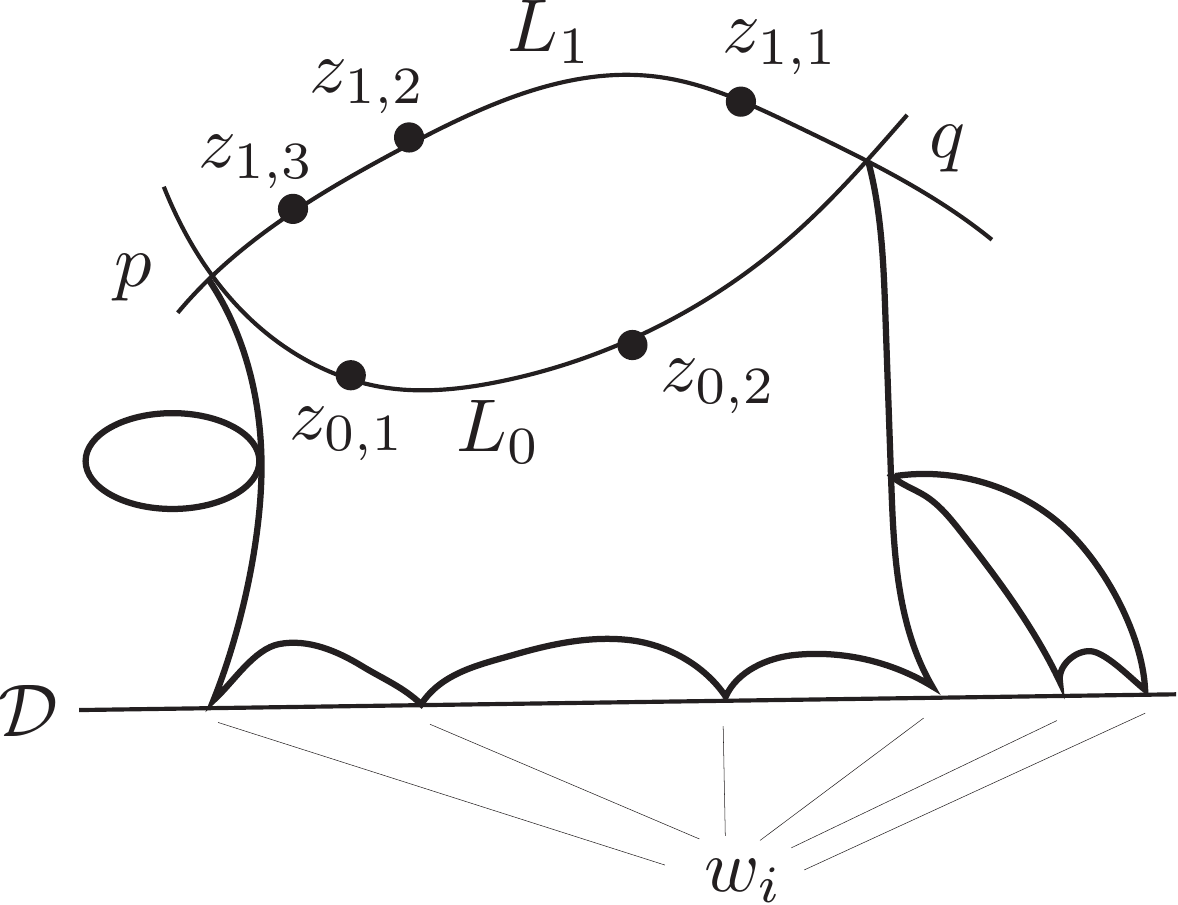}
\caption{An element of $\mathcal M_{k_1,k_0}^{\rm reg}(L_1,L_0;p,q;\beta;{\bf m})$.}
\label{Figurestrip2}
\end{figure}
\begin{definition}
We define evaluation maps:
\[
\hspace{3cm}{\rm ev}_i : \mathcal M_{k_1,k_0}^{\rm reg}(L_1,L_0;p,q;\beta;{\bf m})\to \mathcal D\hspace{1cm}i=1,\dots,\ell
\]
by:
\begin{equation}
{\rm ev}_i((\Sigma,\vec z_0,\vec z_1,\vec w),u) = u(w_i).
\end{equation}
\end{definition}
\begin{definition}
	A {\it strip-divisor describing tree}, or an {\it SD-tree} for short, is an 
	object $\mathcal S = (S,c,m,\alpha,\mathcal T,\le)$
which satisfies the same properties as in Definition \ref{defb333366} except that we replace (2), (3) and (7) by :
\begin{enumerate}
\item[$(2)'$]
$c : C_0(S) \to \{{\rm str},{\rm s},{\rm D}\}$ assigns one of the symbols str, s or D to 
each vertex $\rm v$. We call it the {\it color} of $\rm v$ and denote it by $c({\rm v})$.
\item[$(3)'$]
There exists exactly one vertex whose color is str. The label ``str'' stands for strip. 
	We call the unique vertex with color str the {\it root} of $S$.
	There are also integers $k_0$ and $k_1$ associated to the vertex $v_0$.
	(The root will correspond to a moduli space of the form $\mathcal M_{k_1,k_0}^{\rm reg}(L_1,L_0;p,q;\beta;{\bf m})$.)
\item[$(7)'$]	
If $c({\rm v})$ is equal to str, s, or D, then
$\alpha({\rm v})$ respectively belongs to $\Pi_2(X;L_1,L_0;\bbZ)$,
$\Pi_2(X;\bbZ)$ or
$\Pi_2(\mathcal D;\bbZ)$.
We call $\alpha({\rm v})$ the {\it homology class of $\rm v$.}
\end{enumerate}
Analogous to DD-trees, we can associate a detailed SD-tree to an SD-tree.
\end{definition}
The definition of the moduli space in (\ref{form339}) can be modified 
to define $\widetilde{\mathcal M}^{0}(\mathcal S)$ for an SD-tree.
If $c(v) = {\rm str}$ the moduli space $\widetilde{\mathcal M}^0(\mathcal S;v)$ is defined to be 
\begin{equation}\label{formform343343}
\widetilde{\mathcal M}^0(\mathcal S;v)
=
\mathcal M_{k_1,k_0}^{\rm reg}(L_1,L_0;p,q;\beta;{\bf m}^v)
\end{equation}
where ${\bf m}^v = (m(e^v_1),m(e^v_2),\cdots,m(e^v_{\ell(v)})$ with $e^v_1$, $\dots$, $e^v_{\ell(v)}$ being the edges incident to the root $v$.
If $c(v) = {\rm s}$ or ${\rm D}$, then $\widetilde{\mathcal M}^0(\mathcal S;v)$ is 
defined as in (\ref{form336}), (\ref{fpr315315rev}).
We define $\widetilde{\mathcal M}^{0}(\mathcal S)$, 
$\widehat{\mathcal M}^{0}(\mathcal S)$ and 
${\mathcal M}^{0}(\mathcal S)$ by \eqref{form339},
\eqref{form3440} and \eqref{form342new}, respectively.
Finally we define:
\begin{definition}
For $k_0,k_1 \in \bbZ_{\ge 0}$, $p,q \in L_0 \cap L_1$ and $\beta \in \Pi_2(X;L_1,L_0;\bbZ)$,
 the set $\mathcal M^{0,{\rm RGW}}_{k_1,k_0}(L_1,L_0;p,q;\beta)$ 
is the disjoint union of all ${\mathcal M}^{0}(\mathcal S)$, 
where $\mathcal S$ is an SD-tree 
of type $(\beta;k_0,k_1)$.
\end{definition}

In the next subsection, we define {\it compactifications} of the sets $\mathcal M^{0,{\rm RGW}}_{k+1}(L;\beta)$ and $\mathcal M^{0,{\rm RGW}}_{k_1,k_0}(L_1,L_0;p,q;\beta)$.
The definition of these compactifications uses the notion 
of level shrinking for DD-trees and SD-trees.
We discuss this notion for DD-trees. The case of SD-trees is 
completely similar.
\begin{definition}\label{01levelshrink}
Let 
$(\hat S,c,\alpha,m,\lambda)$ 
be the  detailed DD-tree  associated to a DD-tree $\mathcal S$.
Let $i \in \{0,\dots,\vert \lambda\vert\}$.
If $i > 0$, then $(i,i+1)$-level shrinking of $\mathcal S$ is 
defined as in  
Definition \ref{defn326levelsh}.
We define $(0,1)$-level shrinking of $\mathcal S$ below.
\par
Let $v \in C_0(\hat S)$ with $\lambda(v) = 1$. 
Then $c(v)$ is equal to ${\rm D}$. There are two cases.
\par
\noindent{\bf Case 1}: 
There is no vertex $\hat v$ with $\lambda(\hat v) = 0$, which is joined to $v$. 
In this case, after $(0,1)$ level shrinking, 
$v$ will have the color s and its level is equal to 0.

\noindent{\bf Case 2}:
There exists a vertex $\hat v$ with $\lambda(\hat v) = 0$ which is 
joined to $v$. 
Let $C \subset \hat S$ be the maximal connected 
subgraph of $\hat S$ which contains 
$v$ and whose vertices have level 0 or 1.
We shrink $C$ to a new vertex $v'$.
The color of $v'$ is d if the root of $\hat S$ is in $C$. Otherwise, the color of $v'$ is s.
All the edges $e$, which are joined to a vertex of $C$, but are not in $C$, will be joined to $v'$.
We also define:
$$
\alpha(v') = \sum_{\hat v \in C} \alpha(\hat v),
\qquad \lambda(v') = 0.
$$
\par
We perform this operation to all the 
vertices of level 1.
We change the level of all the vertices $v$ with $\lambda(v) > 1$ 
to $\lambda(v) -1$.
The resulting tree together with $c,\alpha,m$ and $\lambda$
is the detailed DD-tree associated to a DD-tree. We say this DD-tree is obtained from $\mathcal S$ by $(0,1)$ level shrinking.
\end{definition}
\begin{example}
In Figure \ref{Figure20}, we sketch the detailed tree for a DD-tree $\mathcal S$ 
together with its level function.
Figure \ref{Figure21} gives the detailed tree obtained 
from $\mathcal S$ by  $(0,1)$ level shrinking.
\begin{figure}[h]
\centering
\includegraphics[scale=0.3]{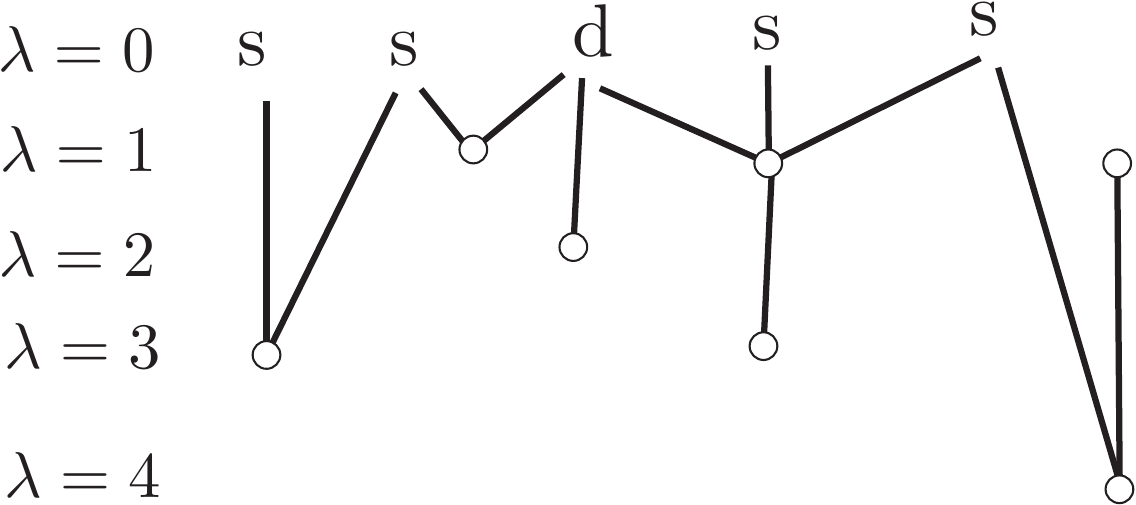}
\caption{Before (0,1) level shrinking.}
\label{Figure20}
\end{figure}
\begin{figure}[h]
\centering
\includegraphics[scale=0.3]{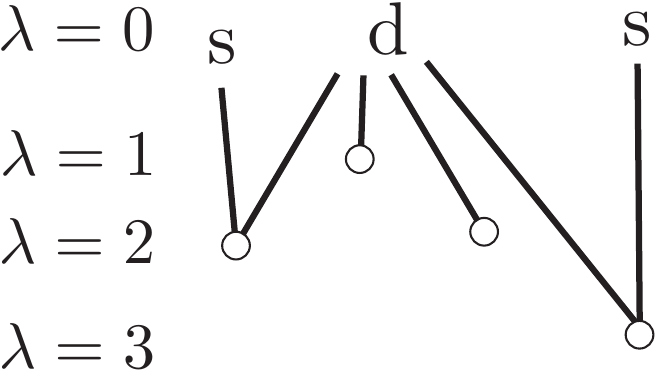}
\caption{After (0,1) level shrinking.}
\label{Figure21}
\end{figure}
\end{example}

\begin{definition}
We write $\mathcal S > \mathcal S'$ if $\mathcal S$ is obtained from 
$\mathcal S'$ by a finite number of $(i,i+1)$ level shrinkings, possibly for different choices of $i$. 
We also write $\mathcal S \ge \mathcal S'$ if $\mathcal S > \mathcal S'$ or $\mathcal S = 
\mathcal S'$.
\end{definition}
We define
\begin{equation}
{\mathcal M}(\mathcal S)
=
\bigcup_{\mathcal S' \le \mathcal S} {\mathcal M}^{0}(\mathcal S').
\end{equation}
We will define a topology on ${\mathcal M}(\mathcal S)$ in Section \ref{sec:topology}.
As an immediate consequence of the definition we have the following:
If an element of ${\mathcal M}^{0}(\mathcal S')$ 
is obtained as a limit of a sequence of elements of 
${\mathcal M}^{0}(\mathcal S)$, then $\mathcal S \ge \mathcal S'$.

\subsection{RGW Compactification in the General Case}
\label{subsec:RGWgeneral}

We now describe the compactifications 
$\mathcal M^{{\rm RGW}}_{k_1,k_0}(L_1,L_0;p,q;\beta)$, 
$\mathcal M^{{\rm RGW}}_{k+1}(L;\beta)$.
These compactifications are obtained 
by taking the union of fiber products of various
spaces of the forms
$\mathcal M^{0,{\rm RGW}}_{k_1,k_0}(L_1,L_0;p,q;\beta)$, 
$\mathcal M^{0,{\rm RGW}}_{k+1}(L;\beta)$. 
Those fiber products are defined using evaluation maps at the boundary 
marked points.
In the following definition we describe combinatorial objects which keep track of 
these fiber products:
\begin{definition}\label{DDdesroottree}
A {\it Disk-Divisor describing rooted ribbon tree}, or a {\it DD-ribbon tree} for short, is 
$\mathcal R = (R;\frak v^0;\mathcal S,\alpha,\le)$
with the following properties:
\begin{enumerate}
\item
$R$ is a ribbon tree.
\item
The set of vertices $C_0(R)$ is divided into disjoint union 
of two subsets $C^{\rm int}_0(R)$ and $C^{\rm ext}_0(R)$, 
the set of all {\it interior} and {\it exterior} vertices.
The valency of any exterior vertex is one.
\item
We fix one exterior vertex $\frak v^0$, which we call the {\it root 
vertex}.
We require that the number of exterior vertices to be equal to $k+1$.
We enumerate them as $\frak v^0,\dots,\frak v^{k}$
so that it respects the counter clockwise orientation,
induced by the ribbon structure.
\item
$\alpha$ is a map from $C_0^{\rm int}(R)$ to $\Pi_2(X,L;\bbZ)$ .
\item
To each interior vertex $\frak v$, we associate $\mathcal S(\frak v)$, which is a  
DD-tree of type $(\alpha(\frak v),k_{\frak v})$. 
Here $k_{\frak v}+1$ is the valency of the vertex $\frak v$.
\item
	Let $\hat S(\frak v)$ denote the detailed DD-tree associated to 
	$\mathcal S(\frak v)$. We define:
	\begin{equation}
		C^{\rm ins}_0(\hat R) = \bigcup_{{\frak v} \in C^{\rm int}_0(R)} C_0^{\rm ins}(\hat S(\frak v)).
	\end{equation}
	We call $C^{\rm ins}_0(\hat R)$ the set of inside vertices of the detailed tree associated to $\mathcal R$.
	The relation $\leq$ is a quasi order on $C^{\rm ins}_0(\hat R)$. 
	We require that the restriction of $\le$ to $C_0^{\rm ins}(\hat S(\frak v))$
	coincides with the partial order determined by the structure of $\mathcal S(\frak v)$.
\end{enumerate}
The homology class $\alpha(\mathcal R)$ of $\mathcal R$ is the sum of the homology 
classes $\alpha(\frak v)$ for all the interior vertices.
We also call $(\alpha(\mathcal R);k)$ the {\it type} of $\mathcal R$.
\end{definition}
\begin{figure}[h]
\centering
\includegraphics[scale=0.4]{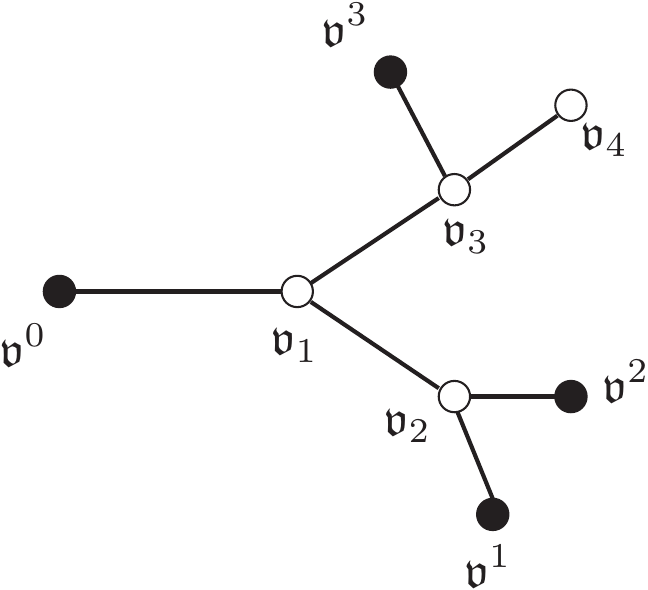}
\caption{A Disk-Divisor Describing Rooted Ribbon Tree.}
\label{Drobbontree}
\end{figure}
\begin{example}\label{exam352}
A DD-ribbon tree $R$ is given in Figure \ref{Drobbontree}.
Here black circles are exterior vertices and white circles are 
interior vertices.
The ribbon tree in Figure \ref{Drobbontree} corresponds to 
the configuration of the disks drawn in Figure \ref{treeofdisks}.
The tree $R$ has four interior verties $\frak v_1,\frak v_2,\frak v_3,\frak v_4$.
The corresponding DD-trees are given  in Figure \ref{FigofS(v)}.
\end{example}
\begin{figure}[h]
\centering
\includegraphics[scale=0.3]{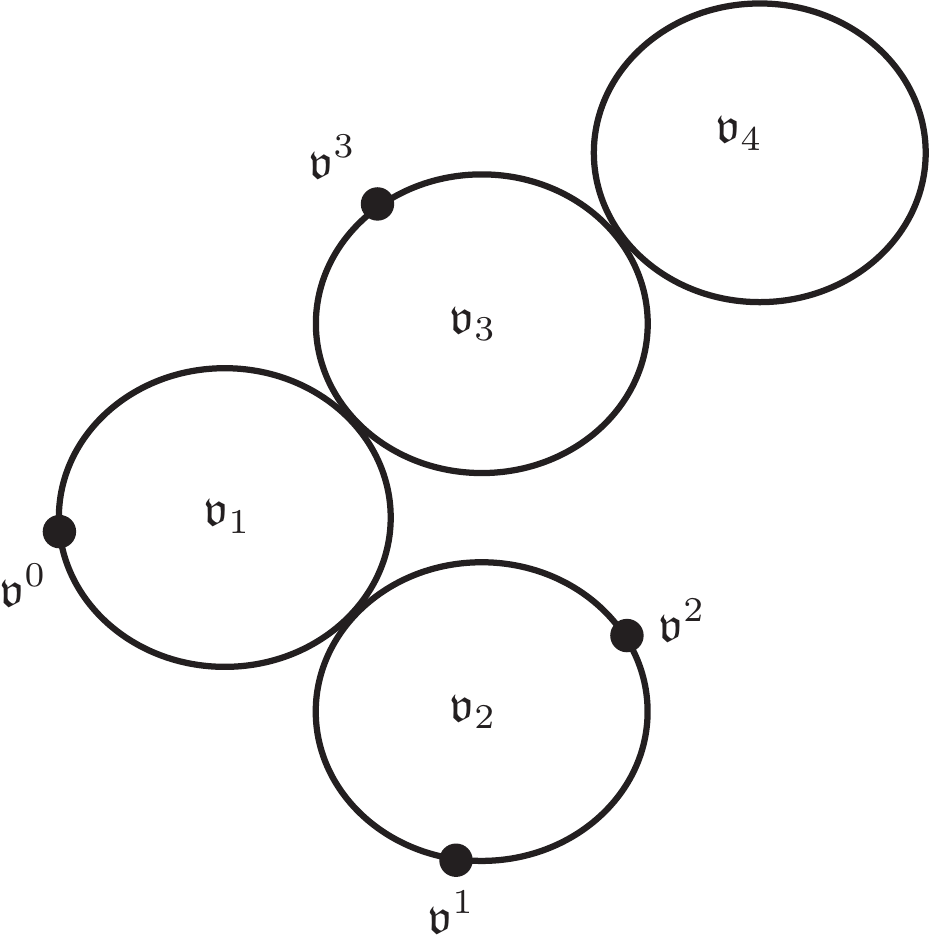}
\caption{Configuration of disks corresponding to Figure \ref{Drobbontree}.}
\label{treeofdisks}
\end{figure}
To each $\mathcal R$ as in Definition \ref{DDdesroottree}, we can associate a detailed tree $\hat R$ by forming the following disjoint union and identifying each interior vertex $\frak v$ of $R$ with the root of $\hat{\mathcal S}(\frak v)$:
$$
R \sqcup \bigcup_{\frak v \in C^{\rm int}_0(R)} \hat{\mathcal S}(\frak v).
$$
%
The detailed tree associated to Example \ref{exam352} is given in Figure \ref{FigofS(v)}.
In this figure we omit the exterior vertices of $R$ and 
the edges incident to them. Nevertheless, they are part of 
the detailed  tree. The edges of level 0 are also
drawn by dotted lines.\footnote{
It is probably more natural to draw 3-dimensional figures. However,
we content ourselves to a 2 dimension figure for simplicity.
} The level function on $\hat{\mathcal S}(\mathfrak v_i)$ induced by $\le$ is given in Figure \ref{FigofS(v)}.

\begin{figure}[h]
\centering
\includegraphics[scale=0.3]{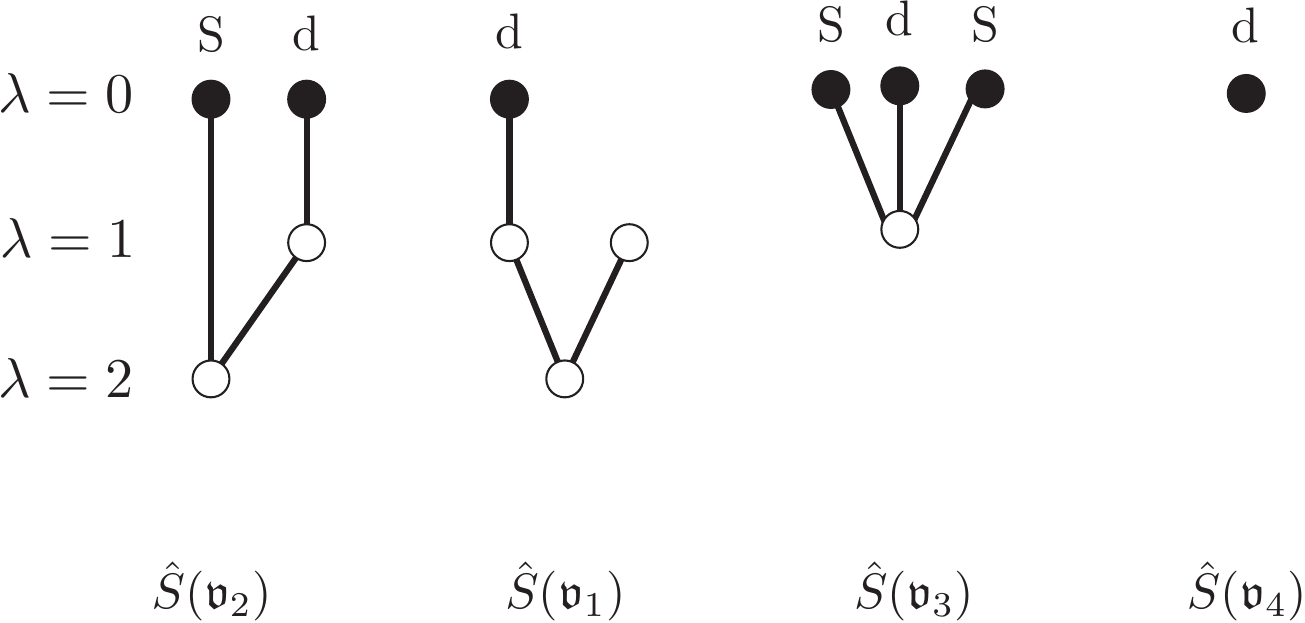}
\caption{$\hat S(\frak v_i)$.}
\label{FigofS(v)}
\end{figure}

\begin{figure}[h]
\centering
\includegraphics[scale=0.3]{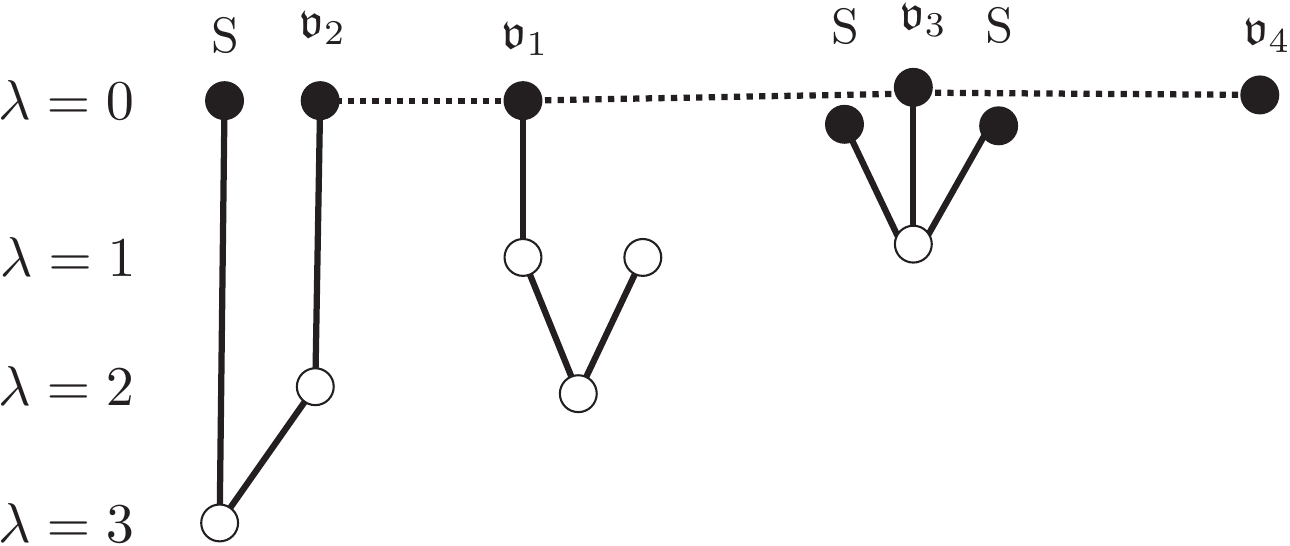}
\caption{The detailed tree $\hat R$ associated to Figures \ref{Drobbontree}
 and \ref{FigofS(v)}.}
\label{DetailedtreehatS}
\end{figure}

The detailed tree for DD-ribbon trees and 
DD-trees have the following differences:
\begin{enumerate}
	\item
	For detailed trees of DD-ribbon trees, the vertices of level 0 have one of the colors d, s or ext. 
	d stands for disks, and there is one vertex with this color for each interior vertex of $R$. 
	$\rm s$ stands for spheres, which appear in $\hat S(\frak v)$ for interior vertices $\frak v$. 
	There is a vertex with color ext for each exterior vertex of $R$.
	\item We fix a root among the vertices with color ext.
	\item There may be level 0 edges that join level 0 vertices of color d.
\end{enumerate}


\begin{definition}\label{defn350550}
Let $L_0,L_1$ be compact Lagrangian submanifolds of $X \setminus \mathcal D$ 
which intersect transversally, and $p,q \in L_0 \cap L_1$.
A {\it Strip-Divisor describing rooted ribbon tree}, or an {\it SD-ribbon tree} for short, is a 7-tuple
 $\mathcal R = (R;\frak v_l,\frak v_r;\mathcal S,{\rm pt},\alpha,\le)$
with the following properties (see Figure \ref{FiguregraphR}.):
\begin{enumerate}
\item
$R$ is a ribbon tree. 
\item
The set of vertices $C_0(R)$ is divided into disjoint union 
of two subsets $C^{\rm int}_0(R)$ and $C^{\rm ext}_0(R)$, 
the set of all {\it interior} and {\it exterior} vertices.
The valency of exterior vertices are one.
\item
$\frak v_l,\frak v_r$ are exterior vertices of $R$, which we call {\it the left most 
vertex and the right most vertex}.
There is a subgraph $C$, which is a path that starts from $\frak v_l$ and ends at $\frak v_r$.
We call a vertex and an edge  of $C$ a {\it strip vertex} and a {\it strip edge}.
We do not regard $\frak v_l$ or $\frak v_r$ as a strip vertex.
In particular, all the strip vertices are interior.
We require $C$ contains at least one strip vertex.
\item
The complement $R \setminus C$ is split into $R_0$ and $R_1$,
as follows. We orient $C$ so that it starts from $\frak v_{l}$
and ends with $\frak v_r$. Then $R_0$ lies to the right and $R_1$ lies 
to the left of $C$. (We use the embedding of $R$ in $\bbR^2$ associated to its 
ribbon structure here.)
The vertices or edges in $R_0$ (resp. $R_1$)  are called ${\rm d_0}$-type (resp. ${\rm d_1}$-type) 
vertices or edges.
The graph $R_0$ (resp. $R_1$) has exactly $k_0$ (resp. $k_1$) 
exterior edges.
\item
${\rm pt}$ assigns to each  strip edge $\frak e$ an element of $L_0\cap L_1$.
If $\frak e$ contains $\mathfrak v_{l}$ then ${\rm pt}(\frak e) = p$
if $\frak e$ contains $\mathfrak v_r$ then ${\rm pt}(\frak e) = q$.
\item
If $\frak v$ is a strip vertex then $\alpha(\frak v) \in \Pi_2(L_1,L_0;
{\rm pt}(\frak e_l(\frak v)),{\rm pt}(\frak e_r(\frak v))$.
Here $\frak e_l(\frak v)$ (resp. $\frak e_r(\frak v)$)
is the edge of $C$ containing $\frak v$ 
which lies in the same connected component of $C \setminus \{\frak v\}$ 
as $\frak v_{l}$ (resp.  $\frak v_r$). 
\item
If $\frak v$ is an interior ${\rm d_0}$ vertex (resp.  an interior ${\rm d_1}$ vertex) then 
$\alpha(\frak v) \in \Pi_2(X,L_0;\bbZ)$ (resp. $\alpha(\frak v) \in \Pi_2(X,L_1;\bbZ)$).
\item
If $\frak v$ is a strip vertex, then $\mathcal S(\frak v)$ is an SD-tree of type $(\alpha(\frak v);k_0,k_1)$ where $k_i$ is the number of edges in $R_i$ incident to $\frak v$.  If $\frak v$ is a ${\rm d_0}$ vertex then $\mathcal S(\frak v)$ is a 
DD-tree for the Lagrangian $L_0$ of type $(\alpha(\frak v);k)$ where $k+1$ is the number of edges incident to $\frak v$.
If $\frak v$ is a ${\rm d_1}$ vertex then $\mathcal S(\frak v)$ is a 
DD-tree for the Lagrangian $L_1$ of type $(\alpha(\frak v);k)$ where $k+1$ is the number of edges incident to $\frak v$.
\item
$\le$ is a quasi partial order on
\begin{equation}
C^{\rm ins}_0(\hat R) = 
\bigcup_{{\frak v} \in C^{\rm int}_0(R)} C_0^{\rm ins}(\hat S(\frak v)).
\end{equation}
The set $C^{\rm ins}_0(\hat R)$  is called the set of inside vertices of the detailed tree associated to $R$.
$\le$ is a quasi partial order on $C^{\rm ins}_0(\hat R)$.
The restriction of $\le$ to $C_0^{\rm ins}(\hat S(\frak v))$ 
coincides with the quasi order coming from the structure of $\mathcal S(\frak v)$.
\end{enumerate}
The homology class $\alpha(\mathcal R)$ of $\mathcal R$ is the sum of the homology 
classes $\alpha(\frak v)$ for all interior vertices $\frak v$.
We say the {\it type} of $\mathcal R$ is $(p,q;\alpha(\mathcal R);k_0,k_1)$.
\end{definition}

\begin{figure}[h]
\centering
\includegraphics[scale=0.4]{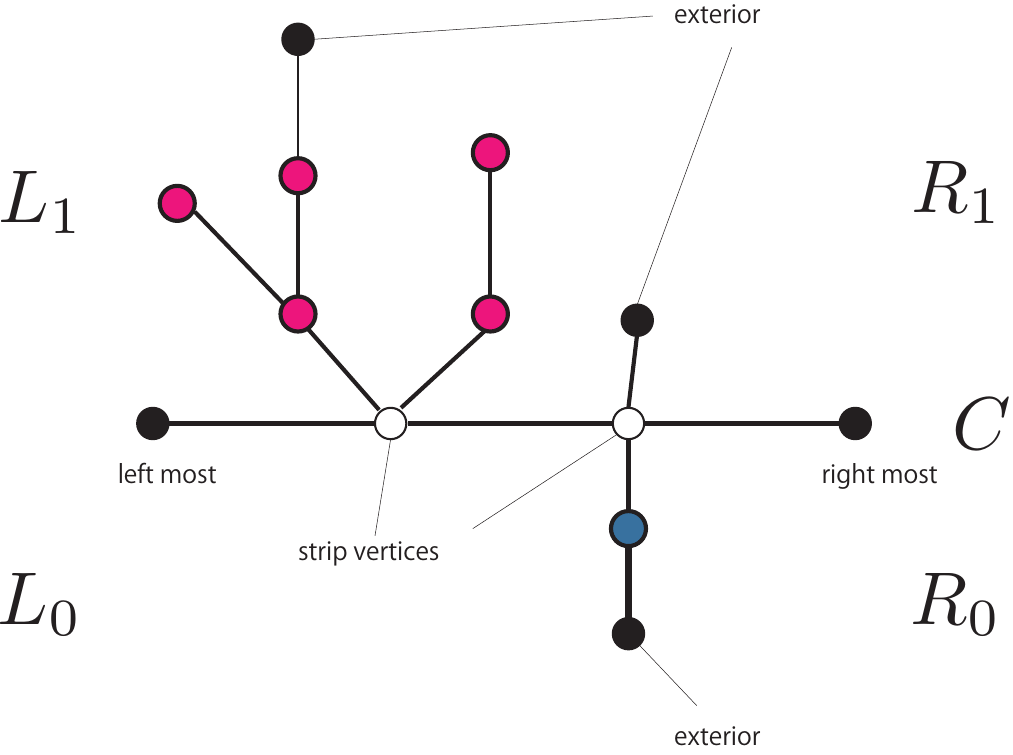}
\caption{Strip-Divisor describing ribbon tree.}
\label{FiguregraphR}
\end{figure}

\begin{figure}[h]
\centering
\includegraphics[scale=0.4]{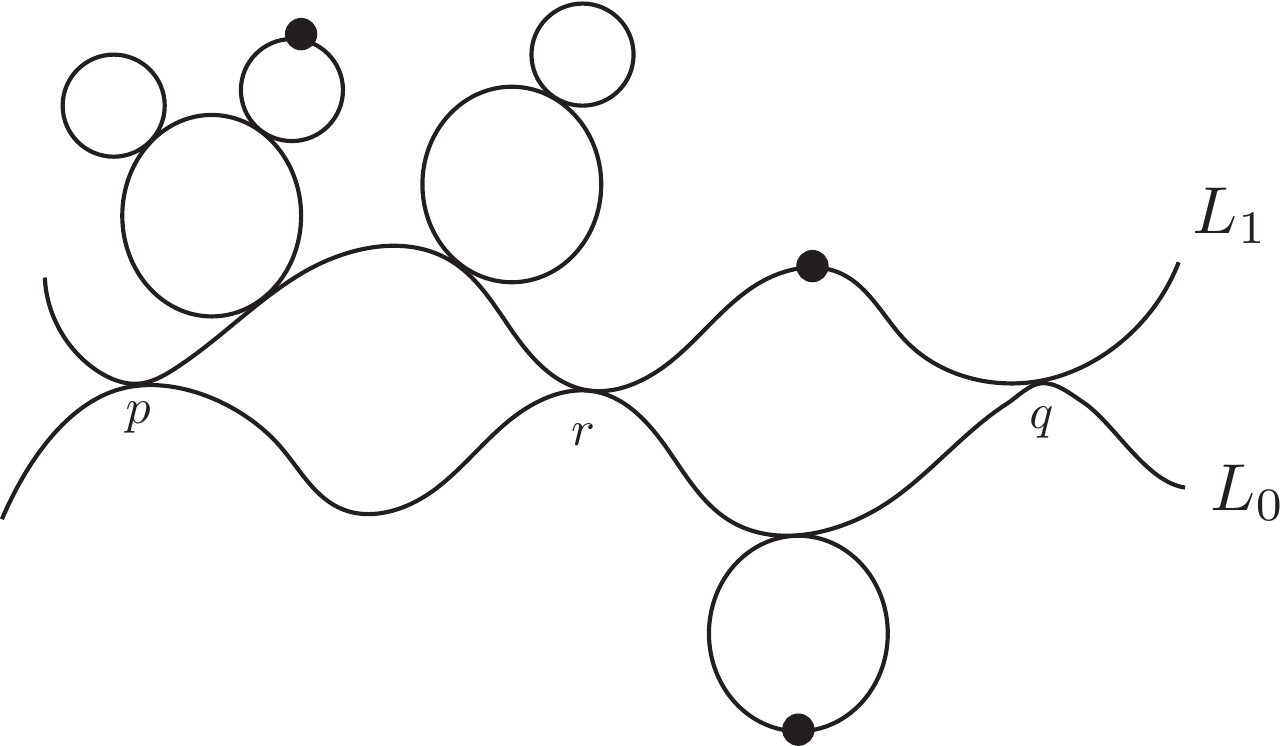}
\caption{Configuration corresponding to Figure \ref{FiguregraphR}.}
\label{Stripdiskconf}
\end{figure}

We can define the notion of the detailed tree $\hat R$ associated to $\mathcal R$
in the same way as in the case of DD-ribbon trees.
We omit the details here.
The main differences are:

\begin{enumerate}
\item
The vertices of level 0 have one of the colors ${\rm d_0}$, ${\rm d_1}$, str, le, ri, s, mk0, mk1. 
\item ${\rm d_0}$ (resp. ${\rm d_1}$) stands for disks with boundary condition $L_0$ (resp. $L_1$). The vertices with color ${\rm d_0}$ (resp. ${\rm d_1}$)  correspond to the interior vertices of $R_0$ (resp. $R_1$).
\item
mk0 (resp. mk1) labels  the exterior vertices of $R_0$ (resp. $R_1$) 
and corresponds to the $k_0$ (resp. $k_1$) boundary marked points in $L_0$ (resp. $L_1$). The vertices with color str 
correspond to the interior vertices $\frak v$ of $C$.
\item
le, ri correspond to the left most and the right most vertices of $C$.
\end{enumerate}
Now we are ready to describe the moduli spaces associated to the 
DD- or SD- ribbon trees. We first define evaluation maps 
at boundary marked points: 
$$
{\rm ev}^{\partial}_j : \mathcal M_{k+1}^{\rm reg, d}(\beta;{\bf m})
\to L
$$
for $j=0,\dots,k$  by
\begin{equation}
{\rm ev}^{\partial}_j((\Sigma,\vec z,\vec w),u) = u(z_i)
\end{equation}
and
$$
{\rm ev}^{\partial}_{i,j} : \mathcal M_{k_1,k_0}^{\rm reg}(L_1,L_0;p,q;\beta;{\bf m})
\to L_i
$$
for $i=0,1$ and $j=1,\dots,k_i$ by:
\begin{equation}
{\rm ev}^{\partial}_{i,j}((\Sigma,\vec z_0,\vec z_1,\vec w),u)
= u(z_{i,j}).
\end{equation}
\begin{definition}
Let $\mathcal R = (R;\frak v_0;\mathcal S,\alpha,\le)$
be a DD-ribbon tree, $\hat R$  the associated detailed tree and $\lambda$  the level function of $\hat R$.
	For an interior vertex $\frak v$ of $R$, let $v \in C^{\rm int}_0(\mathcal S(\frak v))$. We then define:
\begin{equation}
\widetilde{\mathcal M}^0(\mathcal R;v) = \widetilde{\mathcal M}^0(\mathcal S(v))
\end{equation}
where the right hand side is defined as in 
(\ref{form335}), (\ref{form336}) or (\ref{fpr315315rev}).
\par
Let $C^{{\rm int},\lambda=0}_1(\hat R)$ be the set of all 
interior edges of $\hat R$ of level $0$ 
and $C^{\lambda>0}_1(\hat R)$ be the set of all edges with positive level.
We define:
$$
{\rm EV} : \prod_{v \in C^{\rm ins}_0(\hat R)}\widetilde{\mathcal M}^{0}(\mathcal R;v)
\to \prod_{e \in C^{\lambda>0}_1(\hat R)}(\mathcal D \times \mathcal D)
\times \prod_{e \in C^{{\rm int},\lambda=0}_1(\hat R)}(L \times L)
$$
as follows.
Let $e \in C^{\lambda>0}_1(\hat R)$. Then there exists a unique $\frak v \in C^{\rm int}_0(R)$
such that $e \in C_1(\mathcal S(\frak v))$.
We define the $e$-component of $\prod_{e \in C^{\lambda>0}_1(\hat R)}(\mathcal D \times \mathcal D)$
by (\ref{form338}).
Let $e \in C^{{\rm int},\lambda=0}_1(\hat R)$. We define $s(e), t(e) \in C^{\rm int}_0(R)$
the vertices incident to $e$ such that $s(e)$ is in the 
same connected component  of $R \setminus \{ e\}$ as the root.
We label the edges of $\frak v = t(e)$ as $e_0(\frak v),\dots,e_{k_v}(\frak v)$ 
such that $t(e_0(\frak v)) = \frak v$. 
(See Figure \ref{Figuret(e)v(e)} below.)
\begin{figure}[h]
\centering
\includegraphics[scale=0.4]{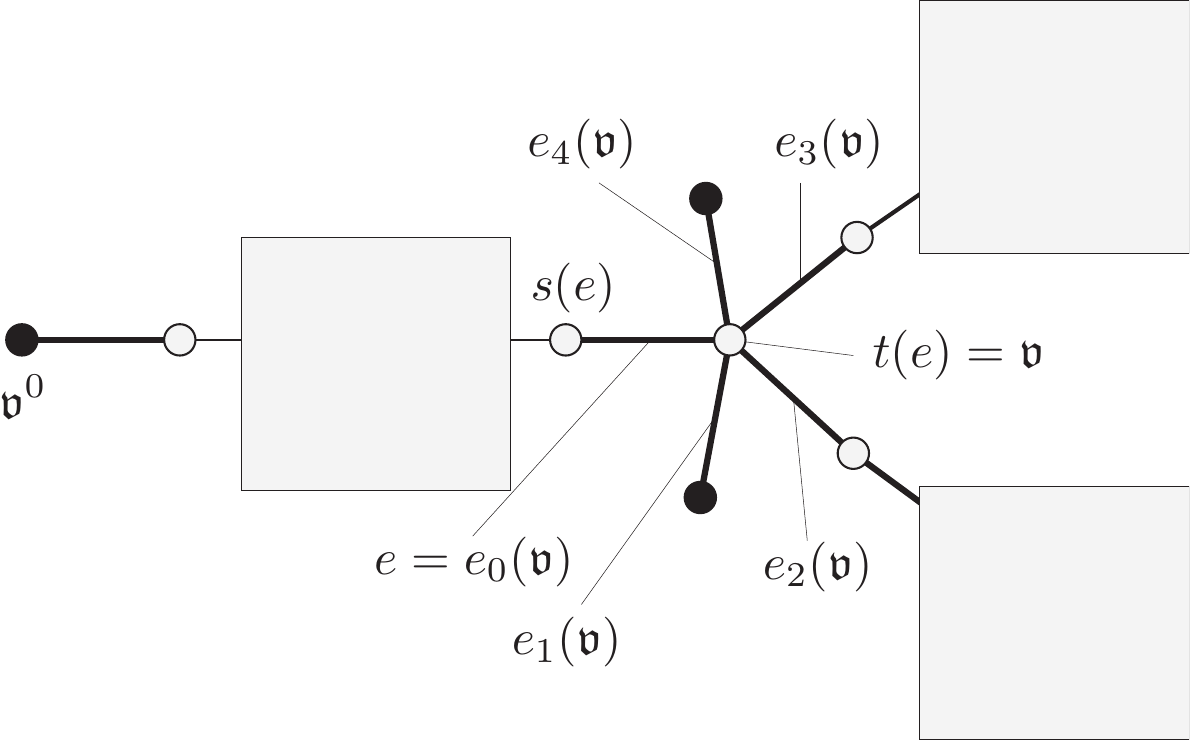}
\caption{$e_i(\frak v)$.}
\label{Figuret(e)v(e)}
\end{figure}
Suppose $e$ is the $k_e$-th edge of $s(e)$.
Now we define the $e$-component of 
$\prod_{e \in C^{\lambda=0}_1(\hat R)}(L \times L)$ by
$$
{\rm EV}_e({\bf x}_v;v \in C^{\rm ins}_0(\hat R) ) = ({\rm ev}^{\partial}_0({\bf x}_{t(e)}),{\rm ev}^{\partial}_{k_e}({\bf x}_{s(e)})).
$$
We now define:
\begin{equation}\label{fpro351}
	\widetilde{\mathcal M}^{0}(\mathcal R) 
	= \prod_{v \in C^{\rm ins}_0(\hat R)}\widetilde{\mathcal M}^{0}(\mathcal R;v) \hspace{3cm}
\end{equation}
\begin{equation*}
	\hspace{3cm}\, _{\rm EV}\times_{\star}
	\left(\left(\prod_{e \in C^{\lambda>0}_1(\hat R)} \Delta_{\mathcal D}\right)  \times
	\left(\prod_{e \in C^{{\rm int},\lambda=0}_1(\hat R)}\Delta_{L}\right)\right),
\end{equation*}
where $\star$ is the inclusion map:
$$
\aligned
&\left(\prod_{e \in C^{\lambda>0}_1(\hat R)} \Delta_{\mathcal D}\right)  \times
\left(\prod_{e \in C^{{\rm int},\lambda=0}_1(\hat R)}\Delta_{L}\right) \\
&\to 
\prod_{e \in C^{\lambda>0}_1(\hat R)}(\mathcal D \times \mathcal D)
\times \prod_{e \in C^{{\rm int},\lambda=0}_1(\hat R)}(L \times L).
\endaligned
$$
Let $\vert \lambda\vert$ be the total number of positive levels,
namely, the image of $\lambda$ is $\{0,1,\dots,\vert \lambda\vert\}$.
We define an action of $\bbC_{*}^{\vert \lambda\vert}$  on $\widetilde{\mathcal M}^{0}(\mathcal R)$
by the same formula as (\ref{form33940}).
(This action is trivial on the components ${\bf x}_v$  with $\lambda(v) = 0$.)
\par
We define 
\begin{equation}\label{form352minus}
\widehat{\mathcal M}^{0}(\mathcal R) = 
\widetilde{\mathcal M}^{0}(\mathcal R)/\bbC_{*}^{\vert \lambda\vert}
\end{equation}
The group of automorphisms 
${\rm Aut}(\mathcal R)$ is the 
direct product 
$$
\prod_{\frak v \in C_0^{\rm int}(R)} {\rm Aut}(\mathcal S(\frak v)).
$$
We also define:
\begin{equation}\label{form352}
{\mathcal M}^{0}(\mathcal R)
=\widehat{\mathcal M}^{0}(\mathcal R)/{\rm Aut}(\mathcal R).
\end{equation}
The RGW compactification $\mathcal M^{\rm RGW}_{k+1}(L;\beta)$ 
of $\mathcal M_{k+1}^{\rm reg}(L;\beta)$, as a set, is defined to be:
\begin{equation}\label{form3531}
\mathcal M^{\rm RGW}_{k+1}(L;\beta)
= \bigcup_{\mathcal R}{\mathcal M}^{0}(\mathcal R)
\end{equation}
where the disjoint union in the right hand side is taken over all
DD-ribbon trees $\mathcal R$ of type $(\beta,k+1)$.
\end{definition}
The $(i,i+1)$ level shrinking of a DD-ribbon tree 
is defined as in Definition \ref{defn326levelsh} for $i>0$
and  as
in Definition \ref{01levelshrink} for $i=0$.
We say $\mathcal R'$ is obtained from $\mathcal R$ by level shrinking and write 
$\mathcal R' >' \mathcal R$ if $\mathcal R'$ is obtained from $\mathcal R$
by finitely many iterations of the level shrinking operations.
\par
We also need shrinkings of level $0$ edges.
\begin{definition}\label{defn3606060}
Let $\hat R$ be the detailed tree associated to a DD-ribbon tree $\mathcal R$ and $e$ be an interior 
level $0$ edge.
We remove the edge $e$ and identify its two vertices $v'$, $v''$ 
to obtain $v$.
All the edges other than $e$ which contains one of $v'$ or $v''$
will be incident to the new vertex. 
The homology class of $v$ is  v
$\beta(v) = \beta(v') + \beta(v'')$.
We obtain a detailed tree associated to a new DD-ribbon tree of the same type.
We say $\mathcal R' >'' \mathcal R$ if $\mathcal R'$ 
is obtained from $\mathcal R$ by applying the above process 
finitely many times. We also say $\mathcal R'$ is obtained 
from $\mathcal R$ by {\it level 0 edge shrinkings}.
\par
We write $\mathcal R' > \mathcal R$ if $\mathcal R'$ 
is obtained from $\mathcal R$ by finitely many iterations of level shrinkings and 
level $0$ edge shrinkings.
\end{definition}

We sketch some of the basic properties of the compactification in \eqref{form3531} which will be proved in the rest of this paper and the sequels. In Section \ref{sec:topology}, we define a topology on $\mathcal M^{\rm RGW}_{k+1}(L;\beta)$, called the RGW topology, that is compact and metrizable. Moreover, for any RD-ribbon tree $\mathcal R$, the space 
\begin{equation*}
	{\mathcal M}(\mathcal R) := {\mathcal M}^{0}(\mathcal R) \cup\bigcup_{\mathcal R' < \mathcal R}{\mathcal M}^{0}(\mathcal R').
\end{equation*}
is a closed subset of $\mathcal M^{\rm RGW}_{k+1}(L;\beta)$. In \cite{part2:kura}, we define a Kuranishi structure on $\mathcal M^{\rm RGW}_{k+1}(L;\beta)$ with corners such that the underlying subset of the codimension $n$ corner is the union of all moduli spaces ${\mathcal M}^{0}(\mathcal R)$, where $R$ has at least $n+1$ interior vertices. A more detailed description of the boundary, which is important for our purposes, will be given in \cite[Subsection 2.2]{part3:FH}.
%
%
%

Next, we define the moduli space associated to an  SD-ribbon tree.

\begin{definition}
	Let $\mathcal R = (R;\frak v_l,\frak v_r;\mathcal S,{\rm pt},\alpha,\le)$ be an SD-ribbon tree, 
	$\hat R$ be the associated detailed tree
	and $\lambda$ be the level function of $\mathcal R$.
	For an interior vertex $\frak v$ of $R$, let $v \in C^{\rm int}_0(\mathcal S(\frak v))$.
	We define:
	\begin{equation}
		\widetilde{\mathcal M}^0(\mathcal R;v) = \widetilde{\mathcal M}^0(\mathcal S(\frak v);v)
	\end{equation}
	where the right hand side is as in 
	\eqref{form335}, \eqref{form336}, \eqref{fpr315315rev} or \eqref{formform343343}.

        Let $C^{{\rm int},\lambda=0}_1(R_i)$ be the set of all 
        interior edges of $R_i$ of level $0$ that are not strip edges ($i=0,1$), 
        and $C^{\lambda>0}_1(\hat R)$ be the set of the edges of level $>0$.
        We define
        $$
        \aligned
        {\rm EV} : &\prod_{v \in C^{\rm ins}_0(\hat R)}\widetilde{\mathcal M}^{0}(\mathcal R;v)
        \\
        &\to \prod_{e \in C^{\lambda>0}_1(\hat R)}(\mathcal D \times \mathcal D)
        \times \prod_{e \in C^{{\rm int},\lambda=0}_1(R_0)}(L_0 \times L_0)
        \times \prod_{e \in C^{{\rm int},\lambda=0}_1(R_1)}(L_1 \times L_1)
        \endaligned
        $$
        as follows.
        Let $e \in C^{\lambda>0}_1(\hat R)$. Then there exists a unique $\frak v \in C^{\rm ins}_0(R)$
        such that $e \in C_1(\mathcal S(\frak v))$.
        We define the $e$-component of $\prod_{e \in C^{\lambda>0}_1(\hat S)}(\mathcal D \times \mathcal D)$
        by (\ref{form338}).
        \par
        Let $e \in C^{{\rm int},\lambda=0}_1(R_1)$ and $s(e), t(e) \in C^{\rm int}_0(R)$ be
        the vertices incident to $e$, such that $s(e)$ is in the 
        same connected component of $R \setminus \{e\}$ as $C$.
        We enumerate the edges of $\frak v' = t(e)$ as $e_0(\frak v'),\dots,e_{k_{\frak v'}}(\frak v')$ 
        such that $s(e_0(\frak v'))$ is $ \frak v=s(e)$. 
        (See Figure \ref{Figure3551} below.)
        
        \begin{figure}[h]
        \centering
        \includegraphics[scale=0.4]{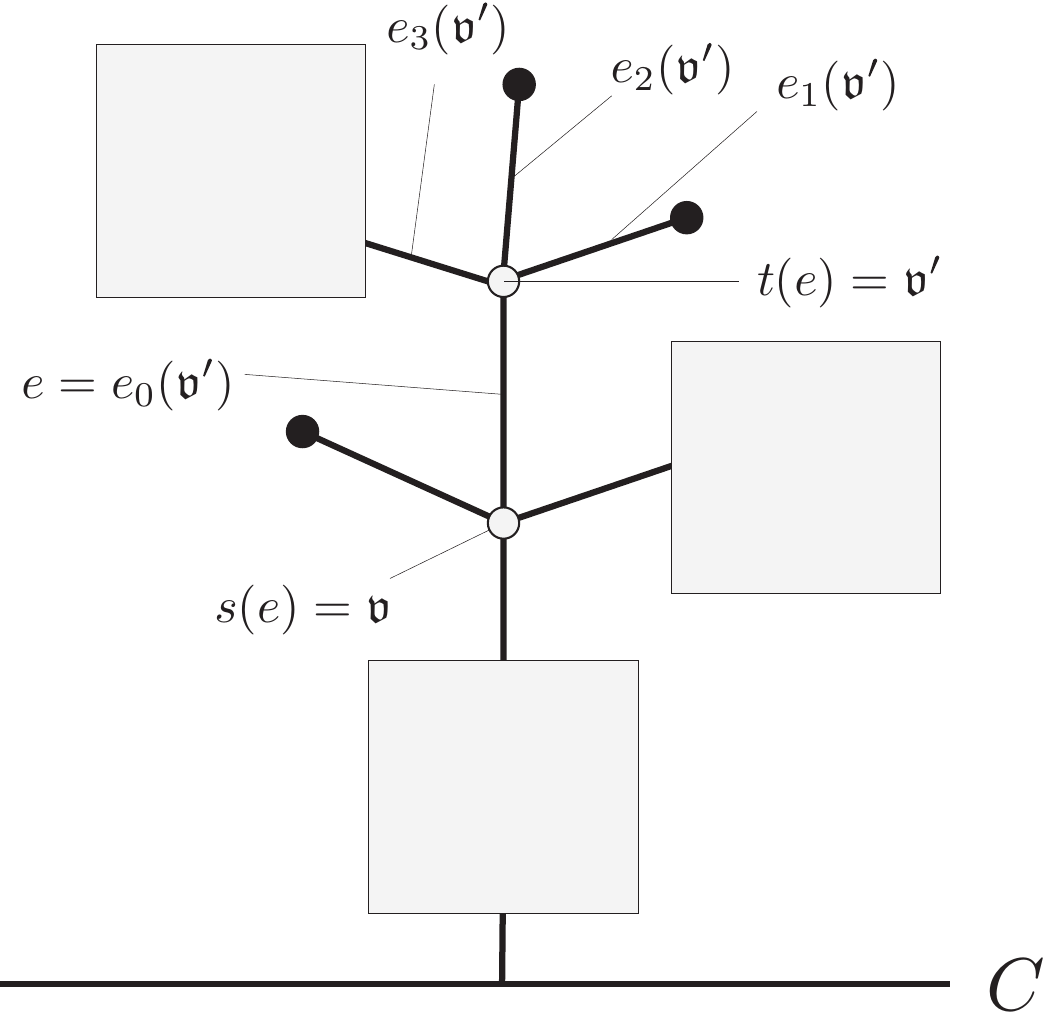}
        \caption{$e_i(\frak v')$.}
        \label{Figure3551}
        \end{figure}
        
        If $s(e)\in R_1$, we enumerate the edges incident to $e$ using the ribbon structure such that 
        the first edge of $s(e)$ is labeled by $0$.
        Suppose $e$ is the $k_e$-th edge of $s(e)$.
        Now we define the $e$-th component of  the product space 
        $\prod_{e \in C^{{\rm int},\lambda=0}_1(R_1)}(L_1 \times L_1)$ by
        $$
        {\rm EV}_e({\bf x}_v;v \in C^{{\rm int},\lambda=0}_0(\hat R) ) = 
        ({\rm ev}^{\partial}_0({\bf x}_{t(e)}),{\rm ev}^{\partial}_{k_e}({\bf x}_{s(e)})).
        $$
        If $s(e) \in C$, then we enumerate the edges of $s(e)$ in $R_1$ in the 
        counter clockwise order. Let $e$ be the $k_e$-th edge among them.
        (See Figure \ref{Figure35522} below).
        We then define $e$-th component of 
        $\prod_{e \in C^{\rm int,\lambda=0}_1(R_1)}(L_1 \times L_1)$ by
        $$
        {\rm EV}_e({\bf x}_v;v \in C^{{\rm int},\lambda=0}_0(\hat R) ) = 
        ({\rm ev}^{\partial}_0({\bf x}_{t(e)}),{\rm ev}^{\partial}_{1,k_e}({\bf x}_{s(e)})).
        $$
        The case $e \in C^{{\rm int},\lambda=0}_1(R_0)$ can be defined in the same way.        
        \begin{figure}[h]
        \centering
        \includegraphics[scale=0.4]{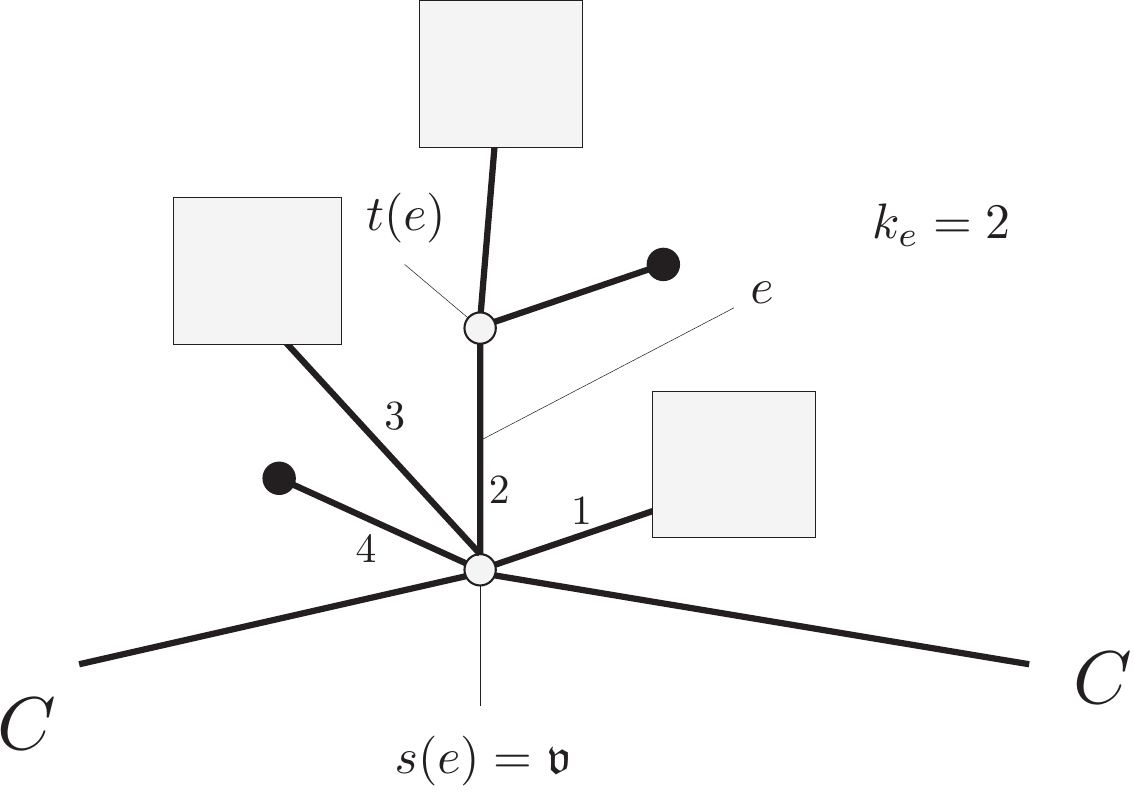}
        \caption{$k_e$.}
        \label{Figure35522}
        \end{figure}

        \par
        We define:
        \begin{equation}\label{form353}
        \widetilde{\mathcal M}^{0}(\mathcal R) 
        = \prod_{v \in C^{\rm ins}_0(\hat R)}\widetilde{\mathcal M}^{0}(\mathcal R;v) \, _{\rm EV}\times_{\star}\hspace{5cm}
        \end{equation}
        \begin{equation*}
	\hspace{1cm}
	\left(\left(\prod_{e \in C^{\lambda>0}_1(\hat R)} \Delta\right)  \times
        \left(\prod_{e \in C^{{\rm int},\lambda=0}_1(R_1)}\Delta\right)
        \times 
        \left(\prod_{e \in C^{{\rm int},\lambda=0}_1(R_0)}\Delta\right)\right),
        \end{equation*}
        where
        $\star$ is the inclusion map.
        Let $\vert \lambda\vert$ be the total number of positive levels.
        We define an action of $\bbC_{*}^{\vert \lambda\vert}$  on $\widetilde{\mathcal M}^{0}(\mathcal R)$
        by the same formula as (\ref{form33940}).
        We define 
        $$
        \widehat{\mathcal M}^{0}(\mathcal R) = 
        \widetilde{\mathcal M}^{0}(\mathcal R)/\bbC_{*}^{\vert \lambda\vert}
        $$
        The group of automorphisms 
        ${\rm Aut}(\mathcal R)$ is  given as the 
        direct product 
        $$
        \prod_{\frak v \in C_0^{\rm int}(R)} {\rm Aut}(\mathcal S(\frak v)).
        $$
        We define:
        $$
        {\mathcal M}^{0}(\mathcal R)
        =\widehat{\mathcal M}^{0}(\mathcal R)/{\rm Aut}(\mathcal R).
        $$

        Now the RGW compactification $\mathcal M^{\rm RGW}_{k_1,k_0}(L_1,L_0;p,q;\beta)$ 
        of  
        $\mathcal M_{k_1,k_0}^{\rm reg}(L_1,L_0;p,q;\beta;\emptyset)$ is defined as
        \begin{equation}\label{form3531-strip}
        \mathcal M^{\rm RGW}_{k_1,k_0}(L_1,L_0;p,q;\beta)
        = \bigcup_{\mathcal R}{\mathcal M}^{0}(\mathcal R)
        \end{equation}
        where the disjoint union in the right hand side is taken over all
        RD-ribbon trees $\mathcal R$ of type $(p,q;\beta;k_0,k_1)$.
\end{definition}
We define the notion of level shrinking and level $0$ edge shrinking for SD-ribbon trees in the same way as in the case of DD-ribbon trees and write $\mathcal R' <\mathcal R$, if $\mathcal R$ is obtained from $\mathcal R'$ by a composition of finitely many iterations of these two operations.

As in the disc case, we list some of the basic properties of the compactification in \eqref{form3531-strip} which will be proved later. In Section \ref{sec:topology}, we define a topology on the set $\mathcal M^{\rm RGW}_{k_1,k_0}(L_1,L_0;p,q;\beta)$, called the RGW topology, that is compact and metrizable. For any RD-ribbon tree $\mathcal R$, the space 
\begin{equation*}
	{\mathcal M}(\mathcal R) := {\mathcal M}^{0}(\mathcal R) \cup\bigcup_{\mathcal R' < \mathcal R}{\mathcal M}^{0}(\mathcal R').
\end{equation*}
is a closed subset of $\mathcal M^{\rm RGW}_{k_1,k_0}(L_1,L_0;p,q;\beta)$. In \cite{part2:kura}, we define a Kuranishi structure on $\mathcal M^{\rm RGW}_{k_1,k_0}(L_1,L_0;p,q;\beta)$ with corners such that the underlying subset of the codimension $n$ corner is the union of all moduli spaces ${\mathcal M}^{0}(\mathcal R)$, where $R$ has at least $n+1$ interior vertices. A more detailed description of the boundary, which is important for our purposes, will be given in \cite[Subsection 2.2]{part3:FH}.

\section{Stable Map Topology and RGW Topology}
\label{sec:topology}

\subsection{Review of Stable Map Topology}\label{stablemaptopology}
 
We first review the definition of stable map topology, which plays an essential role in the definition of the RGW topology. The idea of compactifying moduli spaces of pseudo-holomorphic curves goes back to the groundbreaking work of Gromov in \cite{G:p-holo-symp}. It was pointed out in \cite{K:stable-map} that the notion of stable maps, which is an adaptation of the notion of stable curves due to Mumford, provides a suitable compactification of the moduli space of pseudo-holomorphic curves. The definition of stable map topology in the form that we use in this paper was introduced in \cite[Definition 10.3]{FO}.

Let $\mathcal M^{{\rm d}}_{k+1,\ell}$ be the compactified moduli space of disks with $k+1$ boundary marked points and $\ell$ interior marked points.  This space is the compactification of the space $\mathcal M^{0, d}_{k+1,\ell} = \{(D^2,\vec z,\vec z^+)\}/\sim$. Here $\vec z = (z_0,\dots,z_k)$, $\vec z^+ = (z^+_1,\dots,z^+_{\ell})$ are distinct points such that $z_i \in \partial D^2$, $z_i^+ \in {\rm Int} D^2$, and $(z_0,\dots,z_k)$ respects the counter clockwise orientation of $\partial D^2$. The equivalence relation $\sim$ is defined by the action of $PSL(2,\bbR) = {\rm Aut}(D^2)$. An element of the compactification $\mathcal M^{\rm d}_{k+1,\ell}$ is an equivalence class of $(\Sigma,\vec z,\vec z^+)$, where $\Sigma$ is a tree like union of disks with double points plus trees of sphere components attached to  interior points of the disks. The tuples $\vec z = (z_0,\dots,z_k)$ and $\vec z^+ = (z^+_1,\dots,z^+_{\ell})$ are respectively boundary and interior marked points. The object $(\Sigma,\vec z,\vec z^+)$ representing an element of $\mathcal M^{\rm d}_{k+1,\ell}$ is also required to satisfy a stability condition. See \cite[Definition 2.1.18]{fooobook} for more details. Hereafter with a slight abuse of notation, we say $(\Sigma,\vec z,\vec z^+)$ is an element of $\mathcal M^{\rm d}_{k+1,\ell}$. Note that the symmetric group $S_\ell$ of order $\ell!$ acts on $\mathcal M^{\rm d}_{k+1,\ell}$ by exchanging the order of interior marked points.

The moduli space $\mathcal M^{\rm d}_{k+1,\ell}$ has the structure of a smooth manifold with boundary and corner. Its codimension $m$ corner consists of the elements $(\Sigma,\vec z,\vec z^+)$ such that $\Sigma$ has at least $m+1$ disk components. (See \cite[Theorem 7.1.44]{fooobook}.)

For any pair of injective and order preserving maps $\frak i_{k,k'} : \{1,\dots,k\} \to \{1,\dots,k'\}$ and $\frak i^+_{\ell,\ell'} : \{1,\dots,\ell\} \to \{1,\dots,\ell'\}$, we define the forgetful map:
\begin{equation}\label{forget5151}
	\frak{fg}_{\frak i_{k,k'},\frak i^+_{\ell,\ell'}} : \mathcal M^{\rm d}_{k'+1,\ell'} \to \mathcal M^{\rm d}_{k+1,\ell}
\end{equation}
as follows. Let $(\Sigma',\vec z^{\,\prime},\vec z^{+ \prime}) \in \mathcal M_{k'+1,\ell'}^{\rm d}$. Define $\vec z = (z_0,z_{\frak i_{k,k'}(1)},\dots,\frak i_{k,k'}(k))$ and $\vec z = (z^+_{\frak i^+_{\ell,\ell'}(1)},\dots,z^+_{\frak i^+_{\ell,\ell'}(\ell)})$. The triple $(\Sigma',\vec z,\vec z^{+})$ may not represent an element of $\mathcal M^{\rm d}_{k+1,\ell}$ if it does not satisfy the stability condition. By shrinking all unstable components of $(\Sigma',\vec z,\vec z^{+})$, we obtain $(\Sigma,\vec z,\vec z^{+})$, which satisfies the stability condition. This element  $(\Sigma,\vec z,\vec z^{+})$ represents $\frak{fg}_{\frak i_{k,k'},\frak i^+_{k,k'}}(\Sigma',\vec z^{\,\prime},\vec z^{+ \prime})$. See \cite[page 419]{fooobook2} for more details.

If $\frak i_{k,k'}$ or $\frak i^+_{\ell,\ell'}$ 
is the identity map, we omit it from the notation 
$\frak{fg}_{\frak i_{k,k'},\frak i^+_{\ell,\ell'}}$.
If $\frak i_{k,k'}$ is the identity map and $\frak i^+_{\ell,\ell'}(i) = i$, for $i=1,\dots,\ell$, then 
we write $\frak{fg}_{\ell',\ell}$
instead of 
$\frak{fg}_{\frak i^+_{\ell,\ell'}}$.

We consider the map
\begin{equation}\label{map52new}
	\frak{fg}_{\ell+1,\ell} : \mathcal M^{\rm d}_{k+1,\ell+1}\to \mathcal M^{\rm d}_{k+1,\ell}.
\end{equation}
As proved in \cite[Lemma 7.1.45]{fooobook2}, the fiber $(\frak{fg}_{\ell+1,\ell})^{-1}(\Sigma,\vec z,\vec z^+)$ is diffeomorphic to $\widetilde{\Sigma}$ where $\widetilde{\Sigma}$ is obtained from $\Sigma$ by replacing each boundary node by an interval. (See Figure \ref{Figure5-1}.)
\begin{figure}[h]
        \centering
        \includegraphics[scale=0.4]{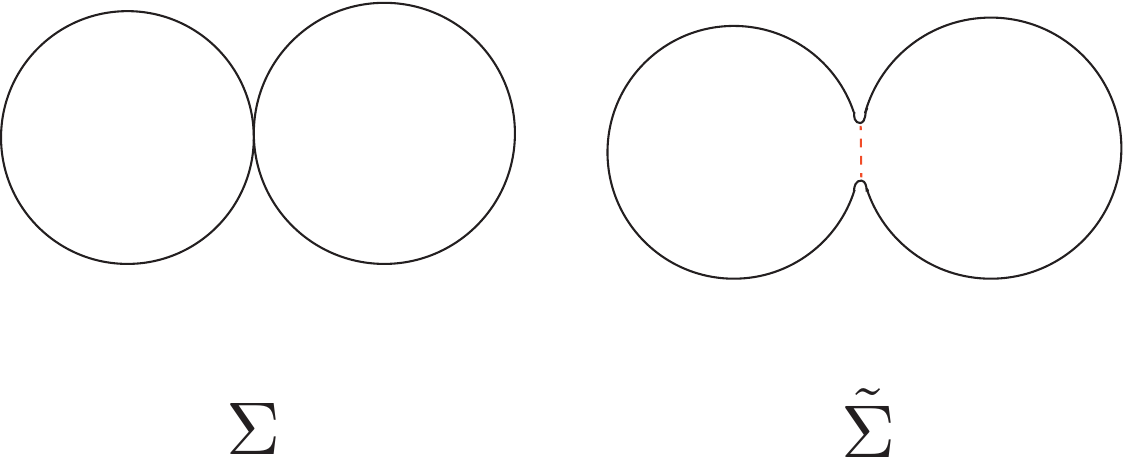}
        \caption{$\widetilde{\Sigma}$}
        \label{Figure5-1}
\end{figure}
The space $\mathcal C_{k+1,\ell}$ is given by shrinking all such intervals to a point. In other words, there exists:
\begin{equation}\label{map52}
	\pi : \mathcal C^{\rm d}_{k+1,\ell}\to \mathcal M^{\rm d}_{k+1,\ell}
\end{equation}
such that the fiber over $(\Sigma,\vec z,\vec z^+)$ is identified with $\Sigma$. Note that $\mathcal C^{\rm d}_{k+1,\ell}$ is merely a topological space and does not carry the structure of a manifold or an orbifold. Let $\mathcal{SC}^{\rm d}_{k+1,\ell}$ be the subset of $\mathcal C^{\rm d}_{k+1,\ell}$ consisting of $x \in \mathcal C^{\rm d}_{k+1,\ell}$ such that if $\pi(x) = (\Sigma,\vec z,\vec z^+)$, then $x$ corresponds to a boundary or interior node of $\Sigma$.

\begin{lemma}\label{lem51}
	The subspace  $\mathcal C^{\rm d}_{k+1,\ell} \setminus \mathcal{SC}^{\rm d}_{k+1,\ell}$ has the structure of a smooth 
	manifold with corners. Moreover, the restriction of \eqref{map52} to $\mathcal C^{\rm d}_{k+1,\ell} \setminus \mathcal{SC}^{\rm d}_{k+1,\ell}$
	is a smooth submersion.
\end{lemma}
\begin{proof}
        By construction $\mathcal C^{\rm d}_{k+1,\ell} \setminus 
        \mathcal{SC}^{\rm d}_{k+1,\ell}$ is an open 
        subset of $\mathcal M^{\rm d}_{k+1,\ell+1}$. 
        This verifies the first part.
        The second part also follows from the 
        corresponding results about $\mathcal M^{\rm d}_{k+1,\ell+1}$,
        which is a consequence of a similar result about the moduli space of marked spheres.
        (The later is classical. See, for example, \cite{alcorGri}.)
\end{proof}
The symmetry group $S_{\ell}$ of order $\ell !$ acts on $\mathcal C^{\rm d}_{k+1,\ell}$ and $\mathcal M^{\rm d}_{k+1,\ell}$ such that \eqref{map52new}, \eqref{map52} are $S_{\ell}$-equivariant. Moreover, the $S_{\ell}$ action on  $\mathcal C^{\rm d}_{k+1,\ell} \setminus \mathcal{SC}^{\rm d}_{k+1,\ell}$ is smooth.

The spaces $\mathcal{C}^{\rm d}_{k+1,\ell}$, $\mathcal M^{\rm d}_{k+1,\ell}$ are 
all metrizable. We fix metrics on them and use these metrics throughout 
the paper. (The whole construction is independent of the
choice of metrics.)
\par
Let $\zeta = (\Sigma,\vec z,\vec z^+) \in \mathcal M^{\rm d}_{k+1,\ell}$.
We define $\Gamma_{\zeta} = \{\gamma \in S_{\ell}
\mid \gamma \zeta = \zeta\}$.
The group $\Gamma_{\zeta}$ has a biholomorphic action on $\Sigma$ 
which permutes interior marked points $\vec z^+$.
This action is necessarily trivial on the disk components.
In the following definition, $S\Sigma$ denotes $\pi^{-1}(\zeta) \cap \mathcal{SC}^{\rm d}_{k+1,\ell}$.
This set consists of boundary and interior nodes of $\Sigma$.

\begin{definition}\label{defn52}
An {\it $\epsilon$-trivialization of the universal family} (\ref{map52}) at $\zeta$ 
is the following object.
\begin{enumerate}
        \item A $\Gamma_{\zeta}$-invariant relatively compact open subset $K \subset \Sigma \setminus S\Sigma$.
        \item A $\Gamma_{\zeta}$-invariant neighborhood $\mathcal U$ of $\zeta$ in  $\mathcal M^{\rm d}_{k+1,\ell}$.
        \item
        A smooth open embedding 
        $\Phi : K \times \mathcal U \to \mathcal C^{\rm d}_{k+1,\ell} \setminus 
        \mathcal{SC}^{\rm d}_{k+1,\ell}$, 
        which is $\Gamma_{\zeta}$-invariant.
        \item
        The following diagram commutes.
       	\[
	\begin{tikzcd}
		K \times \mathcal U \ar[r,"\Phi"]\ar[d]&\mathcal C^{\rm d}_{k+1,\ell} \ar[d,"\pi"]\\
		 \mathcal U\ar[r]&\mathcal M^{\rm d}_{k+1,\ell}
	\end{tikzcd}
	\]
        Here the left vertical arrow is the projection map and the second horizontal arrow is the 
        inclusion map.
        \item
        The image of $\Phi$ contains 
        $\pi^{-1}(\mathcal U) \setminus B_{\epsilon}(S\Sigma)$.
        Here $B_{\epsilon}(S\Sigma)$ is the $\epsilon$ neighborhood 
        of $S\Sigma$ in $\mathcal C^{\rm d}_{k+1,\ell}$.
\end{enumerate}
\end{definition}
The existence of $\epsilon$-trivialization for any $\epsilon$ is a consequence of Lemma \ref{lem51}.
\begin{remark}
        If $\Sigma$ is a disk then $S\Sigma$ is an empty set.
        In this case, an $\epsilon$-trivialization of the universal family 
        is a local trivialization. Note that (\ref{map52}) is a 
        fiber bundle in the $C^{\infty}$ category in a neighborhood 
        of such elements of $\mathcal M^{\rm d}_{k+1,\ell}$.
\end{remark}
\begin{remark}
        A similar `trivialization of the universal family' is 
        described in \cite[Section 9]{FO} 
        or \cite[Section 4]{part2:kura}.
        This notion is employed
        to define a topology in \cite[Definition 10.2]{FO}.
        In \cite{FO} and subsequent works such as 
        \cite[Section 16]{foootech} and \cite[Section 8]{foooexp}, 
        more specific choices of trivializations are 
        used. Namely, a particular choice of coordinate charts
        are used at the nodes and the gluing construction is exploited 
        to obtain a nearby element of $\mathcal M^{\rm d}_{k+1,\ell}$
        and a coordinate of the gluing parameter.
        Such a choice would be useful to study gluing analysis 
        of pseudo-holomorphic curves and to construct Kuranishi neighborhoods, as it was done in \cite{FO,foootech,foooexp}.
        See also \cite{part2:kura}.
        To define stable map topology, we do not
        need these specific choices and can work with 
        any  $\epsilon$-trivialization of the universal family 
        in the above sense.
\end{remark}
Next, we review the stable map compactification $\mathcal M^{\rm d}_{k+1,\ell}(\beta)$ of $\mathcal M^{0,\rm d}_{k+1,\ell}(\beta)$ for $\beta\in\Pi_2(X,L)$. The space $\mathcal M^{0,\rm d}_{k+1,\ell}(\beta)$ consists of $((D^2,\vec z,\vec z^+),u)$ where $(D^2,\vec z,\vec z^+) \in \mathcal M^{\rm 0, d}_{k+1,\ell}$ and $u : (D^2,\partial D^2) \to (X,L)$ is a pseudo-holomorphic map. An element of $\mathcal M_{k+1,\ell}^{\rm d}(\beta)$ is an isomorphism class of  $((\Sigma,\vec z,\vec z^+),u)$, where $\Sigma$ is a tree like union of disks with double points plus trees of sphere components attached to the interior points of the disks, $\vec z$ and $\vec z^+$ are boundary and interior marked points, and $u :(\Sigma,\partial \Sigma) \to (X,L)$ is a pseudo-holomorphic map. The object $((\Sigma,\vec z,\vec z^+),u)$ is required to satisfy the stability condition. (See \cite[Definition 2.1.24]{fooobook} for more details.) We say an element $((\Sigma,\vec z,\vec z^+),u) \in \mathcal M^{\rm d}_{k+1,\ell}(\beta)$  is {\it source stable} if 
$(\Sigma,\vec z,\vec z^+)\in \mathcal M^{\rm d}_{k+1,\ell}$ is stable.

\begin{definition}\label{defn55}
        Let $((\Sigma,\vec z,\vec z^+),u), ((\Sigma_a,\vec z_a,\vec z^+_a),u_a)$
        ($a=1,2,3,\dots$) belong to $\mathcal M^{\rm d}_{k+1,\ell}(\beta)$.
        We assume they are all source stable.
        We say $((\Sigma_a,\vec z_a,\vec z^+_a),u_a)$ converges to 
        $((\Sigma,\vec z,\vec z^+),u)$ in the 
        stable map topology and write
        $$
        \underset{a \to \infty}{\rm lims}
         ((\Sigma_a,\vec z_a,\vec z^+_a),u_a)
        =
        ((\Sigma,\vec z,\vec z^+),u)
        $$
        if the following holds.
        For each $\epsilon$, there exist $\epsilon' > 0$ and an $\epsilon'$-trivialization of the universal family 
        at $\zeta = (\Sigma,\vec z,\vec z^+)$, denoted by $(K,\mathcal U,\Phi)$, with the following properties:
        \begin{enumerate}
                \item Let $\xi_a = (\Sigma_a,\vec z_a,\vec z^+_a)$. This sequence converges to $\xi$ in $\mathcal M^{\rm d}_{k+1,\ell}$.
                \item If $a$ is large, then for any connected component $C$ of $\pi^{-1}(\xi_a) \setminus \Phi(K \times \{\xi_a\})$, the diameter of $u_a(C)$ is smaller than $\epsilon$.
                \item We define $u'_a : K \to X$ by $u'_a(z) = u_a(\Phi(z,\xi_a))$. Then the $C^2$ distance between $u'_a$ and $u$ is smaller than $\epsilon$ for sufficiently large 
                values of $a$.
        \end{enumerate}
\end{definition}
In the same way as in (\ref{forget5151}), we define:
\begin{equation}\label{forget515122}
\frak{fg}_{\frak i_{k,k'},\frak i^+_{\ell,\ell'}} : 
\mathcal M^{\rm d}_{k'+1,\ell'}(\beta) \to \mathcal M^{\rm d}_{k+1,\ell}(\beta).
\end{equation}
We shall use this map in the case $k=k'$, $\frak i_{k,k'}(i) = i$  and 
$\frak i^+_{\ell,\ell'}(i) = i$ for $i \le \ell$. In this case, this map is denoted by the simplified notation $\frak{fg}_{\ell',\ell}$.

\begin{definition}\label{defn5656}
Let $\zeta_a = ((\Sigma_a,\vec z_a,\vec z^+_a),u_a),
\zeta = ((\Sigma,\vec z,\vec z^+),u)$ be elements of 
$\mathcal M^{\rm d}_{k+1,\ell}(\beta)$.
We say that $\zeta_a$ converges to $\zeta$ in the {\it stable map topology} and 
write 
$$
\lim_{a \to \infty} \zeta_a = \zeta
$$
if there exists 
$\zeta'_a, \zeta' \in \mathcal M^{\rm d}_{k+1,\ell'}(\beta)$
($\ell' \ge \ell$) with the following properties:
\begin{enumerate}
\item
$\zeta'_a, \zeta'$ are source stable.
\item
$\frak{fg}_{\ell,\ell'}(\zeta'_a) = \zeta_a$.
$\frak{fg}_{\ell,\ell'}(\zeta') = \zeta$.
\item
$
\underset{a \to \infty}{\rm lims}\,
\zeta'_a
=
\zeta'$ in the sense of Definition \ref{defn55}.
\end{enumerate}
\end{definition}
Related to this definition, we have the following lemma, whose proof is
given in \cite[Lemma 12.13]{fu2017}. 
It is also a consequence of \cite[Lemma 4.14]{fooo:const1}.
\begin{lemma}\label{lem57}
        Let $\zeta,\zeta_a \in \mathcal M^{\rm d}_{k+1,\ell}(\beta)$.
        We assume $
        \lim_{a \to \infty} \zeta_a = \zeta
        $.
        Suppose $\zeta' \in \mathcal M^{\rm d}_{k+1,\ell'}(\beta)$
        such that
        \begin{enumerate}
        \item[(a)]
        $\zeta'$ is source stable.
        \item[(b)]
        $\frak{fg}_{\ell,\ell'}(\zeta') = \zeta$.
        \end{enumerate}
        Then there exists $\zeta'_a\in \mathcal M^{\rm d}_{k+1,\ell'}(\beta)$ such that
        Definition \ref{defn5656} {\rm (1), (2)} and {\rm (3)} hold.
\end{lemma}
\begin{remark}
	Lemma \ref{lem57} implies that for source stable objects, Definition \ref{defn5656} coincides with Definition \ref{defn55}.
\end{remark}
We define the closure operator 
$c$ for subsets of $\mathcal M^{\rm d}_{k+1,\ell}(\beta)$
as follows.
If $A \subset \mathcal M^{\rm d}_{k+1,\ell}(\beta)$, then
$A^c$ is the set of all limits of sequences of 
elements of $A$. Here the limit is taken in the sense 
of Definition \ref{defn5656}.

\begin{lemma}
	The closure operator $c$ satisfies the Kuratowsky's axioms. 
	Namely, we have:
		{\rm (a)} ${\emptyset}^c = \emptyset$,
		{\rm (b)} $A \subseteq A^c$,
		{\rm (c)} $(A^c)^c =  A^c$,
		{\rm (d)} $(A\cup B)^c = A^c \cup B^c$.
\end{lemma}
See \cite[Lemma 12.15]{fu2017} for the proof. This closure operator allows us to define a topology on $\mathcal M^{\rm d}_{k+1,\ell}(\beta)$, called the {\it stable map topology}. In the same way as in \cite[Lemma 10.4]{FO}, we can prove that the stable map topology is Hausdorff. In the same way as in \cite[Lemma 11.1]{FO}, 
we can prove that the stable map topology is compact.
We can define the stable map topology for moduli spaces of  pseudo-holomorphic spheres or strips in the same way.


\subsection{RGW Topology} 
\label{subsec:RGW1def}

In the rest of this section, we define the RGW topology and study some of its properties. In this subsection, we define limits of sequences of pseudo-holomorphic disks, spheres and strips in the RGW topology, and then we prove sequential compactness for the RGW topology. As a part of the definition of the RGW topology, we explain when two elements of compactified moduli spaces are {\it $\epsilon$-close} to each other. The way that this notion is defined makes it clear that Kuranishi neighborhoods of points of the moduli space (to be constructed in \cite{part2:kura}) contains 
a neighborhood of that point in the moduli space. We shall also show that the compactified moduli spaces are metrizable. In particular, they are Haussdorff and their sequential compactness imply that RGW moduli spaces are compact. 
Since the definitions and proofs are mostly similar in the case of strips and spheres, we mainly focus on the case of disks and then make some comments on how they should be adapted to the case of strips and spheres.

\subsubsection{RGW Topology for Disks 1: Introducing Interior 
Marked Points}
\label{subsub:disktopirr1}
We first generalize the RGW compactification in 
Section \ref{Sec:RGW-Compactification} to the case of the moduli space 
of pseudo-holomorphic disks 
equipped with interior marked points.
We modify the moduli space 
$\mathcal M_{k+1}^{\rm reg, d}(\beta;{\bf m})$ 
of Definition \ref{defn334444}
and define $\mathcal M_{k+1,h}^{\rm reg, d}(\beta;{\bf m})$
as follows.
\begin{definition}
        The set $\mathcal M_{k+1,h}^{\rm reg, d}(\beta;{\bf m})$
        consists of isomorphism classes of
        $((\Sigma,\vec z,\vec z^+,\vec w),u)$
        such that
        (1)-(6) of Definition \ref{defn334444} and the following conditions 
        are satisfied.
        \begin{enumerate}
                \item[(i)]
                $\vec z^+ = (z^+_1,\dots,z^+_{h})$, where
                $z^+_i \in {\rm Int}\, \Sigma$. These points are distinct and away from $\vec w$.
                \item[(ii)]
                $((\Sigma,\vec z,\vec z^+\cup \vec w),u)$
                is stable in the sense of stable maps.
        \end{enumerate}
\end{definition}
We define $\mathcal M^{\rm reg, s}_h(\alpha;{\bf m})$ by modifying $\mathcal M^{\rm reg, s}(\alpha;{\bf m})$ of Definition \ref{defn333555} in a similar way. Namely, an element of $\mathcal M^{\rm reg, s}_h(\alpha;{\bf m})$ has the form  $((\Sigma,\vec z^+,\vec w),u)$ which satisfies Definition \ref{defn333555} (1)-(5), $\vec z^+$ is an $h$-tuple of distinct marked points away from $\vec w$, and $((\Sigma,\vec z^+\cup \vec w),u)$ is stable.

Let ${\bf m} = (m_0,\dots,m_{\ell})$. We define $\widetilde{\mathcal M}^{0}_{h}(\mathcal D\subset X;\alpha;{\bf m})$ as the set of strong isomorphism classes of $((\Sigma,\vec z^+,\vec w);u;s)$ such that $((\Sigma,\vec w);u;s)$ satisfies Definition \ref{defn3434} (1)-(4) and $\vec z^+$ is an $h$-tuple of additional distinct marked points disjoint from $\vec w$.
We also require stability of $((\Sigma,\vec w\cup \vec z^+);u)$ instead of Definition \ref{defn3434} (5).
\par
Using these spaces as in Section \ref{Sec:RGW-Compactification}, we may define a compactification of 
$\mathcal M_{k+1,h}^{\rm reg, d}(\beta)$
as follows. 
We consider a generalization of the notion of detailed DD-ribbon trees of homology class $\beta$, where one such detailed DD-ribbon tree $\hat R$ is required to satisfy the additional conditions:
\begin{enumerate}
        \item[(DD+.1)]
        We have a map ${\rm mk} : \{1,\dots,h\} \to C^{\rm int}_0(\hat R)$,
        that describes each element of $\vec z^+$ on which component lies.
        Define:
        $$
        h_v = \# \{i \in \{1,\dots,h\} \mid {\rm mk}(i) = v\}.
        $$
        \item[(DD+.2)] We modify stability as follows. 
        For each interior vertex $v$ of $\hat R$ we assume one of the  
        following holds.
        \begin{enumerate}
                \item
                The homology class of $v$ is nonzero.
                \item
                If the color of $v$ is $\rm s$ or $\rm{D}$, then the number 
                of edges containing $v$ plus $h_v$ is not smaller than $3$.
                \item
                If the color of $v$ is $\rm d$, then the following 
                inequality holds:
                $$
                \aligned
                &2 (\text{the number of edges of positive level}) \\
                &+ 
                (\text{the number of edges of level 0})
                + 2 h_v \ge 3.
                \endaligned
                $$
        \end{enumerate}
\end{enumerate}
We denote by $\mathcal R^+$ the pair $(\mathcal R,{\rm mk})$. We then modify the fiber product \eqref{fpro351} as follows. Firstly we need to replace $\widetilde{\mathcal M}^0(\mathcal R,v)$ with $\widetilde{\mathcal M}^0(\mathcal R^+,v)$, defined as below.
If we are in the situation of \eqref{form335}, 
where the color of $v$ is $\rm d$, then:
\begin{equation}\label{fpr315315rev22}
\widetilde{\mathcal M}^0(\mathcal R^+,v) 
= 
\mathcal M_{k+1,{h_v}}^{\rm reg, d}(\alpha(v);{\bf m}^v).
\end{equation}
Here ${\bf m}^v$ is defined as in \eqref{form335}. 
If we are in the situation of \eqref{form336}, where the color of $v$ is $\rm s$, then:
\begin{equation}\label{form3362}
	\widetilde{\mathcal M}^0(\mathcal R^+,v) = \mathcal M_{h_v}^{\rm reg, s}(\alpha(v);{\bf m}^v).
\end{equation}
Here ${\bf m}^v$ is defined as in \eqref{form336}. 
If we are in the situation of \eqref{fpr315315rev}, where the color of $v$ is $\rm{D}$, then:
\begin{equation}\label{fpr31531522}
	\widetilde{\mathcal M}^0(\mathcal R^+,v)= \widetilde{\mathcal M}^{0}_{h_v}(\mathcal D\subset X;\alpha(v);{\bf m}^v).
\end{equation}
Here ${\bf m}^v$ is defined as in \eqref{fpr315315rev}. 

Thus we modify \eqref{fpro351}, \eqref{form352minus} and \eqref{form352} to obtain, $\widetilde{\mathcal M}^0(\mathcal R^+)$, $\widehat{\mathcal M}^0(\mathcal R^+)$, and ${\mathcal M}^{0}(\mathcal R^+)$. We finally modify \eqref{form3531} to: 
\begin{equation}
	\mathcal M^{\rm RGW}_{k+1,h}(L;\beta)= \bigcup_{\mathcal R^+}{\mathcal M}^{0}(\mathcal R^+).
\end{equation}

\subsubsection{RGW Topology for Disks 2: Definition of Convergence}
\label{subsub:disktopirr2}

In our definition of the RGW topology, we use the obvious forgetful map
\begin{equation}
\frak{forget} : \mathcal M^{\rm RGW}_{k+1,h}(L;\beta)
\to \mathcal M_{k+1,h}(L;\beta)
\end{equation}
from the RGW compactification to the stable map compactification. Namely, for each factor $\widetilde{\mathcal M}^0(\mathcal R^+,v)$, we forget various parts of the information associated to that element (such as the section $s_v$ in the case that the color of $v$ is $\rm{D}$) and glue them according to the detailed DD-ribbon tree $\hat{\mathcal R}^+$. (Note that in the case that the color of $v$ is $\rm{D}$, the target of the map $u_v$ is $\mathcal D$ which is a subset of $X$.
So we can regard it as a map to $X$.)
\begin{shitu}\label{situ51212}
        We consider the following situation.
        \begin{enumerate}
                \item
                $\zeta_a = ((\Sigma(a),\vec z(a),\vec z^+(a)),u_a) 
                \in \mathcal M_{k+1,h}^{\rm reg, d}(\beta;\emptyset)$.
                \item $\zeta \in {\mathcal M}^{0}(\mathcal R^+)
                \subset\mathcal M^{\rm RGW}_{k+1,h}(L;\beta)$ where
                $\zeta = (\zeta(v);v \in C_0^{\rm int}(\hat R^+))$ and:
                \begin{enumerate}
                \item
                $\zeta(v) 
                \in \mathcal M_{k+1,{h_v}}^{\rm reg, d}(\alpha(v);{\bf m}_+^v)$
                when the color of $v$ is $\rm d$.
                We write 
                $\zeta(v)= ((\Sigma(v),\vec z(v),\vec z^+(v),\vec w(v)),u_v)$.
                \item
                $\zeta(v) 
                \in 
                \mathcal M_{h_v}^{\rm reg, s}(\alpha(v);{\bf m}^v)
                $
                when the color of $v$ is $\rm s$, and
                $\zeta(v)$ is given by $((\Sigma(v),\vec z^+(v),\vec w(v)),u_v)$.
                \item
                $\zeta(v) 
                \in \widetilde{\mathcal M}^{0}_{h_v}
                (\mathcal D\subset X;\beta(v);{\bf m}^v)$
                when the color of $v$ is $\rm{D}$.
                We write 
                $\zeta(v) = ((\Sigma(v),\vec z^+(v),\vec w(v));u_v;s_v)$.       
                \end{enumerate}
                \item
                We assume 
                $$
                \lim_{a\to \infty} \frak{forget}(\zeta_a) = \frak{forget}(\zeta).
                $$
                Here the convergence is given by the stable map topology.
                \item
                We assume $\frak{forget}(\zeta_a)$ and 
                $\frak{forget}(\zeta)$ are source stable.
        \end{enumerate}
\end{shitu}
We firstly define when a sequence $\zeta_a$ as in Situation \ref{situ51212} converges to $\zeta$ in the RGW topology. Later, we reduce the definition of the RGW topology in the general case to this special situation. Let $\xi_a$ and $\xi$ be source curves of $\frak{forget}(\zeta_a)$ and $\frak{forget}(\zeta)$, respectively. By assumption (Situation \ref{situ51212} (3) and (4)),
$\xi_a$ converges to $\xi$ in $\mathcal M^{\rm d}_{k+1,h}$. For each sufficiently small $\epsilon$, we take an $\epsilon$-trivialization of the universal family in the sense of Definition \ref{defn55}, which we denote by $(K,\mathcal U,\Phi)$. We also define:
\begin{equation}\label{Kv}
  K(v) = K \cap \Sigma(v).
\end{equation}
By Definition \ref{defn55} (3), $u'_a(z) = u_a(\Phi(z,\zeta_a))$ converges to $u$ in the $C^2$ topology on each $K(v)$. We denote by $u'_{a,v}$ the restriction of $u'_a$ 
to $K(v)$.
\par
Let $v$ be a vertex with color $\rm{D}$. Then for sufficiently large $a$, we may assume
\begin{equation} \label{inclusion}
	u'_{a,v}(z) \in \mathcal N_{\mathcal D}^{\leq c}(X).
\end{equation}
for $z\in K(v)$. Here $\mathcal N_{\mathcal D}^{\leq c}(X)$ is the set of $(p,x) \in \mathcal N_{\mathcal D}(X)$ such that $p \in \mathcal D$ and $x$ is in the fiber of $\mathcal N_{\mathcal D}(X)$ with $\Vert x\Vert \le c$, which is also identified with a regular neighborhood of $\mathcal D$ in $X$ using an almost complex structure preserving symplectomorphism. (See Subsection \ref{subsec:sympproj}.) 
We use this symplectomorphism to make sense of \eqref{inclusion}. Therefore, we can use \eqref{inclusion}, to obtain maps:
\begin{equation}\label{511form}
u'_{a,v} : K(v) \to \mathcal N_{\mathcal D}(X).
\end{equation}
The data of $\zeta$ include a section $s_v$ of $u_v^*\mathcal N_{\mathcal D}(X)$.
We use this section to obtain the map
\begin{equation}\label{512form}
U_v : \Sigma(v)\setminus \vec w(v) \to \mathcal N_{\mathcal D}(X)\setminus \mathcal D.
\end{equation}
Recall that the definition of $\mathcal M^0(\mathcal R) 
\subset \mathcal M^{\rm RGW}_{k+1,h}(L;\beta)$ involves taking the 
quotient by the $\bbC_{*}^{\vert \lambda\vert}$-action and ${\rm Aut}(\mathcal R)$.
(See \eqref{form352minus} and \eqref{form352}.)
Here we fix one representative for these quotients.

The main requirement that we need to define is how the sections $s_v$ are related to the objects $\zeta_a$. Let $\vert \lambda\vert$ be the number of levels of $\hat R$. We have the identification:
\begin{equation}\label{form513}
	\mathcal N_{\mathcal D}(X) \setminus \mathcal D=\bbR \times S(\mathcal N_{\mathcal D}(X))
\end{equation}
where $S\mathcal N_{\mathcal D}(X)$ is the unit $S^1$-bundle associated to $\mathcal N_{\mathcal D}(X)$.

\begin{definition}\label{defn513}
        Suppose we are in Situation \ref{situ51212}.
        We say that $\zeta_a$ converges to $\zeta$ in the RGW topology and write
        $$
        \underset{a \to \infty}{\rm lims}\,\zeta_a = \zeta,
        $$
        if for each $j \in \{1,\dots,\vert \lambda\vert\}$, there exists a sequence $\rho_{a,j} \in \bbC_{*}$ and for each $\epsilon>0$ there are an $\epsilon$-trivialization $(K,\mathcal U,\Phi)$
        as above and an integer $N(\epsilon)$ such that the following properties hold:
        \begin{enumerate}
                \item
                For $a\geq N(\epsilon)$ and any $v$, we have $\xi_a\in \mathcal U$ and:
                $$
                d_{C^2}\left({\rm Dil}_{1/\rho_{a,\lambda(v)}}\circ u'_{a,v},U_v\right) < \epsilon
                $$
                We use the product metric on \eqref{form513} to define the $C^2$ distance for the maps with domain $K(v)$. 
                The map ${\rm Dil}_{c}$ is given by dilation in the fiber direction of $\mathcal N_{\mathcal D}(X)$.
                In particular, it is an isometry with respect to the product metric.
                \item
                If $j < j'$ then 
                $$
                \lim_{a \to \infty} \frac{\rho_{a,j}}{\rho_{a,j'}} = \infty.
                $$
        \end{enumerate}
\end{definition}
Roughly speaking, Item (1) says that on $K(v)$, the sequence of maps $u_a$ converges to $U_v$ after scaling by $\rho_{a,\lambda(v)}$. Item (2) asserts that the distance between $u_a(K(v'))$ and $\mathcal D$ goes to zero faster than the distance between $u_a(K(v))$ and $\mathcal D$ if $\lambda(v) < \lambda(v')$. The idea of using dilation in the above definition of convergence goes back to \cite{LR}.
\begin{remark}
	As we mentioned before, there is an ambiguity of the choice of the representative of $\zeta_a$ because of the $\bbC_{*}^{\vert \lambda\vert}$-action. 
	If we take another choice, we can change $\rho_{a,j}$ to $\rho_{a,j}c_j$ where $c_j \in \bbC_{*}$. Therefore, 
	Definition \ref{defn513} is independent of the choice of representatives.
\end{remark}
In Definition \ref{defn513}, we define the RGW convergence in the case that source curves are stable. We can define the general case analogous to the stable map topology as follows. Firstly note that we can define a forgetful map
$$
\frak{fg}_{h',h} : 
\mathcal M^{\rm RGW}_{k+1,h'}(L;\beta)
\to
\mathcal M^{\rm RGW}_{k+1,h}(L;\beta)
$$
of interior marked points for $h' > h$. Namely, we forget the marked points with labels  $(h+1),\dots,h'$ and shrink the components which become unstable. In this process, the level function $\lambda$ may not be preserved because all the components in a certain level may be shrunk. In the case that this happens, we remove such a level, say $m$, and decrease the value of the level function for each vertex of level $>m$. In other words, the quasi order $\le$ is preserved by this shrinking process. In particular, this construction is similar to the forgetful map \eqref{forget515122} in the discussion of stable map topology. See also \cite[Subsection 4.2]{part3:FH} on the forgetful map of the boundary marked points.

\begin{definition}\label{defn515}
        Let $\zeta_a \in \mathcal M_{k+1,h}^{\rm reg, d}(L;\beta;\emptyset)$
        and $\zeta \in \mathcal M^{\rm RGW}_{k+1,h}(L;\beta)$.
        We say that $\zeta_a$ converges to $\zeta$ and write 
        $$
        \lim_{a \to \infty} \zeta_a = \zeta,
        $$
        if there are source stable elements $\zeta'_a \in \mathcal M_{k+1,h'}^{\rm reg, d}(L;\beta;\emptyset)$
        and $\zeta' \in \mathcal M^{\rm RGW}_{k+1,h'}(L;\beta)$ as in Situation \ref{situ51212} such that
        \begin{enumerate}
        \item
        $\frak{fg}_{h',h}(\zeta'_a) = \zeta_a$, 
        $\frak{fg}_{h',h}(\zeta') = \zeta$.
        \item
        $
        \underset{a \to \infty}{\rm lims}\,\zeta'_a = \zeta'.
        $
        \end{enumerate}
\end{definition}

\subsubsection{RGW Topology for Disks 3: Compactness}
\label{subsub:disktopirr23}
The next proposition provides the main part in showing that the space $\mathcal M^{\rm RGW}_{k+1,h}(L;\beta)$ is compact:

\begin{prop}\label{prop516}
	For any sequence $\zeta_a \in \mathcal M_{k+1,h'}^{\rm reg, d}(L;\beta)$,
	there exists a subsequence which converges in the sense of Definition \ref{defn515}.
\end{prop}

As a preparation for the proof of this proposition, we need two lemmas. The first one is a standard exponential decay result.

\begin{lemma}\label{lem5170}
	There is a positive constant $\epsilon$ and for any positive integer $k$, there are constants $C_k$, $e_k$ such that the following holds. 
	Let $A_T:= [-T,T] \times S^1$ and $u:A_T \to  \mathcal D$ be a $J_{\mathcal D}$-holomorphic map such that 
	\[
	  \int_{A_T} u^*\omega_{\mathcal D} <\epsilon.
	\]
	Then
        	\[
	  \sum_{\ell = 1}^{k}\vert (\nabla^{\ell} u)(\tau,t)\vert \le C_k e^{-e_k (T- \vert \tau\vert)}
	\]
        	for $(\tau,t) \in [-T+1,T-1] \times S^1$. Here the left hand side is the $C^{k-1}$ norm of 
	the first derivative of $u$ with respect to the standard coordinates on $[-T,T] \times S^1$.
\end{lemma}

The following lemma is a standard fact about conformal maps between annuli. 
\begin{lemma}\label{lem517}
	Let $A(c_1,c_2) $ be the annulus given as 
	\[
	  \{ z \in \bbC \mid c_1 \le \vert z\vert \le c_2\}.
	\]	
  Let $u:[-T,T] \times S^1 \to  A(c_1,c_2)$ be a holomorphic map that
        \[
          \vert u(-T,t)\vert = c_1 \hspace{1cm}\vert u(T,t)\vert = c_2.
        \]
 	Then there exist a complex number $z_0$ and a positive integer $m$ such that:
	\begin{equation} \label{shape-u-1}
          u(\tau,t) = \exp (2\pi m(\tau + \sqrt{-1}t)-z_0).
        \end{equation}
\end{lemma}

\begin{proof}[Proof of Proposition \ref{prop516}]
	Suppose $\zeta_a=((\Sigma(a),\vec z(a),\vec z^+(a)),u_a)$ is an element of the moduli space
	$\mathcal M_{k+1,h'}^{\rm reg, d}(L;\beta)$. 
	Using compactness of the stable map compactification 
	\cite[Theorem 11.1]{FO}, we may assume that 
	there exist $\zeta'_a$ and $\zeta'$ which are source stable, 
	$\zeta'_a$ converges to $\zeta'$ in the stable map topology (Definition \ref{defn55})
	and $\frak{fg}(\zeta'_a) = \zeta_a$. Here $\frak{fg}$ is the forgetful map of interior marked points. 
	Without loss of generality, we can replace $\zeta_a$ by $\zeta'_a$. We also assume that 
	$\zeta = ((\Sigma,\vec z,\vec z^+),u)$. Let $G$ be the set of irreducible components of $\Sigma$.
	We will write $\Sigma_w$ and $u_w$
	for the irreducible component associated to $w\in G$ and the restriction of the map $u$ to this component. 
	The set $G$ can be divided into two parts:
	\[
	  G=G_0\cup G_{>0},
	\] 
	where $w\in G_{>0}$ if and only if $u(\Sigma_w)$ is contained in $\mathcal D$.

	We need to find 
	$\hat\zeta\in \mathcal M^{\rm RGW}_{k+1,h}(L;\beta)$ with $\frak{forget}(\hat\zeta) = \zeta$ and show that,
	after passing to a subsequence of $\{\zeta_a\}_{a\in \mathbb N}$,
	the properties of Definition \ref{defn513} hold.
	In particular, we need to find the following objects:
	\begin{enumerate}
	\item[(I)] A meromorphic section $s_w$ of $u_w^*(\mathcal N_{\mathcal D}(X))$
		for each $w\in G_{>0}$.
	\item[(II)] A level function $\lambda:G \to \{0,1,\dots,|\lambda|\}$.
	\item[(III)] A multiplicity function $m$ associated to any intersection point of two irreducible components 
	$\Sigma_w$ and $\Sigma_{w'}$ such that $\lambda(w)\neq \lambda(w')$.
	\item[(IV)] An element $\rho_{a,j} \in \bbC_{*}$ for each $j\leq |\lambda|$ and $a$.
	\end{enumerate}
	
	We construct the objects in (I)-(IV) in an inductive way.
	Let 
	\[
	  \Sigma(0) = \bigcup_{w \in G_{0}} \Sigma_w\hspace{2cm} \Sigma'(0) = \bigcup_{w \in G_{>0}} \Sigma_w
	\]
	In other words, $\Sigma(0)$ is the union of the irreducible components $\Sigma_w$ of $\Sigma$ 
	that $u(\Sigma_w)$ is {\it not} contained in $\mathcal D$, and $\Sigma'(0)$ is the union of all the remaining components.
	The intersection $S(0):=\Sigma(0) \cap \Sigma'(0)$ consists of finitely many nodal points.
	For each point $p\in S(0)$, we fix a small neighborhood $U_p$ (resp. $U'_p$) in $\Sigma(0)$
	(resp. $\Sigma'(0)$). Let $\frak U(0)$ (resp. $\frak U'(0)$) denote the union of the open sets $U_p$ (resp. $U'_p$).
	For each $p\in S(0)$, we also fix an open neighborhood of $u(p)\in \mathcal D$ that is contained in the 
	subspace\footnote{Here we assume that after a rescaling of the hermitian metric on the 
	bundle $ \mathcal N_{\mathcal D}(X)$,
	the disk bundle $ \mathcal N_{\mathcal D}^{<1}(X)$ can be identified with a neighborhood of $\mathcal D$ in $X$.} 
	$\mathcal N_{\mathcal D}^{<1}(X)$ of $X$.
	We may assume that this open set has the form $B(1)\times V_p$ with respect to a local unitary trivialization of 
	$\mathcal N_{\mathcal D}(X)$ where $B(r)$ is the ball of radius $r$ and 
	$V_p$ is the open ball of radius $1$ in $\mathbb C^{\dim (\mathcal D)}$. 
	After modifying the choice of $U_p$, we can also assume that $u$ maps $U_p$ to $B(\sigma)\times V_p$, 
	the boundary of $U_p$ to $S(\sigma)\times V_p$, $U_p'$ to $\{0\}\times V_p$ and $p$ to $(0,0)$. 
	Here $S(r)\subset \mathbb C$ 
	is the circle of radius $r$ and $\sigma<\frac{1}{2}$ is a positive real number, independent of $p$. 
	Finally, let $K(0)$ and $K^+(0)$ denote the subspaces $\Sigma'(0) \setminus \frak U'(0)$
	and $\overline{\Sigma'(0) \cup \frak U(0)}$.
	(See Figure \ref{Figure5-2}.)
        \begin{figure}[h]
        \centering
        \includegraphics[scale=0.3]{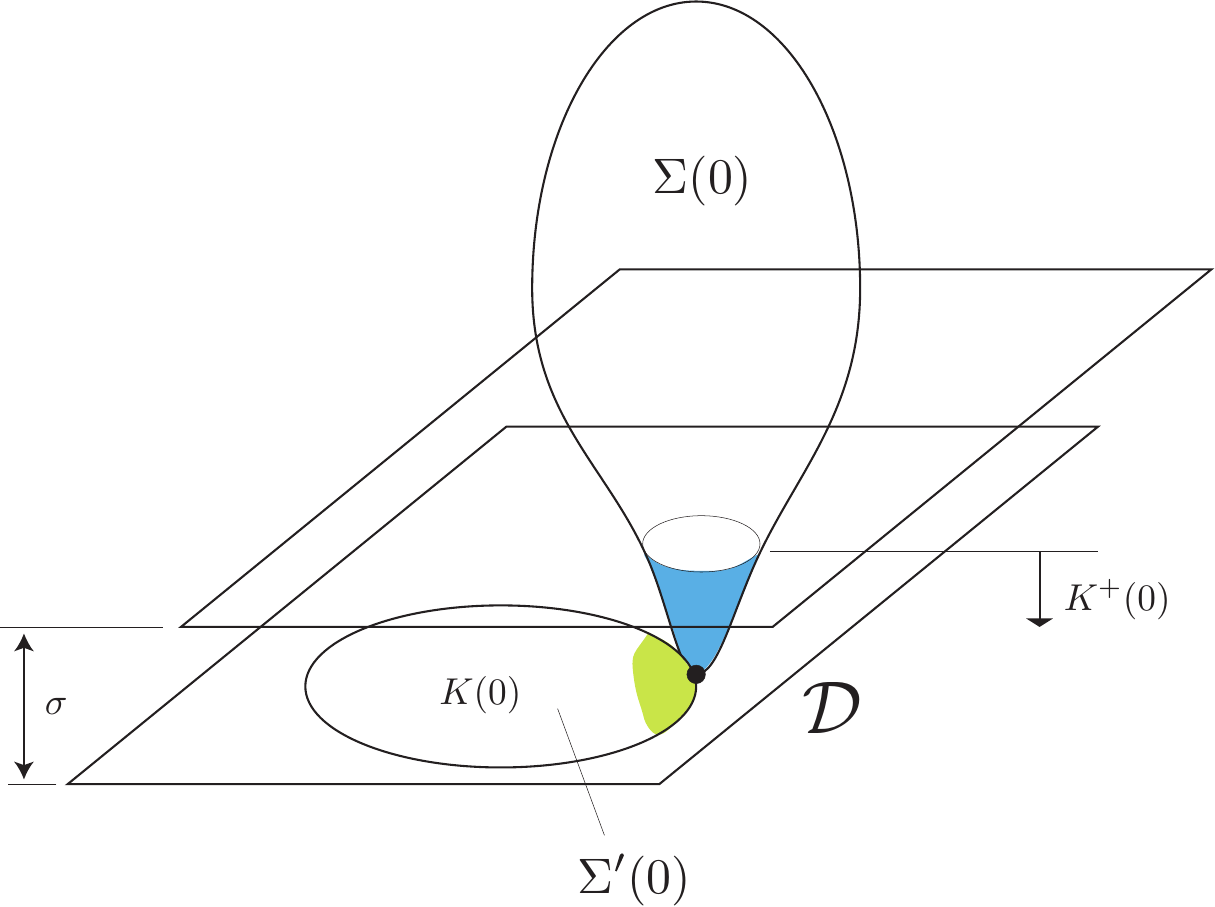}
        \caption{$K(0)$ and $K^+(0)$}
        \label{Figure5-2}
        \end{figure}	
	
	Let $U(\xi)$ be a small neighborhood of $\xi= (\Sigma,\vec z,\vec z^+)$
	in $\mathcal M^{\rm d}_{k+1,h}$. We fix $\mathcal K(0), \mathcal K^+(0) \subset \mathcal C^{\rm d}_{k+1,h}$
	with the following properties: (See Figure \ref{Figure5-3}.)
        \begin{enumerate}
                \item
                $\mathcal K(0), \mathcal K^+(0)$ are manifolds with boundary 
                and $\mathcal K(0) \setminus \partial \mathcal K(0)$, 
                $\mathcal K^+(0) \setminus \partial \mathcal K^+(0)$
                are open subsets of $\mathcal C^{\rm d}_{k+1,h}$.
                \item
                $\mathcal K(0) \cap \pi^{-1}(\xi) = K(0)$ and
                $\mathcal K^+(0) \cap \pi^{-1}(\xi) = K^+(0)$.
                \item
                $\mathcal K(0), \mathcal K^+(0)$ are closed subsets of 
                $\pi^{-1}(U(\xi))$.
        \end{enumerate}
	If $a$ is large enough, then $\xi_a = (\Sigma(a),\vec z(a),\vec z^+(a))$ is an element of $U(\xi)$.
	For such $a$, we denote $\mathcal K(0) \cap \pi^{-1}(\xi_a)$ and $\mathcal K^+(0) \cap \pi^{-1}(\xi_a)$ by $K_a(0)$
	and $K_a^+(0)$.
	For large enough values of $a$, $u_a(K_a^+(0))$ 
	is a subset of $\mathcal N_{\mathcal D}^{<2\sigma}(X)$ and the intersection 
	$u_a(K_a^+(0))\cap \mathcal N_{\mathcal D}^{<\sigma/2}(X)$ is non-empty. 
	Moreover, $K_a^+(0)\setminus K_a(0)$ has one connected component for each $p\in S(0)$.
	We write $K_{a,p}$ for the interior of this connected component which is an annulus. 
	If we only consider large enough values of $a$, then
	$u_a$ maps $K_{a,p}$ to $B(2\sigma)\times V_p$.
        \begin{figure}[h]
        \centering
        \includegraphics[scale=0.3]{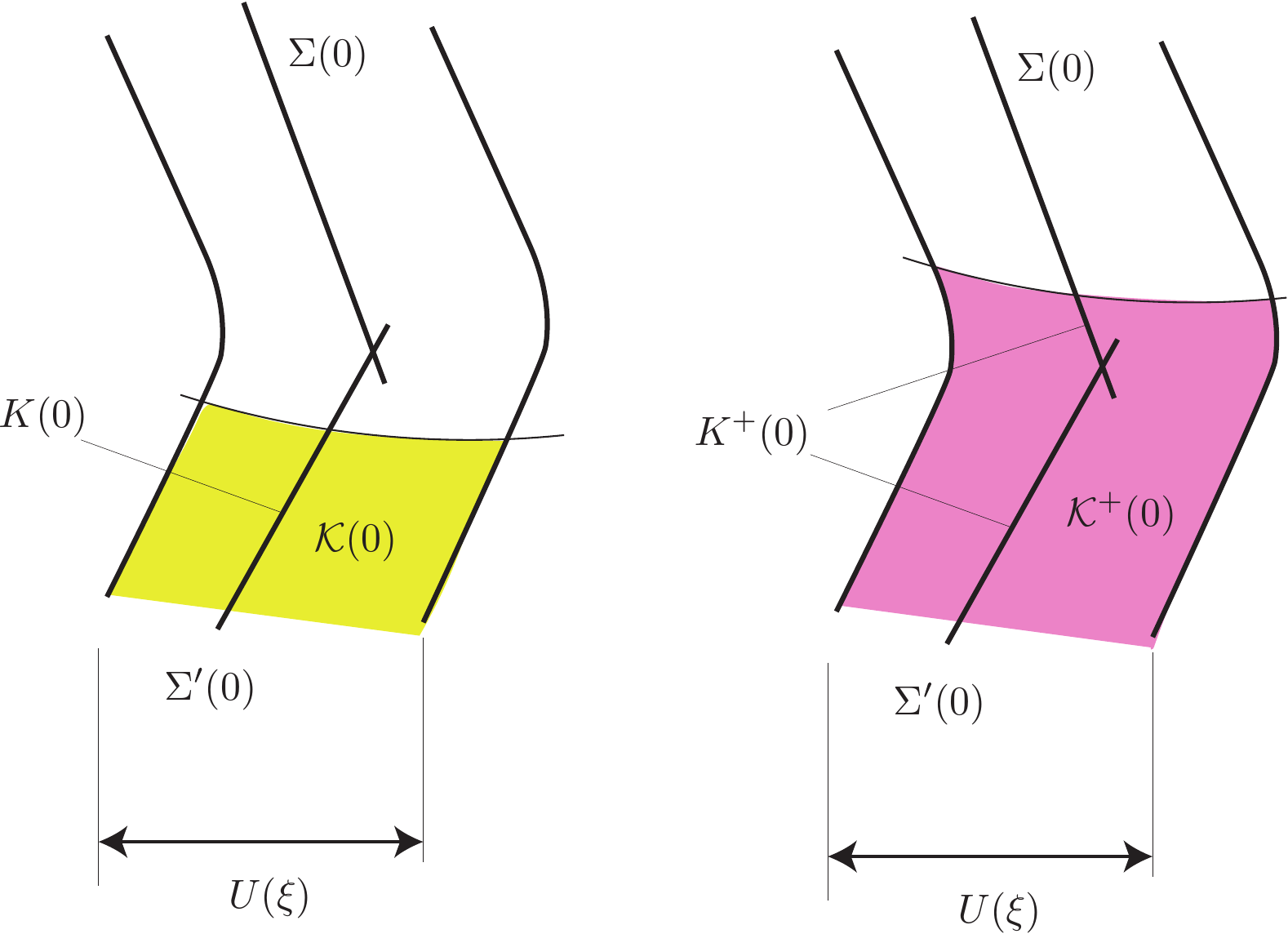}
        \caption{$\mathcal K(0)$ and $\mathcal K^+(0)$}
        \label{Figure5-3}
        \end{figure}

	We define:
	\[
	  \rho_{a,1} = \sup \{\Vert u_a(z)\Vert \mid z \in K_a(0)\}.
	\]
	where $\Vert \cdot \Vert$ denotes the fiber norm of elements of $\mathcal N_{\mathcal D}(X)$. Note that the supremum may be achieved either on the 
	boundary of $K_a(0)$ or at a point on an 
	irreducible component which does not intersect $S(0)$.
	(See Figure \ref{Figure5-4}.) 
	\begin{figure}[h]
		\centering
		\includegraphics[scale=0.5]{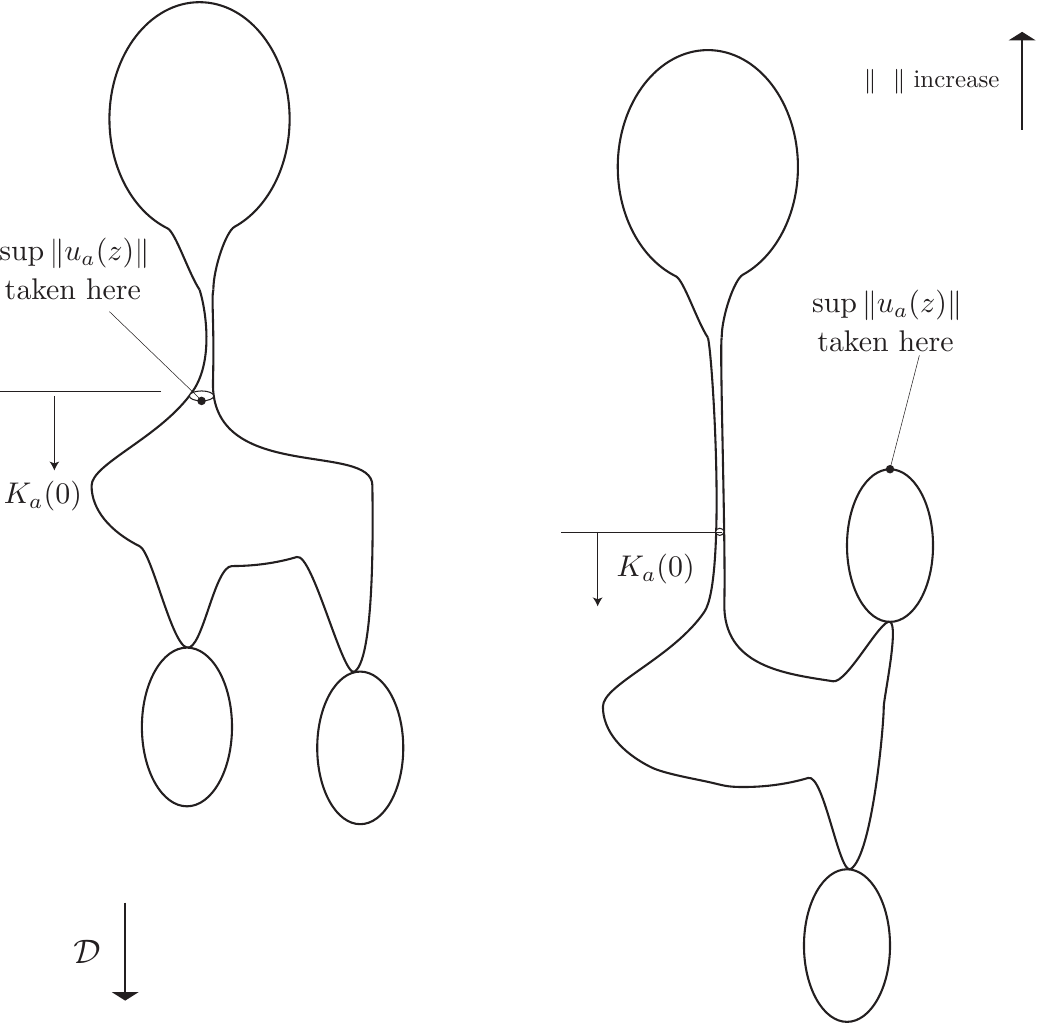}
		\caption{where $\sup$ is taken}
		\label{Figure5-4}
	\end{figure}
	By definition $\lim_{a\to \infty}\rho_{a,1} = 0$.
	The composition:
	\begin{equation}\label{fprm512}
		u'_a:={\rm Dil}_{1/\rho_{a,1}} \circ u_a : K^{+}_a(0) \to \mathcal N_{\mathcal D}^{<2\sigma/\rho_{a,1}}(X),
	\end{equation}
	is a holomorphic map, which sends $K_a(0)$ to $\mathcal N_{\mathcal D}^{\le1}(X)$.

	We pick $d_1 \in (1,2)$ and $d_{2,a} \in (\frac{\sigma}{2\rho_{a,1}}-1,\frac{\sigma}{2\rho_{a,1}})$
	such that they are regular values of the function: 
	\[
	  z\in K^{+}_a(0) \mapsto \Vert u'_a(z)\Vert.
	\]
	For $p\in S(0)$, 
	we define:
	\[
	  U_{a,p} = \{z \in K_{a,p}\mid  \Vert u'_a(z)\Vert \in [d_1,d_{2,a}] \}.
	\]
	
	The domain $U_{a,p}$ is conformally equivalent to an annulus $[-T_{a,p},T_{a,p}] \times S^1$. 
	Otherwise,  there exists a domain $D \subset K_{a,p} \setminus U_{a,p}$, 
	which is a connected component of 
		$D \subset K_{a,p} \setminus U_{a,p}$ and its boundary lies on $U_{a,p}$. (See Figure \ref{Figure5-4-2}.)
		Then one of the following occurs:
		\begin{enumerate}
			\item $\Vert u'_a(z)\Vert = d_1$ for $z \in \partial D$ and $\Vert u'_a(z)\Vert \le d_1$ for $z \in D$.
			\item $\Vert u'_a(z)\Vert = d_{2,a}$ for $z \in \partial D$ and $\Vert u'_a(z)\Vert \ge d_{2,a}$ for $z \in D$.
		\end{enumerate}
		\begin{figure}[h]
			\centering
			\includegraphics[scale=0.5]{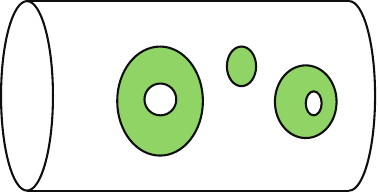}
			\caption{Subdomain D}
			\label{Figure5-4-2}
		\end{figure}
	Both cases are impossible, because of the maximum principle.	
%

	Let $\overline u'_a : [-T_{a,p},T_{a,p}] \times S^1 \to V_p$ be the composition of a conformal isomorphism from $[-T_{a,p},T_{a,p}] \times S^1$
	to $U_{a,p}$, the map $u'_a$ and 
	the projection $\mathcal N_{\mathcal D}(X) \to \mathcal D$. If $a$ is large enough, then we may assume that 
	the assumption of Lemma \ref{lem5170} is satisfied and hence we have the exponential decay result of this lemma.
	
	We next discuss almost complex structure of the pull back bundle 
	$(\overline u'_a)^*\mathcal N_{\mathcal D}(X)$. 
	The chosen unitary trivialization of $\mathcal N_{\mathcal D}(X)$ in a neighborhood 
	of $u(p)$ (which identifies this neighborhood with the direct product $D(1) \times V_p$) induces a 
	unitary trivialization of the bundle $(\overline u'_a)^*\mathcal N_{\mathcal D}(X)$.
	We may consider two connections on this bundle over $U_{a,p}$: the connection $\theta_0$ given by the trivialization and 
	the connection $\theta$ given by puling back the connection on  $\mathcal N_{\mathcal D}(X)$ that is used in the definition of 
	the almost complex structure $J$ on $X$ in a neighborhood of $\mathcal D$. (See Subsection \ref{subsec:sympproj}.)
	After fixing a large enough positive integer $k$ ($k=2$ will be sufficient for our purposes), we can apply Lemma \ref{lem5170} to obtain constants
	 $c$ and $\delta$ such that
	\begin{equation}\label{exp-dec-bound}
	  \sum_{\ell = 1}^{k}\vert \nabla^{\ell} (\theta - \theta_0)(\tau,t)\vert \le c e^{-\delta(T- \vert \tau\vert)}
	\end{equation}	 
	holds for any $(\tau,t)\in [-T_{a,p},T_{a,p}] \times S^1$. Here by letting $\epsilon$ go to zero and keeping $\delta$ fixed, 
	we may assume that $c$ is arbitrarily small.
	
	We use $\theta$ to obtain the structure of a holomorphic line bundle on 
	$(\overline u'_a)^*\mathcal N_{\mathcal D}(X)$, which we denote by $\mathcal L$.
	We denote by $\mathcal L_{0}$ the holomorphic line bundle $(\overline u'_a)^*\mathcal N_{\mathcal D}(X)$
	with direct product holomorphic structure, that is induced by $\theta_0$.
	We will construct an isomorphism of holomorphic line bundles 
	$g : \mathcal L_{0} \to \mathcal L$.
	Regarding $g$ as a $\bbC_{*}$-valued function on $[-T_{a,p},T_{a,p}] \times S^1$, we need to solve 
	\begin{equation}\label{eqfmu}
		\overline{\partial} g = (\theta_0 - \theta)^{(0,1)}g,
	\end{equation}
	In fact, if $(\theta_0 - \theta)^{(0,1)}=\alpha (d\tau-idt)$, then $g=\exp(f)$ for any solution $f$ of the equation
	\begin{equation}\label{CR-non-hg}
	  \frac{df}{d\tau}+i\frac{df}{dt}=2\alpha
	\end{equation}
	gives a solution of \eqref{eqfmu}. Lemma \ref{exp-dec-cr-non-hg} implies that there is a solution $f$ of \eqref{CR-non-hg} 
	such that the $C^k$ norm of $f$ over $[-T_{a,p}+1,T_{a,p}-1] \times S^1$ is bounded 
	by a constant $\epsilon'$ that depends only on the constants $C$ and $\delta$ in \eqref{exp-dec-bound}. In particular, 
	we have the following bound for $g=\exp(f)$:
	\begin{equation}\label{estimateg}
		\left\vert g - 1 \right\vert_{C^k([-T_{a,p}+1,T_{a,p}-1] \times S^1)} < \epsilon',
	\end{equation}
	after possibly increasing the value of $\epsilon'$.

The pseudo-holomorphic map $u'_a$ induces a holomorphic section $\hat u'_a : [-T_{a,p},T_{a,p}] \times S^1 \to \mathcal L$.
In particular, \[g^{-1} \cdot \hat u'_a: [-T_{a,p},T_{a,p}] \times S^1 \to \mathcal L_0\cong [-T_{a,p},T_{a,p}] \times S^1 \times \bbC\] is holomorphic, where the complex structure of the codomain is the direct product of the complex structures $ [-T_{a,p},T_{a,p}] \times S^1$ and $\bbC$. 
Since $g$ is a bundle isomorphism it preserves the $\bbC_{*}$ action and projection.
Although it does not preserve norm, we have the following inequality for any $v\in \mathcal L_0$:
\[1-\epsilon' < \vert g(v)/v\vert < 1+\epsilon'.\]
We define
\[
  U'_{a,b} = \left\{ z \in [-T_{a,p}+1,T_{a,p}-1] \mid (g^{-1} \circ \hat u'_a)(z) \in [(1+\epsilon')d_1,(1-\epsilon')d_{2,a}]\right\}.
\]
In the same way as before, $U'_{a,b}$ is conformal to an annulus $[-T'_{a,p},T'_{a,p}] \times S^1$.
	 
	Compose a conformal isomorphism from $[-T'_{a,p},T'_{a,p}] \times S^1$ to $U'_{a,p}$ with the map $g^{-1}\circ u_a'$ to define
	\[
	  u''_{a,p}=(u''_{a,p,1},u''_{a,p,2}) : [-T'_{a,p},T'_{a,p}] \times S^1 \to A((1+\epsilon')d_1,
	  (1-\epsilon')d_{2,a}) \times V_p.
	\]
	Lemma \ref{lem517} asserts that there are a positive integer $m_{a,p}$ and a complex number $z_{a,p}$
	such that
	\begin{equation}\label{formula-u-cyl-comp}
	  u''_{a,p,1}(\tau,t) = \exp (2\pi m_{a,p}(\tau + \sqrt{-1}t) - z_{a,p}).
	\end{equation}
	The convergence of $\zeta_a$ to $\zeta$ in the stable map topology implies that 
	$m_{a,j}$ is independent of $a$ for sufficiently large values of $a$. In fact, this common value, denoted by $m_p$, is 
	the order of tangency of $u(\Sigma(0))$ to $\mathcal D$ at $p$. Consequently, $\lim_{a \to \infty} d_{2,a} = \infty$ implies 
	that
	\[
	  \lim_{a\to \infty}T'_{a,p} = \infty. 
	\]

	The functions $g \cdot u''_{a,p}$ can be perturbed and extended into functions 
	$v_{a,p}:[-T_{a,p},\infty) \times S^1\to \bbC\times V_p$.
	Firstly let $\chi:\bbR \to [0,1]$ be a smooth function satisfying the following properties:
	\begin{itemize}
		\item[(i)] $\chi(\tau) = 1$ for $\tau \in (-\infty,0]$;
		\item[(ii)] $\chi(\tau) = 0$ on $\tau \in [1,\infty)$.
	\end{itemize}
	Then we define
	\[
	  v_{a,p}(\tau,t) =((\chi(\tau)(g-1)+1)u''_{a,p,1}(\tau,t) ,\chi(\tau)(u''_{a,p,2}(\tau,t) - u''_{a,p,2}(0,0)) + u''_{a,p,2}(0,0))
	\]
	if $(t,\tau)\in [-T'_{a,p},T'_{a,p}] \times S^1$, and
	\[
	  v_{a,p}(\tau,t) =(\exp (2\pi m_{a,p}(\tau + \sqrt{-1}t) - z_{a,p}),u''_{a,p,1}(\tau,t),u''_{a,p,2}(0,0))
	\]
	if $(t,\tau)\in [1,\infty) \times S^1$. We may regard $v_{a,p}$ as a map into ${\bf P}(\mathcal N_{\mathcal D}(X)\oplus \bbC)$	by identifying $\bbC\times V_p$ with the open subset of 
	$\mathcal N_{\mathcal D}(X) \subset {\bf P}(\mathcal N_{\mathcal D}(X)\oplus \bbC)$,
	given by fibers of $\mathcal N_{\mathcal D}(X)$ over $V_p$.
	Lemma \ref{lem517} and the shape of the symplectic form $\omega_{\bf P}$ on ${\bf P}(\mathcal N_{\mathcal D}(X)\oplus \bbC)$
	given in Subsection \ref{subsec:sympproj}. implies that 
	there is a positive constant $C$, independent of $a$, such that:
	\begin{equation}\label{cyl-ineq}
	  \int_{[0,1] \times S^1} (v_{a,p})^* \omega_{\bf P} > -C\hspace{1cm}
	  \int_{[0,1] \times S^1} |dv_{a,p}|^2 < C.
	\end{equation}
	where $[0,1] \times S^1\subset [-T_{a,p},\infty) \times S^1$ is the cylinder where the map $u''_{a,p,2}$ 
	is not necessarily holomorphic anymore.
	The norm of the differential $dv_{a,p}$ is defined using the metric associated to $\omega_{\bf P}$ and $J_{\bf P}$.
	
	Suppose $\Sigma(a;1)$ is a sphere given by gluing the cylinders $[-T_{a,p},\infty) \times S^1$ to $K_a^+(0)$ and then adding 
	a point for each $p$. The maps $u_a'$ and $v_{a,p}$ can be used to define a map
	\[
	  u'''_{a} : \Sigma(a;1) \to {\bf P}(\mathcal N_{\mathcal D}(X)\oplus \bbC).
	\]
	The map $u'''_{a}$ is equal to $u_a'$ on the subspace of $K_a^+(0)$ where $\Vert u_a' \Vert$ is 
	at most $d_1$ and is equal to $v_{a,p}$ on the cylinder $[-T_{a,p},\infty) \times S^1$. In particular, $u'''_{a}$
	is holomorphic except on the cylinders $[0,1]\times S^1 \subset [-T_{a,p},\infty) \times S^1$. 

	
	The convergence of $\zeta_a$ to $\zeta$ in the stable map topology implies that $\pi\circ u'''_{a} $ is convergent to 
	$u|_{\Sigma'(0)}$. This observation and the behavior of $u'''_{a}$ on each cylinder $[-T_{a,p},\infty) \times S^1$
	allows us to conclude that the maps $u'''_{a}$
	represent the same homology class in ${\bf P}(\mathcal N_{\mathcal D}(X)\oplus \bbC)$. Thus we can use 
	\eqref{cyl-ineq} to conclude that there is a uniform constant $M$ such that
	\[
	  \int_{\Sigma'(a,1)} |du'''_{a}|^2<M
	\]
	where $|\cdot |$ is defined with respect to $\omega_{\bf P}$ and $J_{\bf P}$.
	Gromov compactness implies that after passing to a subsequence, the sequence $(\Sigma(a;1),u'''_{a,1})$ 
	together with the marked points converges as a stable map. 
	Let $u_{\infty,1} : \Sigma(\infty,1) \to {\bf P}(\mathcal N_{\mathcal D}(X) \oplus \bbC)$ be the limit.

	Any irreducible component $\Sigma(\infty,1)_h$ of $\Sigma(\infty,1)$ such that $u_{\infty,1}(\Sigma(\infty,1)_h)$ is not 		
	contained in a fiber of ${\bf P}(\mathcal N_{\mathcal D}(X) \oplus \bbC)$ corresponds to an irreducible component of 
	$\Sigma_w$ with $w\in G_{>0}$. This follows from convergence of $\zeta_a$ to $\zeta$ in the stable map topology. 
	Therefore, if $\Sigma(\infty,1)_h$ does not correspond to an irreducible component of $\Sigma$, then $(\pi \circ u_{\infty,1})(\Sigma(\infty,1)_h)$ 
	should be a point where $\pi : {\bf P}
	(\mathcal N_{\mathcal D}(X) \oplus \bbC) \to \mathcal D$ is the projection map.
	Convergence of $\zeta_a$ to $\zeta$ in the stable map topology also implies that any irreducible component $\Sigma_w$ with $w\in G_{>0}$ is in
		correspondence with a unique component $\Sigma(\infty,1)_h$ of $\Sigma(\infty,1)$.

%
	Let $\Sigma(\infty,1)_h$ be an irreducible component of $\Sigma(\infty,1)$ with $u_{\infty,1}(\Sigma(\infty,1)_h)$ being not contained in a fiber of 
	${\bf P}(\mathcal N_{\mathcal D}(X) \oplus \bbC)$. Let $\Sigma_w$, for $w\in G_{>0}$, be the corresponding component. 
	We also assume that $u_{\infty,1}(\Sigma(\infty,1)_h)$
	is not included in the zero section $\mathcal D_0$ of ${\bf P}(\mathcal N_{\mathcal D}(X) \oplus \bbC)$. Then we 
	define the level of $w$ to be $1$. We can also use the restriction of $u_{\infty,1}$ to 
	$\Sigma_w$ to define a meromorphic section $s_w$ of $\mathcal N_{\mathcal D}(X)$.
	Let $G_1$ be the set of all elements $w$ of $G$ with $\lambda(w)=1$. 
	Our choices of the constants $\rho_{a,1}$ guarantee that $G_1$ is not an empty set.
	
	We also write $G_{>1}$ for $G_{>0}\setminus G_1$.
	If $w\in G_{>1}$, then there is a corresponding component $\Sigma(\infty,1)_h$ of $\Sigma(\infty,1)$ whose image under $u_{\infty,1}$
	is contained in the section $\mathcal D_0$. Let $w'\in G_1$ be chosen such that $\Sigma_{w}$ 
	 and $\Sigma_{w'}$ share a nodal point $p$. Then the image of $\Sigma_{w'}$ by the map $u_{\infty,1}$ 
	 intersects $\mathcal D_0$ at the point $p$. The value of the multiplicity function $m$ at the point $p$ is defined to be the 
	 order of tangency of this intersection. Thus we partially obtained the required objects in (I)-(IV).
	 We repeat a similar construction to obtain a partition of $G_{>0}$ as follows:
	 \[
	   G_{>0}=G_1\sqcup G_2 \sqcup \dots
	 \]
	Since each $G_i$ is non-empty, this process terminates in finitely many steps. Therefore, we construct the objects claimed in 
	(I)-(IV). Using them, it is straightforward to construct an element $\hat \zeta \in \mathcal M^{\rm RGW}_{k+1,h}(L;\beta)$
	such that $\zeta_a$ are convergent to $\hat \zeta$.
\end{proof}


\begin{lemma}\label{exp-dec-cr-non-hg}
	For any integer $k$, there is a constant $c$ such that the following claim holds. Suppose $h:[-T,T]\times S^1\to \bbC$ satisfies
	\begin{equation}\label{exp-dec-bound-h}
	  \sum_{\ell = 1}^{k}\vert \nabla^{\ell} h(\tau,t)\vert \le C e^{-\delta(T- \vert \tau\vert)}.
	\end{equation}	 
	Then there is a function $f:[-T,T]\times S^1\to \bbC$ such that 
	\begin{equation}\label{CR-non-hg-2}
		\frac{df}{d\tau}+i\frac{df}{dt}=h
	\end{equation}
	and 
	\begin{equation}\label{exp-dec-bound-f}
		\sum_{\ell = 1}^{k}\vert \nabla^{\ell} f(\tau,t)\vert \le c\cdot C^2(1+\frac{1}{\delta^2}),
	\end{equation}	
	for any $\tau\in [-T+1,T-1]$.
\end{lemma}
\begin{proof}
	Throughout the proof $c$ is a positive constant which might increase from each line to the next one.
	Suppose that we have the following Fourier series presentation for the function $h$:
	\[
	  h(\tau,t)=\sum_{n=-\infty}^\infty \varphi_n(\tau)e^{2\pi i n t}.
	\]
	The assumption implies that for any $\tau$, we have
	\begin{equation}\label{L2-bound}
	  \sum_{n=-\infty}^\infty \vert\varphi_n(\tau)\vert^2\leq c\cdot C^2 e^{-2\delta(T- \vert \tau\vert)}
	\end{equation}
	Then we can take the function $f$ to be
	\[
	  f(\tau,t)=\sum_{n=-\infty}^\infty \psi_n(\tau)e^{2\pi i n t}
	\]
	where 
	\[
	  \psi_n(\tau):=\left\{
	  \begin{array}{ll}
		\displaystyle-\int_{\tau}^T e^{2\pi n(\tau-s)}\varphi_n(s)ds&n>0,\\
		\\
	  	\displaystyle\int_0^\tau\varphi_0(s)ds&n=0,\\
		\\
		\displaystyle\int_{-T}^\tau e^{2\pi n(\tau-s)}\varphi_n(s)ds&n<0.\\
	  \end{array}
	  \right.
	\]
	One can easily see that for non-zero integer $n$, we have
	\[
	  \Vert \psi_n\Vert_{L^2([-T,T])}^2\leq \frac{1}{4\pi |n|}\Vert \varphi_n\Vert_{L^2([-T,T])}^2,\hspace{1cm}
	\]
	and as a consequence of \eqref{L2-bound}, we have
	\[
	  	  \Vert \psi_0\Vert_{L^2([-T,T])}^2\leq c_0\frac{C^2}{\delta^2}.  
	\]
	In summary, we have
	\[
	  \Vert f \Vert_{L^2}\leq c_0C^2(1+\frac{1}{\delta^2})
	\]
	Now for any point $(\tau,t)\in [-T+1,T-1]$, the cylinder $(-\tau-1,\tau+1)\times S^1$ is a subspace of $[-T,T]\times S^1$, and we can use elliptic 
	regularity to obtain the desired claim.
\end{proof}

\subsubsection{RGW Topology : Strips and Spheres}
\label{subsub:otheropirr}

In this part, we give the definition of the RGW topology in several 
other cases. Since the definition is similar to the case of discs, we 
skip the details of the construction.

We first consider the case of strips. Let $L_0$, $L_1$ be Lagrangian submanifolds of $X \setminus D$ which intersect transversally. Let $p,q \in L_0\cap L_1$, and form $\mathcal M_{k_1,k_0}^{\rm reg}(L_1,L_0;p,q;\beta;\emptyset)$ as in Definition \ref{defn33strip}. We also consider the case that the pseudo-holomorphic strips have $h$ interior marked points and denote the resulting moduli space by $\mathcal M_{k_1,k_0;h}^{\rm reg}(L_1,L_0;p,q;\beta;\emptyset)$. The stable map compactification of this space, denoted by $\mathcal M_{k_1,k_0;h} (L_1,L_0;p,q;\beta;\emptyset)$, is defined in \cite[Subsection 3.8.8]{fooobook}. 
\begin{figure}[h]
\centering
\includegraphics[scale=0.5]{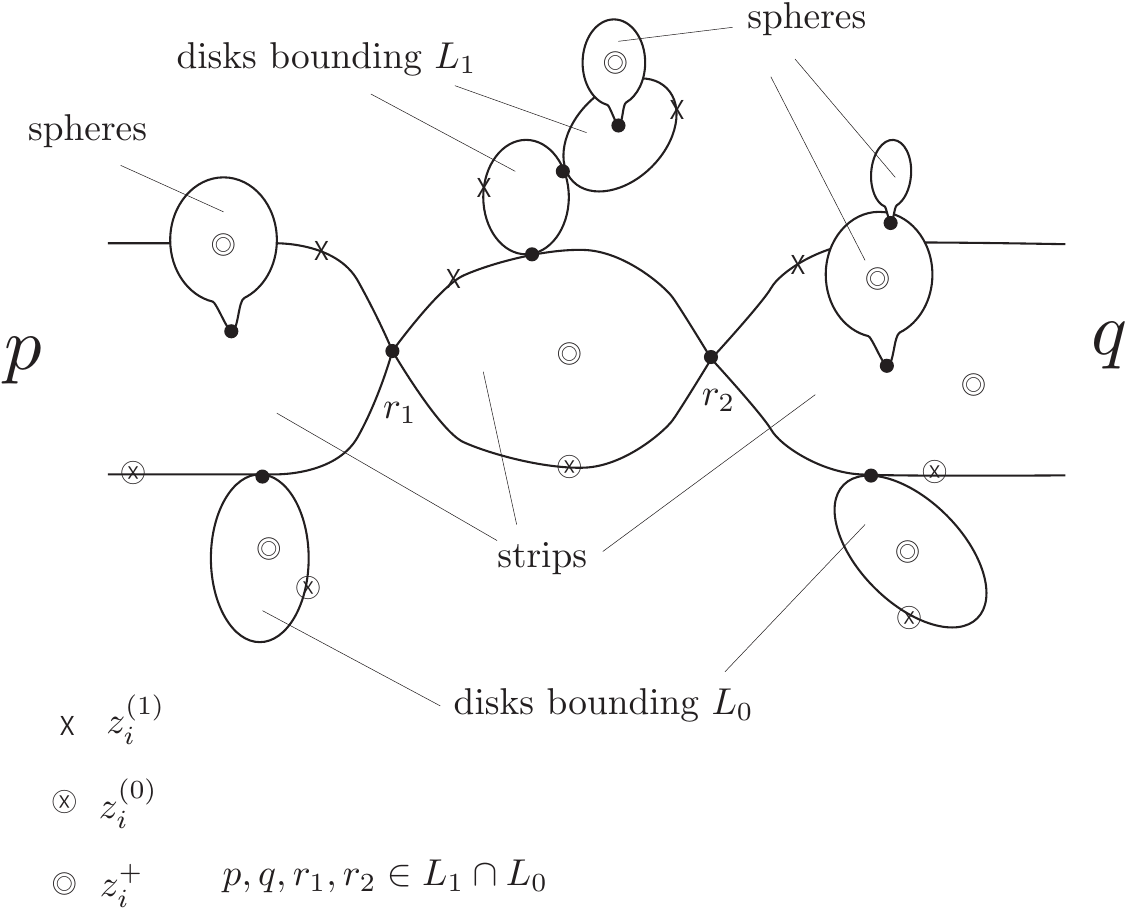}
\caption{$((\Sigma,\vec z^{(1)},\vec z^{(0)},\vec z^+),u)$}
\label{5figurestrip}
\end{figure}

An element of the compactification is represented by $((\Sigma,\vec z^{(1)},\vec z^{(0)},\vec z^+),u)$ where $\Sigma$ is the union of a line of strips, trees of disks attached to the sides of the strips, and trees of spheres attached to the interior of strips or disks. (See Figure \ref{5figurestrip}.) The points $\vec z^{(1)} = (z^{(1)}_1,\dots,z^{(1)}_{k_1})$ are boundary marked points on one side of the boundary of strips, and $\vec z^{(0)} = (z^{(0)}_1,\dots,z^{(0)}_{k_0})$ are boundary marked points on the other side of the boundary of strips. (See Figure \ref{5figurestrip}.) Moreover, $\vec z^+$ is an $h$-tuple of interior marked points, which lie on the interior of the disks, the strips and the spheres of $\Sigma$. Such configurations are described by an SD-tree $\mathcal R$ as in Figure \ref{FiguregraphR}.
The only difference is that we also need to include a marking map 
${\rm mk} : \{1,\dots,h\} \to C^{\rm ins}_0(\mathcal R)$.
As in the previous subsection, we denote the pair $(\mathcal R,{\rm mk})$ by $\mathcal R^+$. The definition of the stable map topology is essentially the same as Definitions \ref{defn55} and \ref{defn5656}, and we omit it here.

\begin{shitu}\label{situ51212rev}
        We consider the following situation.
        \begin{enumerate}
                \item
                $\zeta_a = ((\Sigma(a),\vec z^{(1)}(a),\vec z^{(0)}(a),\vec z^+(a)),u_a) 
                \in \mathcal M_{k_1,k_0;h}^{\rm reg}
                (L_1,L_0;p,q;\beta;\emptyset)$.
                \item $\zeta = (\zeta(v);v \in C_0^{\rm int}(\hat R)) \in {\mathcal M}^{0}(\mathcal R^+)
                \subset\mathcal M^{\rm RGW}_{k_1,k_0;h}(L_1,L_0;p,q;\beta)$.
                with the following components: 
       		\begin{enumerate}
                        \item If $c(v)={\rm str}$, then
                        		$\zeta(v) \in \mathcal M_{k_1,k_0;h_v}^{\rm reg} (L_1,L_0;r(v),r'(v);\beta;{\bf m}_+^v)$.
	                        Here $r(v)$, $r'(v) \in L_1\cap L_0$ are points assigned to the two edges of $C$ containing $v$.
			\item If  $c(v)={\rm d}_j$ ($j=0,1$), then 
				$\zeta(v)\in \mathcal M_{k+1,{h_v}}^{\rm reg, d}(L_j;\alpha(v);{\bf m}_+^v)$.
				We write $\zeta(v)= ((\Sigma(v),\vec z(v),\vec z^+(v),\vec w(v)),u_v)$.
                        \item If $c(v)=\rm s$, then $\zeta(v)$ is an element of $\mathcal M_{h_v}^{\rm reg, s}(\alpha(v);{\bf m}^v)$.
	                        We write $\zeta(v)= ((\Sigma(v),\vec z(v),\vec z^+(v),\vec w(v)),u_v)$.
			\item If $c(v)=\rm{D}$, then 
                        		$\zeta(v) \in \widetilde{\mathcal M}^{0}_{h_v}(\mathcal D\subset X;\beta(v);{\bf m}^v)$.
	                        We write $\zeta(v) = ((\Sigma(v),\vec z^+(v),\vec w(v));u_v;s_v)$.	
                \end{enumerate}
                \item We assume 
                		\[
	                  \lim_{a\to \infty} \frak{forget}(\zeta_a) = \frak{forget}(\zeta).
			\]
			Here the convergence is by the stable map topology.
                \item We assume $\frak{forget}(\zeta_a)$ and $\frak{forget}(\zeta)$ are source stable.
        \end{enumerate}
\end{shitu}
We take an $\epsilon$-trivialization $(K,\mathcal U,\Phi)$ in the same way as in Definition \ref{subsec:RGW1def} and obtain the map $u'_{a,v}$ in the same way as in \eqref{511form} for $a \in \bbZ_+$ and a vertex $v$ with color $\text{D}$. We can also define a map $U_v : \Sigma_v\setminus \vec w(v) \to \mathcal N_{\mathcal D}(X)\setminus \mathcal D$ for any such vertex $v$ in the same way as in \eqref{512form}. Now we define the notion of convergence ${\rm lims}_{a \to \infty}\,\zeta_a = \zeta$ analogous to Definition \ref{defn513}, and generalize this definition as in  Definition \ref{defn515} to define ${\rm lim}_{a \to \infty}\,\zeta_a = \zeta$ using forgetful maps. Finally we can prove the analogue of Proposition \ref{prop516} for strips using a similar argument.

We next consider $\mathcal M_{h}^{\rm reg, s}(\alpha;\emptyset)$ for $\alpha \in \Pi_2(X)$. We did not define the RGW compactification of this space in Section \ref{Sec:RGW-Compactification}. However, by following an essentially the same construction as in the definition of $\mathcal M_{k+1,h}^{\rm  RGW}(L;\beta)$, we can define this RGW compactification denoted by $\mathcal M_{h}^{\rm  RGW}(\alpha;\emptyset)$. The main difference is that the root vertex has color $\rm s$ instead of $\rm d$. The definition of the RGW topology is also similar.
We can prove an the analogue of Proposition \ref{prop516} in this case, too.

We finally define the RGW convergence of a sequence of elements of $\mathcal M^0(\mathcal D\subset X;\alpha;\emptyset)$ to an element of $\mathcal M(\mathcal D\subset X;\alpha;\emptyset)$. Let $X'={\bf P}(\mathcal N_{\mathcal D}(X) \oplus \bbC)$ and $\mathcal D'=\mathcal D_0\cup \mathcal D_\infty$. Then $\mathcal M^0(\mathcal D\subset X;\alpha;\emptyset)$ (resp. $\mathcal M(\mathcal D\subset X;\alpha;\emptyset)$) can be identified with the quotient of the moduli space $\mathcal M^{\rm reg, s}(\hat \alpha;\emptyset)$ associated to the pair $(X',\mathcal D')$ (resp. $\mathcal M^{\rm RGW}(\hat \alpha;\emptyset)$ associated to the pair $(X',\mathcal D')$) with respect to the obvious $\bbC_{*}$ action. Here $\hat \alpha$ is defined as in Subsection \ref{subsec:nbdofdivisor}. Given a sequence $\zeta_a\in\mathcal M^0(\mathcal D\subset X;\alpha;\emptyset)$, represented by elements $\widetilde \zeta_a \in\mathcal M^{\rm reg, s}(\hat \alpha;\emptyset)$, we say $\zeta_a$ converges to $\zeta\in\mathcal M(\mathcal D\subset X;\alpha;\emptyset)$, represented by $\widetilde \zeta \in\mathcal M^{\rm RGW}(\hat \alpha;\emptyset)$, if there are complex numbers $z_a$ such that $z_a\cdot \zeta_a$ converges to $\zeta$ with respect to the notion of the convergence of the last paragraph. As in the previous cases, an analogue of Proposition \ref{prop516} holds in this case. We can also generalize the discussion of this paragraph to the case that we include interior marked points and define the moduli space $\mathcal M_h(\mathcal D\subset X;\alpha;\emptyset)$.

\begin{remark}
	In all three cases that we have discussed so far, we can replace $\emptyset$ with $\bf m$ without much change. That is to say, we can discuss convergence of a sequence of holomorphic maps which 
	intersects the divisor in a prescribed way in all three cases and prove the analogues of Proposition \ref{prop516}.
\end{remark}

\subsubsection{RGW topology: General Case}
\label{subsec:RGW2def}
We are ready to define the RGW convergence in the general case. We focus on the moduli space $\mathcal M_{k+1,h}^{\rm  RGW}(L;\beta)$. A similar discussion applies to the case of strips.

Suppose $\mathcal R^+ = (\mathcal R,{\rm mk})$ and $(\mathcal R')^+ = (\mathcal R',{\rm mk}')$ are DD-ribbon trees with interior marked points. We can define level shrinking and level $0$ edge shrinking of such DD-ribbons as in Section \ref{Sec:RGW-Compactification}. We write $(\mathcal R')^+\geq \mathcal R^+$ if $(\mathcal R')^+$ is obtained from $\mathcal R^+$ by finitely many iterations of these operations. 


\begin{lemma}\label{lem5421}
	Suppose $(\mathcal R')^+\geq \mathcal R^+$, and $\vert \lambda\vert$, $\vert \lambda'\vert$ are the total numbers of levels of $\mathcal R^+$, $(\mathcal R')^+$. 
	Suppose $\hat R,\, \hat R'$ denote the detailed trees of $\mathcal R^+,\,(\mathcal R')^+$.
	Then there are a surjective and non-decreasing map 
	${\rm levsh} : \{0,1,\dots,\vert \lambda\vert\} \to \{0,1,\dots,\vert 	\lambda'\vert\}$
        and a surjective simplicial map ${\rm treesh} : \hat R \to  \hat R'$
        with the following properties:
        \begin{enumerate}
                \item If $v \in C^{\rm int}_0(\hat R)$, then
	                \[
			  \lambda'({\rm treesh}(v)) = {\rm levsh}(\lambda(v)).
			\]
                \item The inverse image of each vertex by ${\rm treesh}$ is connected.
                \item Let $v \in C^{\rm int}_0(\hat R')$ be a vertex of level $0$. If ${\rm treesh}^{-1}(v)$ 
	                contains a vertex of color $\rm d$, then the color of $v$ is $\rm d$. Otherwise, the color of $v$ is $\rm s$.
                \item ${\rm treesh}$ is bijective on the subset of exterior vertices and exterior edges.
        \end{enumerate}
\end{lemma}
\begin{proof}
        This is obvious in the case of $(i,i+1)$ level shrinking
        and shrinking of a single level $0$ edge. 
        For a composition of level shrinking and level $0$ edge shrinking operations, we can also use the composition of the corresponding maps ${\rm treesh}$
        and ${\rm levsh}$. Moreover, the properties in (1)-(4) 
        are preserved by the composition. 
\end{proof}
Now we consider the following situation:
\begin{shitu}\label{sisisi5.22}
	Suppose $(\mathcal R')^+ \ge \mathcal R^+$ and ${\rm treesh}$,
	${\rm levsh}$ are as in Lemma \ref{lem5421}. Let $\zeta_a \in \mathcal M^0((\mathcal R')^+)$ be a sequence and 
	$\zeta \in \mathcal M^0(\mathcal R^+)$. Let $\widetilde{\zeta}_a \in \widetilde{\mathcal M}^0((\mathcal R')^+)$
	and $\widetilde{\zeta} \in \widetilde{\mathcal M}^0(\mathcal R^+)$
	denote elements representing $\zeta_a$ and $\zeta$, respectively. We write
	\[
	   \widetilde{\zeta} = (\zeta(v)),\quad \widetilde{\zeta}_a = (\zeta_{a}(v')),
	\]
	where $v$ belongs to $C_0^{\rm int}(\hat R)$ (resp. $v'$ belongs to $C_0^{\rm int}(\hat R')$)
	and $\zeta(v) \in \widetilde{\mathcal M}^0(\mathcal R^+,v)$
	(resp. $\zeta_{a}(v') \in \widetilde{\mathcal M}^0((\mathcal R')^+,v')$).
	We require that $\frak{forget}({\zeta}_a)$ converges to $\frak{forget}({\zeta})$ 
	in the stable map topology. Furthermore, we assume ${\zeta}_a$, ${\zeta}$ are source stable.
\end{shitu}
For ${\zeta}_a$ and ${\zeta}$ as in Situation \ref{sisisi5.22}, we wish to explain when ${\zeta}_a$ converges to ${\zeta}$ in the RGW topology. We firstly need to introduce some notations:
\begin{enumerate}
	\item
            	For $v\in C_0^{\rm int}(\hat R)$, if $c(v)=\rm{d}$, then 
            	$\zeta(v) \in \mathcal M_{k+1,{h_v}}^{\rm reg, d}(\alpha(v);{\bf m}^v)$, and
            	we write $\zeta(v)= ((\Sigma(v),\vec z(v),\vec z^+(v),\vec w(v)),u_v)$.
            	For $v'\in C_0^{\rm int}(\hat R')$, if $c(v')=\rm{d}$, then 
            	$\zeta_a(v') \in \mathcal M_{k+1,{h_{v'}}}^{\rm reg, d}(\alpha(v');{\bf m}^{v'})$, and
            	we write $\zeta_a(v')= ((\Sigma_a(v'),\vec z_a(v'),\vec z_a^+(v'),\vec w_a(v')),u_{a,v'})$.
	\item For $v\in C_0^{\rm int}(\hat R)$, if $c(v)=\rm{s}$, then 
            	$\zeta(v) \in \mathcal M_{h_v}^{\rm reg, s}(\alpha(v);{\bf m}^v)$, and
            	we write $\zeta(v)= ((\Sigma(v),\vec z(v),\vec z^+(v),\vec w(v)),u_v)$.
            	For $v'\in C_0^{\rm int}(\hat R')$, if $c(v')=\rm{s}$, then 
            	$\zeta_a(v') \in \mathcal M_{h_{v'}}^{\rm reg, s}(\alpha(v');{\bf m}^{v'})$, and
            	we write $\zeta_a(v')= ((\Sigma_a(v'),\vec z_a(v'),\vec z^+_a(v'),\vec w_a(v')),u_{a,v'})$.
	\item For $v\in C_0^{\rm int}(\hat R)$, if $c(v)=\rm{D}$, then 
		$\zeta(v) \in \widetilde{\mathcal M}^{0}_{h_v}(\mathcal D\subset X;\beta(v);{\bf m}^v)$, and
            	we write $\zeta(v) = ((\Sigma(v),\vec z^+(v),\vec w(v));u_v;s_v)$.
            	For $v'\in C_0^{\rm int}(\hat R')$, if $c(v')=\rm{D}$, then 
            	$\zeta_a(v')$ is an element of $\widetilde{\mathcal M}^{0}_{h_{v'}}(\mathcal D\subset X;\beta(v');{\bf m}^{v'})$, and
            	we write $\zeta_a(v') = ((\Sigma_a(v'),\vec z_a^+(v'),\vec w_a(v'));u_{a,v'};s_{a,v'})$.
	\item For $v' \in C_0^{\rm int}(\hat R')$, we define:
		\[
		  \hat R(v') = {\rm treesh}^{-1}(v') \subset \hat R.
		\]
\end{enumerate}

For a sufficiently small $\epsilon$, we take an $\epsilon$-trivialization $(K,\mathcal U,\Phi)$ of the universal family at the source curve $\xi$ of $\zeta$ in the sense of Definition \ref{defn55}. For $v \in C_0^{\rm int}(\hat R)$, we define
$$
K(v) = K \cap \Sigma(v).
$$

Suppose $v \in \hat R(v')$ with $v' \in C_0^{\rm int}(\hat R')$. If $a$ is large enough, then we may regard $K(v) \subset \Sigma_a(v')$ by $z \mapsto \Phi(z,\xi_a)$, where $\xi_a$ is the source curve of $\frak{forget}(\zeta_a)$. (See Figure \ref{Figure5-5}.) By Definition \ref{defn55} (3), $\tilde u_{a,v}(z) := u_{a,v'}(\Phi(z,\xi_a))$ converges to $u_v$ in $C^2$ topology on $K(v)$. We denote by $\tilde u_{a,v}$ the restriction of $\tilde u_{a,v'}$ to $K(v)$.

\begin{figure}[h]
\centering
\includegraphics[scale=0.3]{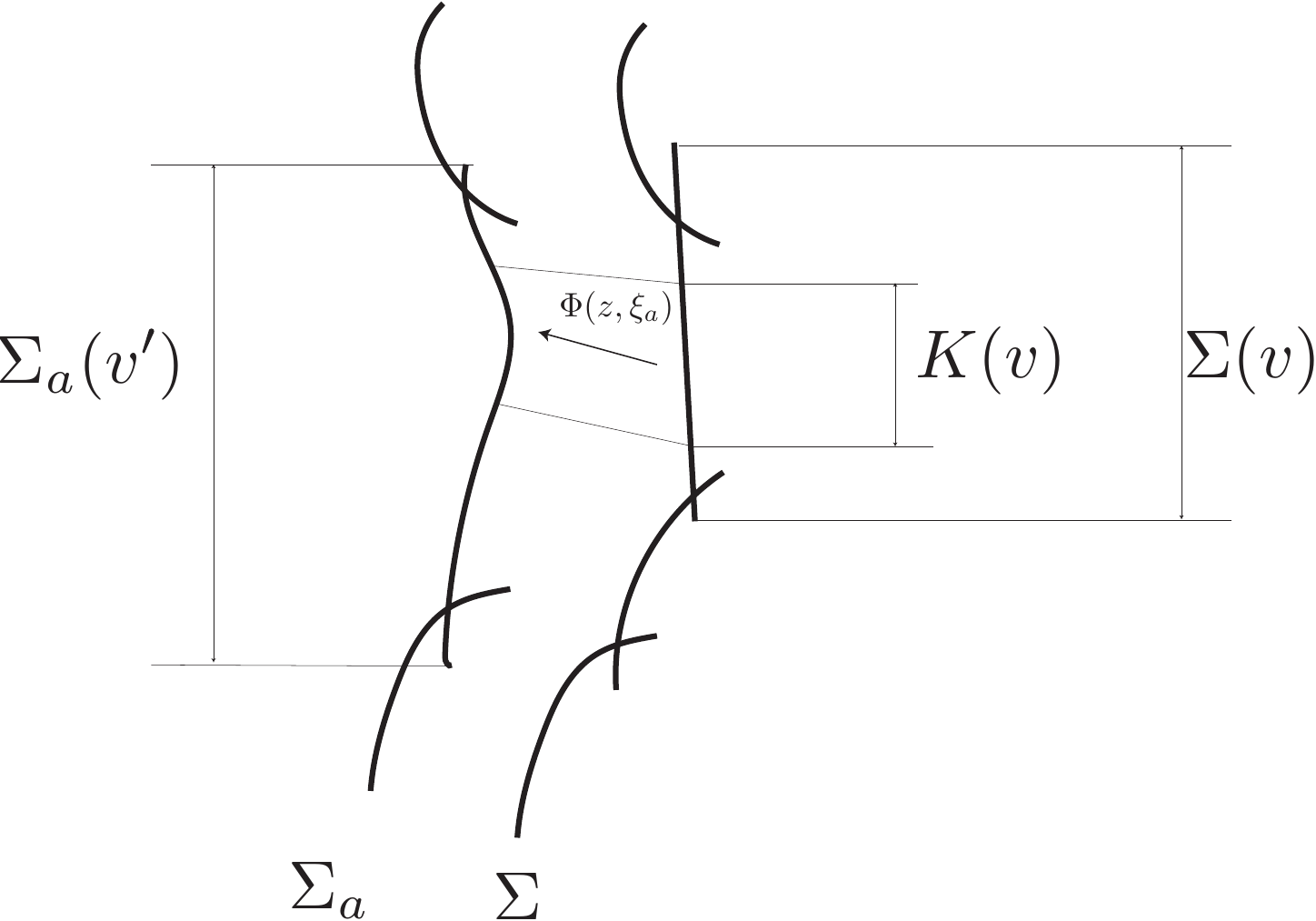}
\caption{$K(v)$ and $\Sigma_a(v')$}
\label{Figure5-5}
\end{figure}

Suppose $c(v)=\rm{D}$ and $c(v')={\rm s}$ or ${\rm d}$. Then for sufficiently large $a$, we may assume that for any $z$ in $K(v)$, we have $\tilde u_{a,v}(z) \in \mathcal N_{\mathcal D}^{\leq c}(X)$ in the same sense as in \eqref{inclusion}. Thus $\tilde u_{a,v}$ may be regarded as a map of the following form
\begin{equation}\label{512-0}
	\tilde u_{a,v} : K(v) \to \mathcal N_{\mathcal D}(X).
\end{equation}
We also regard the section $s_v$ of $u_v^*\mathcal N_{\mathcal D}(X)$ as a map
\begin{equation}\label{512form22}
	U_v : \Sigma_v\setminus \vec w(v) \to \mathcal N_{\mathcal D}(X)\setminus \mathcal D
\end{equation}
\par
If $c(v')=\rm{D}$, then
we use $s_{a,v'}$ and $\tilde u_{a,v}$ to obtain 
\begin{equation}\label{512form22-2}
	\widetilde U_{a,v} : K(v) \to \mathcal N_{\mathcal D}(X).
\end{equation}
Note that in this case $c(v)=\rm{D}$ 
automatically.
\begin{definition}\label{defn523}
	We say that $\zeta_a$ converges to $\zeta$ and write 
	\[
	  \underset{a \to \infty}{\rm lims}\,\zeta_a = \zeta,
	\]
	if for each $j \in \{1,\dots,\vert \lambda\vert\}$, there is a sequence $\rho_{a,j} \in \bbC_{*}$  and for each 
	$\epsilon>0$, there exist an $\epsilon$-trivialization as above and an integer $N(\epsilon)$
	such that the following properties hold for any $v\in C_0^{\rm int}(\hat R)$ and $a\geq N(\epsilon)$:
	\begin{enumerate}
            	\item Suppose $c(v)=\rm{D}$ and $c(v')={\rm s}$ or ${\rm d}$ with ${\rm treesh}(v) = v'$.
            		Then: 
            		\[
            		  d_{C^2}\left({\rm Dil}_{1/\rho_{a,\lambda(v)}}\circ \tilde u_{a,v},U_v\right) < \epsilon.
            		\]
            	\item Suppose $c(v)=\rm{D}$ and $c(v')=\rm{D}$ with ${\rm treesh}(v) = v'$. Then:
            		\[
            		  d_{C^2}\left({\rm Dil}_{1/\rho_{a,\lambda(v)}}\circ \widetilde U_{a,v},U_v\right) < \epsilon.
            		\]
            	\item If $j < j'$ and ${\rm levsh}(j) = {\rm levsh}(j')$, then 
            		\[
            		  \lim_{a \to \infty} \frac{\rho_{a,j}}{\rho_{a,j'}} = \infty.
            		\]
	\end{enumerate}
\end{definition}
This definition is very similar to Definition \ref{defn513}. The only difference is $\zeta_a \in {\mathcal M}^0(\mathcal R')$ and it may have several levels. We need to define convergence for each level. In Definition \ref{defn523}, we use the identification in \eqref{form513} and the product metric on $\mathcal N_{\mathcal D}(X) \setminus \mathcal D$ to define $C^2$ norms in Items (1) and (2).

Analogous to Definition \ref{defn515}, we can extend the definition of convergence to the case that the source curves of $\zeta_a$ and $\zeta$ may not be stable.
Finally we can include the case when $\zeta_a \in {\mathcal M}^0(\mathcal R'_a)$ where 
$\mathcal R'_a$ varies, using the fact that there is only a 
finite number of $\mathcal R'$ with $\mathcal R' \ge \mathcal R$.
This completes the definition of convergence with respect to the RGW topology. If $\zeta_a$ converges to $\zeta$ in this topology, we write:
\begin{equation}\label{form517}
	\lim_{a \to \infty} \zeta_a = \zeta.
\end{equation}

\begin{lemma}\label{prop516rev}
	For any sequence $\zeta_a \in \mathcal M^{\rm RGW}_{k+1,h}(L;\beta)$, there exists a subsequence which converges in the sense of \eqref{form517}.
\end{lemma}

\begin{proof}
This is a consequence of Proposition \ref{prop516} and similar 
results for strips and spheres.
\end{proof}

\begin{definition}
	Let $A \subset \mathcal M^{\rm RGW}_{k+1,h}(L;\beta)$.
	Define the closure of $A$, denoted by $A^c$, to be the set of all the limits of sequences
	of elements of $A$ in the sense of  \eqref{form517}.
\end{definition}

\begin{lemma}\label{Kurarev}
	The closure operator $c$ satisfies the Kuratowsky's axioms. 
	Namely, 
	{\rm (a)} ${\emptyset}^c = \emptyset$,
	{\rm (b)} $A \subseteq A^c$,
	{\rm (c)} $(A^c)^c =  A^c$ and
	{\rm (d)} $(A\cup B)^c = A^c \cup B^c$.
\end{lemma}
\begin{proof}
	(a), (b) and (d) are obvious.
	In order to check (c), let $\zeta_{a,b}, \zeta_a, \zeta \in \mathcal M^{\rm RGW}_{k+1,h}(L;\beta)$
        for $a,b \in \bbZ_+$. We assume
        \begin{equation}\label{form51818}
        \lim_{b\to\infty}\zeta_{a,b} = \zeta_a,
        \qquad
        \lim_{a\to\infty}\zeta_{a} = \zeta.
        \end{equation}
        It suffices to prove that there exists $b(a)$ such that
        $\lim_{a\to \infty} \zeta_{a,b(a)} = \zeta$.

	Using a result similar to Lemma \ref{lem57} (which can be proved in the same way),
	we may assume that $\zeta_{a,b}, \zeta_a, \zeta$ are 
	all source stable and replace $\lim$ by  ${\rm lims}$. Since for each DD-ribbon tree $\mathcal R^+$, there 
	are only finitely many DD-ribbon trees $(\mathcal R')^+$ with $(\mathcal R')^+ \ge \mathcal R^+$, 
	we may also assume that  there are DD-ribbon trees $\mathcal R, \mathcal R', \mathcal R''$ such that
	$\mathcal R'' \ge \mathcal R' \ge \mathcal R$ and 
	\[
	  \zeta_{a,b} \in \mathcal M^0(\mathcal R''),\quad\zeta_{a} \in \mathcal M^0(\mathcal R'),
	  \quad\zeta \in \mathcal M^0(\mathcal R).
	\]
	Moreover, Lemma \ref{lem5421} provides us with ${\rm levsh} : \{0,\dots,\vert \lambda\vert\} \to 
	\{0,\dots,\vert \lambda'\vert\}$,
	${\rm levsh}' : \{0,\dots,\vert \lambda'\vert\} \to \{0,\dots,\vert \lambda''\vert\}$,
	${\rm treesh} : \hat R \to  \hat R'$ and 
	${\rm treesh}' : \hat R' \to  \hat R''$.
	
        \begin{figure}[h]
        \centering
        \includegraphics[scale=0.4]{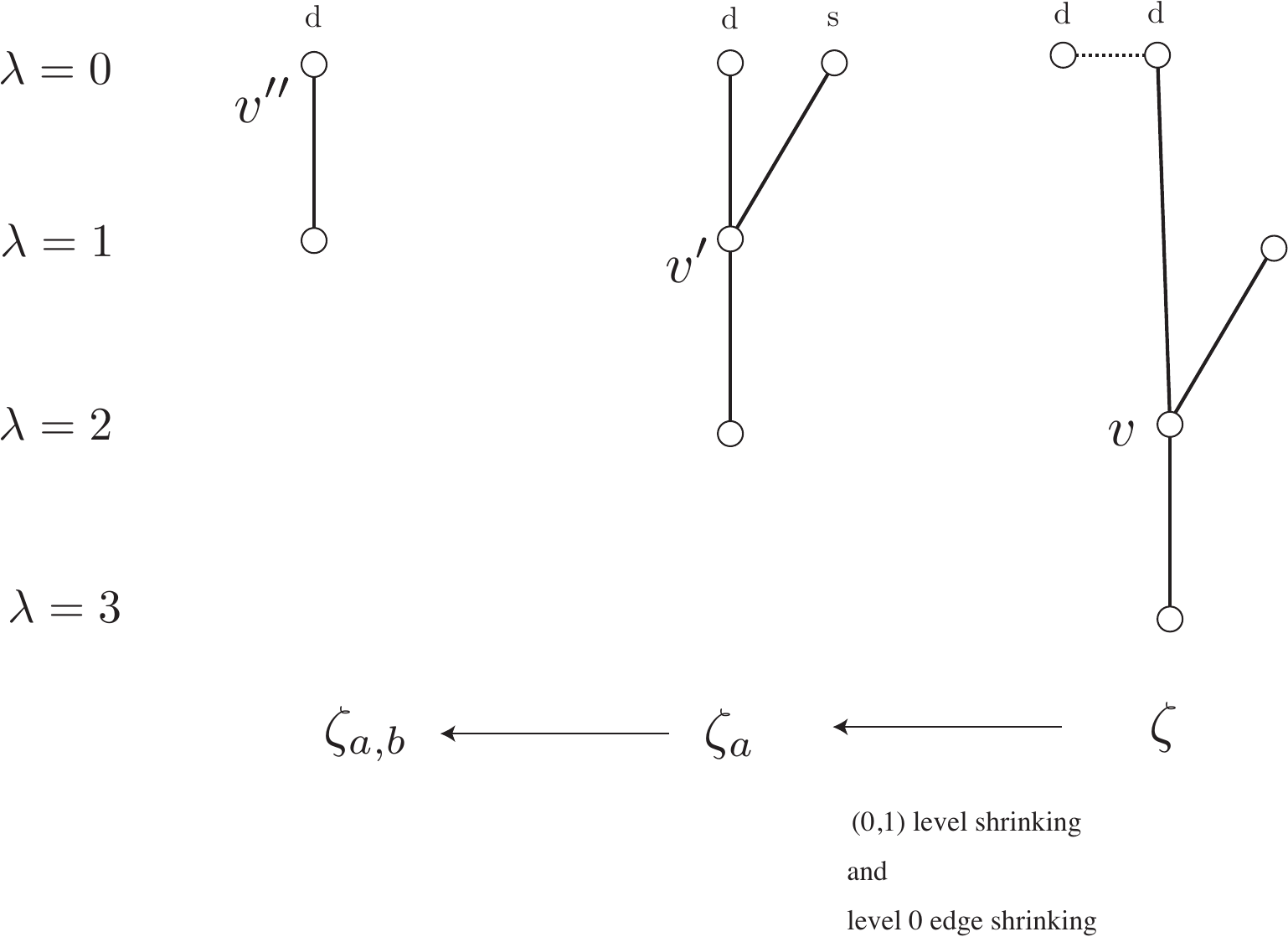}
        \caption{$\zeta_{a,b}, \zeta_{a}, \zeta$ (graph)}
        \label{Figure5-6}
        \end{figure}
        \begin{figure}[h]
        \centering
        \includegraphics[scale=0.4]{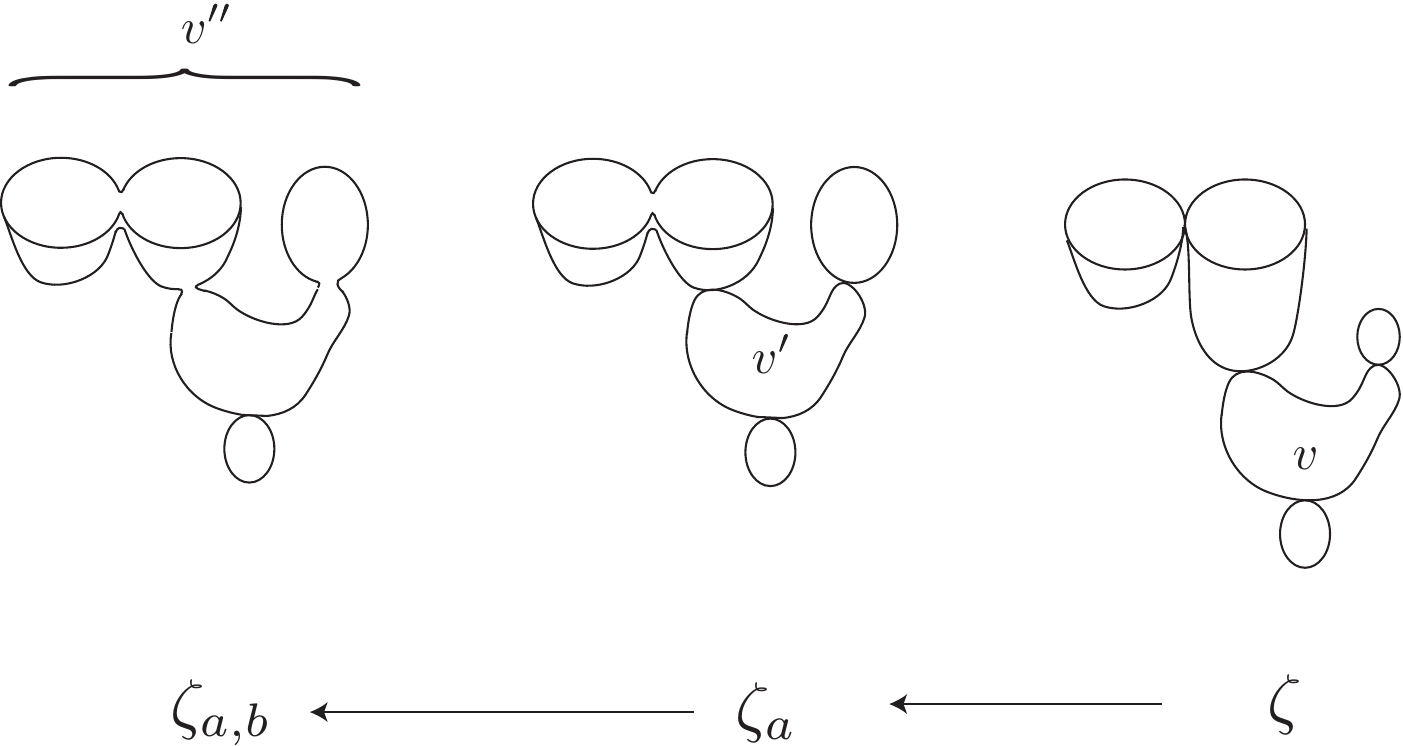}
        \caption{$\zeta_{a,b}, \zeta_{a}, \zeta$}
        \label{Figure5-7}
        \end{figure}

	The assumptions in \eqref{form51818} gives us non-zero complex numbers $\rho_{a,j}$ and $\rho_{ab,j'}$ 
	where $1\leq j \leq \vert \lambda\vert$ and $1\leq j' \leq \vert \lambda'\vert$.
	We denote these numbers with $\rho_{a,\emptyset;j}$ and $\rho_{ab,a;j'}$ to distinguish 
	them from $\rho_{ab,\emptyset;j}$ which shall be introduced momentarily to prove 
	$\lim_{a\to \infty} \zeta_{a,b(a)} = \zeta$. We extend these constants to the case that $j=0$
	and $j'=0$ by setting $\rho_{a,\emptyset;0}=\rho_{ab,a;0}=1$.

	Now we define
	\begin{equation}
		\rho_{ab,\emptyset;j}= \rho_{a,\emptyset;j}\cdot \rho_{ab,a;{\rm levsh}(j)} \in \bbC_{*}.
	\end{equation}
	For each $a$, there is an integer $N(a)$ such that if $b\geq N(a)$, then for any $k$, $k'$ with 
	$0\leq k<k'\leq |\lambda'|$ and ${\rm levsh'}(k)={\rm levsh'}(k')$, we have:
	\[
	  \left | \frac{\rho_{ab,a;k}}{\rho_{ab,a;k'}}\right| \geq 
	  a\cdot \max_{0\leq j'<j\leq |\lambda|}(\left | \frac{\rho_{a,\emptyset;j}}{\rho_{a,\emptyset;j'}}\right |)
	\]
	Thus it is easy to see that if $b(a)\geq N(a)$, then we can conclude for $j < j'$ with 
	${\rm levsh'}\circ {\rm levsh}(j) = {\rm levsh'}\circ{\rm levsh}(j')$ that:
	\[
	  \lim_{a \to \infty} \frac{\rho_{ab(a),\emptyset;j}}{\rho_{ab(a),\emptyset;j'}} = \infty.
	\]
	Next, we show that Definition \ref{defn523} (1) and (2) are satisfied for appropriate choices of $b(a)\geq N(a)$.
		
	Let $v \in C^{\rm int}_0(\hat R)$, $v' = {\rm treesh}(v)$, $v'' = {\rm treesh}'(v')$, $j = \lambda(v)$ and $j'= \lambda'(v')$.
	We also assume that $c(v) = {\rm D}$, $c(v') = {\rm D}$ and $c(v'') = {\rm d}$.
	We also fix an $\epsilon$-trivialization $(K,\mathcal U,\Phi)$ of the universal family in the sense of 
	Definition \ref{defn55} at the source curve $\xi$ of $\zeta$.
	Let $U_v:\Sigma(v) \to \mathcal N_{\mathcal D}(X)$ be defined from the data of $\zeta$ in the same way as in 
	\eqref{512form22}. We define $\widetilde U_{a,v}:K(v) \to \mathcal N_{\mathcal D}(X)$ 
	analogous to \eqref{512form22-2}. 
	Finally, we define $\tilde u_{ab,v}:K(v) \to \mathcal N_{\mathcal D}(X)$ from the data of $\zeta_{ab}$ as in \eqref{512-0}.

	By assumption
	\begin{equation} \label{ineq-1}
	 d_{C^2}\left({\rm Dil}_{1/\rho_{a,\emptyset;j}}\circ \widetilde U_{a,v},U_v\right) < \delta
	\end{equation}
	for sufficiently large $a$ and
	\begin{equation} \label{ineq-2}
	  d_{C^2}\left({\rm Dil}_{1/\rho_{ab,a,j'}}\circ \tilde u_{ab,v},U'_{a,v'}\right) < \frac{1}{a}
	\end{equation}
	if $b$ is a large integer depending on $a$. We denote one such $b$ by $b(a)$, which is also greater than $N(a)$.
	Since ${\rm Dil}_c$ is an isometry, these two inequalities imply
	\[
	  d_{C^2}\left({\rm Dil}_{1/\rho_{ab(a),\emptyset;j}}\circ\tilde u_{ab(a),v},U_v\right) < 2\delta
	\]
	if $a$ is large enough so that \eqref{ineq-1} and \eqref{ineq-2} hold and $\frac{1}{a}<\delta$.
	This verifies Definition \ref{defn523} (1) in the case $c(v) = {\rm D}$, $c(v') = {\rm D}$ and $c(v'') = {\rm d}$.
	The other cases and Definition \ref{defn523} (2)
	can be proved in the same way.
\end{proof}
The above lemma completes the definition of the RGW topology. The following theorem asserts that the RGW topology has the desired properties in the case of moduli spaces of discs. 
\begin{theorem}\label{thm-comp-top}
	The topological space $\mathcal M^{\rm RGW}_{k+1,h}(L;\beta)$ is compact and metrizable.
	Evaluation at any of the boundary marked points (resp. interior marked points) determines a continuous map
	$\mathcal M^{\rm RGW}_{k+1,h}(L;\beta) \to L$ (resp. $\mathcal M^{\rm RGW}_{k+1,h}(L;\beta) \to X$).
	Moreover, for any DD-ribbon tree $\mathcal R^+$, the space 
	\begin{equation}\label{closure}
		{\mathcal M}(\mathcal R^+) := {\mathcal M}^{0}(\mathcal R^+) \cup
		\bigcup_{\mathcal (R')^+ < \mathcal R^+}{\mathcal M}^{0}(\mathcal (R')^+).
	\end{equation}
	is a closed subset of $\mathcal M^{\rm RGW}_{k+1,h}(L;\beta)$.
\end{theorem}
\begin{proof}

	Suppose $ \mathcal R$ is a DD-ribbon tree of type $(\beta,k)$ with detailed ribbon tree $\hat R$. 
	Suppose also ${\rm mk}:\{1,\dots,h\}\to C_0^{\rm int}(\hat R)$ is a 
	marking map for interior points. Then $\mathcal R^+:=(\mathcal R,{\rm mk})$ describes a 
	stratum of $\mathcal M^{\rm RGW}_{k+1,h}(L;\beta)$, and we assume that 
	$\zeta_0$ is a source stable element of this stratum. 
	We wish to construct a countable neighborhood basis for $\zeta_0$.
	
	We fix a $\frac{1}{n}$-trivialization $(K,\mathcal U,\Phi)$ of the universal family at $\zeta_0$.
	We define $B_n(\zeta_0)$ to be the set of the elements 
	$\zeta$ of $\mathcal M^{\rm RGW}_{k+1,h}(L;\beta)$, which satisfy the following properties:
	\begin{itemize}
		\item[(i)] There is a DD-ribbon tree $(\mathcal R')^+$ such that $\zeta \in \mathcal M^0((\mathcal R')^+)$ 
			      and $(\mathcal R')^+\geq \mathcal R^+$. Suppose a representative 
			      $\widetilde{\zeta} = (\zeta(v'))$ is fixed for $\zeta$.
		\item[(ii)] The source curve of $\zeta$ belongs to $\mathcal U$.
		\item[(iii)] The distance between the stable maps $\frak{forget}(\zeta_0)$ and $\frak{forget}(\zeta)$
				with respect to a fixed metric representing the sable map topology is less than $\frac{1}{n}$.
		\item[(iv)] For each $v\in C_0^{\rm int}(\hat R^+))$, we define $K(v)$ in the same way as in \eqref{Kv}.
				Let $\vert\lambda\vert$ be the number of the levels of $\mathcal R^+$. 
				Then for each $0\leq j \leq |\lambda|$, there is a constant $\rho_j$ such that the following conditions are satisfied.
				\begin{enumerate}
					\item[(1)] Let $c(v)={\rm D}$, $c(v')={\rm s}$ or ${\rm d}$ with ${\rm treesh}(v)=v'$.
						We define $\tilde u_v:K(v)\to \mathcal N_{\mathcal D}(X)$ and 
						$U_v:\Sigma(v)\to \mathcal N_{\mathcal D}(X)$ as in \eqref{512-0}
						and \eqref{512form22}. Then the $C^2$-distance of 
						${\rm Dil}_{1/\rho_{\lambda(v)}}\circ \tilde u_{v}$ and $U_v$
						is less than $\frac{1}{n}$.
					\item[(2)] Let $c(v)={\rm D}$, $c(v')={\rm D}$ with ${\rm treesh}(v)=v'$.
						We define $\widetilde U_v:K(v)\to \mathcal N_{\mathcal D}(X)$ and 
						$U_v:\Sigma(v)\to \mathcal N_{\mathcal D}(X)$ as in \eqref{512form22-2}
						and \eqref{512form22}. Then the $C^2$-distance of 
						${\rm Dil}_{1/\rho_{\lambda(v)}}\circ \widetilde U_{v}$ and $U_v$
						is less than $\frac{1}{n}$.
					\item[(3)] $\rho_0=1$ and for $1\leq j \leq |\lambda|$ with 
						${\rm levsh}(j)={\rm levsh}(j-1)$, we have $\rho_{j}> n\cdot \rho_{j-1}$.
				\end{enumerate}			
	\end{itemize}
	It is easy to see that $B_n(\zeta_0)$ is an open set containing $\zeta$. Moreover, any open neighborhood
	of $\zeta_0$ contains $B_n(\zeta_0)$ for large values of $n$. Using the by now familiar trick of forgetting 
	interior marked points, we can extend this construction for any point 
	$\zeta_0\in \mathcal M^{\rm RGW}_{k+1,h}(L;\beta)$. Thus $\mathcal M^{\rm RGW}_{k+1,h}(L;\beta)$
	is a first countable topological space. 
	
	The topology of each stratum of $\mathcal M^{\rm RGW}_{k+1,h}(L;\beta)$ is given by the 
	stable map topology. In particular, it is a separable metric space. Since 
	there are also countably many strata, we can form a sequence $\{\zeta_i\}$ of the elements 
	such that the subsequence of the elements belonging to a given stratum forms a dense subset.
	Then it is easy to see that $\{B_n(\zeta_i)\}$ gives a countable basis for the RGW topology of 
	$\mathcal M^{\rm RGW}_{k+1,h}(L;\beta)$. Since $\mathcal M^{\rm RGW}_{k+1,h}(L;\beta)$ is 
	sequentially compact (Lemma \ref{prop516rev}) and second countable, it is a compact topological space.
	
	Next we show that $\mathcal M^{\rm RGW}_{k+1,h}(L;\beta)$ is Hausdorff. 
	Since the first axiom of countability is satisfied, it suffices to show that any convergent sequence 
	has a unique limit. Let $\zeta_a$ be a sequence which converges to both $\zeta$ and $\zeta'$.
	The stable maps $\frak{forget}(\zeta)$ and $\frak{forget}(\zeta')$ are equal to each other, 
	because the stable map topology is Hausdorff (\cite[Lemma 10.4]{FO}).
	Using a lemma similar to Lemma \ref{lem57}, we may assume that $\zeta_a$, $\zeta$, $\zeta'$ 
	are all source stable. In this case, it is straightforward to see that $\zeta=\zeta'$.
	
	The space $\mathcal M^{\rm RGW}_{k+1,h}(L;\beta)$ is compact and Hausdorff, hence 
	it is a regular space. Therefore, Urysohn's metrization theorem implies that 
	$\mathcal M^{\rm RGW}_{k+1,h}(L;\beta)$ is metrizable.
	
	The claim about continuity of evaluation maps at the marked points follows from the facts that the RGW topology is 
	stronger than the stable map topology and these evaluation maps are continuous with respect to the 
	stable map topology. Finally it is an immediate consequence of the definition that the space in \eqref{closure} is closed. 
\end{proof}

One can prove similar results as in the above theorem for the case of strips and spheres with the same arguments.

\begin{theorem}
	The topological space $\mathcal M^{\rm RGW}_{k_1,k_0,h}(L_1,L_0;p,q;\beta)$ is compact and metrizable.
	Evaluation at any of the boundary or interior marked points determines a continuous map.
	Moreover, for any SD-ribbon tree $\mathcal R^+$, the space 
	\begin{equation}\label{closure}
		{\mathcal M}(\mathcal R^+) := {\mathcal M}^{0}(\mathcal R^+) \cup
		\bigcup_{\mathcal (R')^+ < \mathcal R^+}{\mathcal M}^{0}(\mathcal (R')^+).
	\end{equation}
	is a closed subspace of $\mathcal M^{\rm RGW}_{k_1,k_0,h}(L_1,L_0;p,q;\beta)$.
\end{theorem}

\begin{figure}[h]
\centering
\includegraphics[scale=0.4]{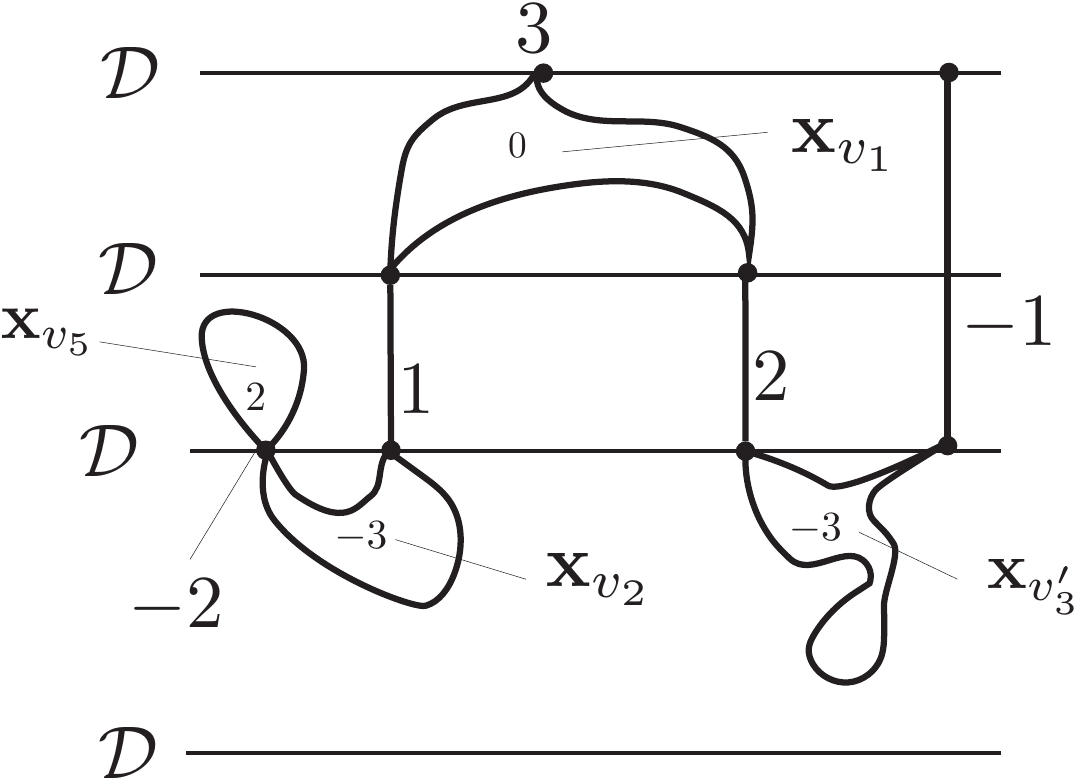}
\caption{Configuration associated to the decorated rooted tree
in Figure \ref{Figure12}}
\label{Figure13}
\end{figure}

\begin{theorem}
	The topological space $\mathcal M(\mathcal D\subset X;\alpha;{\bf m})$ is compact and metrizable.
	Evaluation at any of the (interior) marked points determines a continuous map from the moduli space to $\mathcal D$.
	There also exists a continuous map
        \begin{equation}\label{map321321}
    	    \frak{forget} :\mathcal M(\mathcal D\subset X;\alpha;{\bf m})\to \mathcal M_{\ell+1}(\mathcal D;\alpha).
        \end{equation}
	such that it coincides with \eqref{form3737} on ${\mathcal M}^0(\mathcal T^0_{\alpha;{\bf m}})$ 
	via the identification \eqref{identify30}.
	Moreover, for any decorated rooted tree $\mathcal T$, the space
	\begin{equation}
		\mathcal M(\mathcal D;\mathcal T):=\mathcal M^{0}(\mathcal D;\mathcal T) \cup
		\bigcup_{\mathcal T' < \mathcal T}\mathcal M^{0}(\mathcal D;\mathcal T')
	\end{equation}
	is a closed subspace of $\mathcal M(\mathcal D\subset X;\alpha;{\bf m})$.
\end{theorem}

\begin{proof}
	To prove existence of \eqref{map321321},	let 
	$({\bf x}_v;v\in C_0^{\rm ins}(T)) \in \widetilde{\mathcal M}^0(D;\mathcal T)$ 
	where ${\bf x}_v = [(\Sigma_v,\vec w_v);u_v;s_v]$. We glue $(\Sigma_v,\vec w_v)$ along the tree $T$ in an  
	obvious way to obtain a marked Riemann surface $(\Sigma,\vec w)$.
	Since an element $({\bf x}_v;v\in C_0^{\rm ins}(T))$ 
	lies in the fiber product (\ref{form31666}), various maps
	$u_v$ can be glued to define a continuous map $u : \Sigma \to \mathcal D$.
        We thus obtain an element of $\mathcal M_{\ell+1}(\mathcal D;\alpha)$.
        It is easy to see that this construction induces a map as in (\ref{map321321}).
        The continuity of the map (\ref{map321321}) is immediate from the 
        definition of the RGW topology.
        The remaining claims can be verified in the same way as in Theorem \ref{thm-comp-top}.
\end{proof}

\begin{example}
	Figure \ref{Figure13} sketches an element of $\mathcal M^{0}(D;\mathcal T)$ corresponding to the 
	decorated rooted tree in Figure \ref{Figure12}. This element of the moduli space
	is obtained from Figure \ref{Figure11} by resolving the double point which is 
	the intersection of ${\bf x}_{v_3}$ and ${\bf x}_{v_4}$.
	The new component ${\bf x}_{v'_3}$ in Figure \ref{Figure13} is obtained by this gluing construction.
\end{example}

\begin{remark}\label{thm323}
	Using Lemma \ref{lem313}, the restriction of the map in \eqref{map321321} gives an 
	open embedding from $\mathcal M^{0}(\mathcal D\subset X;\alpha;{\bf m})$ into 
	$\mathcal M^0_{\ell+1}(\mathcal D;\alpha)$. A standard dimension formula from Gromov-Witten theory
	shows that the virtual dimension of $\mathcal M^0_{\ell+1}(\mathcal D;\alpha)$ is equal to 
	\begin{equation}\label{dimensionformula}
		c_1(\mathcal D) \cdot \alpha + 2\dim_{\bbC} X + 2(\ell+1) - 8.
	\end{equation}
	To be more precise, the compact space $\mathcal M_{\ell+1}(\mathcal D;\alpha)$ admits a Kuranihsi structure 
	(without boundary) whose dimension is given in \eqref{dimensionformula} \cite{FO}. 
	In the subsequent paper of this series, we define 
	a Kuranishi structure on  $\mathcal M(\mathcal D\subset X;\alpha;{\bf m})$ such that 
	its restriction to the open subspace $\mathcal M^{0}(\mathcal D\subset X;\alpha;{\bf m})$ 
	agrees with the Kuranishi structure of 
	$\mathcal M_{\ell+1}(\mathcal D;\alpha)$ via the the homeomorphism $\frak{forget}$ in \eqref{map321321}.
	(The relationship between the map $\frak{forget}$ and the Kuranishi structures 
	on the compactifications is more delicate.)
	In particular, the dimension of $\mathcal M(\mathcal D\subset X;\alpha;{\bf m})$ is also given by the formula
	in \eqref{dimensionformula}.
\end{remark}

\bibliography{references}
\bibliographystyle{alpha.bst}
\end{document}